\documentclass{article}%
\usepackage{amsfonts}
\usepackage{amsmath}
\usepackage{amssymb}
\usepackage[showonlyrefs]{mathtools}
\usepackage{graphicx}%
\usepackage{color}
\mathchardef\emptyset="001F
\providecommand{\U}[1]{\protect\rule{.1in}{.1in}}
\newtheorem{theorem}{Theorem}[section]

\newtheorem{lemma}[theorem]{Lemma}

\newtheorem{proposition}[theorem]{Proposition}
\newtheorem{remark}[theorem]{Remark}

\newenvironment{proof}[1][Proof]{\noindent\textbf{#1.} }{\ \rule{0.5em}{0.5em}}
\mathchardef\emptyset="001F
\numberwithin{equation}{section}
\begin{document}

\title{Asymptotic Analysis of Second Order Nonlocal Cahn-Hilliard-Type Functionals}
\author{Gianni Dal Maso\\SISSA, \\Via Bonomea 265, 34136 Trieste, Italy
\and Irene Fonseca\\Department of Mathematical Sciences,\\Carnegie Mellon University,\\Pittsburgh PA 15213-3890, USA
\and Giovanni Leoni\\Department of Mathematical Sciences, \\Carnegie Mellon University, \\Pittsburgh PA 15213-3890, USA}
\maketitle

\begin{abstract}
In this paper the study of a nonlocal second order Cahn--Hilliard-type
singularly perturbed family of functions is undertaken. The kernels considered
include those leading 
to Gagliardo fractional seminorms for gradients. Using
$\Gamma$ convergence the integral representation of the limit energy is
characterized leading to an anisotropic surface energy on interfaces
separating different 
phases.

\bigskip\bigskip\bigskip

\end{abstract}

\section{Introduction}

In the van der 
 Waals--Cahn--Hilliard theory of phase transitions
\cite{cahn-hilliard1958}, \cite{rowlinson1979}, \cite{van-der-walls1893},
\cite{gurtin1985}, the total energy is given by
\begin{equation}
\frac{1}{\varepsilon}\int_{\Omega}W(u(x))\,dx+\varepsilon\int_{\Omega
}\left\vert \nabla u(x)\right\vert ^{2}\,dx,\label{functional modica mortola}%
\end{equation}
where the open bounded set $\Omega\subset\mathbb{R}^{n}$ represents a
container, $u:\Omega\rightarrow\mathbb{R}$ is the fluid density, and
$W:\mathbb{R}\rightarrow\lbrack0,+\infty)$ is a double-well potential
vanishing only at the phases $-1$ and $1$. The perturbation $\varepsilon
\int_{\Omega}\left\vert \nabla u(x)\right\vert ^{2}\,dx$ penalizes rapid
changes of the density $u$, and it plays the role of an interfacial energy.
This problem has been extensively studied in the last four decades (see, e.g.,
\cite{baldo1990}, \cite{barroso-fonseca1994}, \cite{bouchitte1990},
\cite{fonseca-tartar1989}, \cite{modica-mortola1977}, \cite{modica1987},
\cite{owen-sternberg1991}, \cite{owen1988}, \cite{sternberg1988},
\cite{sternberg1991}).

Higher order perturbations were considered in the study of shape deformation
of unilamellar membranes undergoing inplane phase separation (see,
e.g.,\ \cite{kawakatsu-andelman-kawasaki-taniguchi1993},
\cite{taniguchi-kawasaki-andelman-kawakatsu1994},
\cite{leibler-andelman1987,seul-andelman1995}). A simplified local version of
that model (see \cite{seul-andelman1995}) leads to the study of a
Ginzburg-Landau-type energy
\begin{equation}
\frac{1}{\varepsilon}\int_{\Omega}W(u(x))\,dx+q\varepsilon\int_{\Omega
}\left\vert \nabla u(x)\right\vert ^{2}\,dx+\varepsilon^{3}\int_{\Omega
}\left\vert \nabla^{2}u(x)\right\vert ^{2}dx~,\label{functional chermis}%
\end{equation}
where $q\in\mathbb{R}$. This functional is also related to the
Swift--Hohenberg equation (see \cite{swift-hohenberg1977}). When $q=0$, the
functional reduces to the second order version of
(\ref{functional modica mortola}), to be precise,
\begin{equation}
\frac{1}{\varepsilon}\int_{\Omega}W(u(x))\,dx+\varepsilon^{3}\int_{\Omega
}\left\vert \nabla^{2}u(x)\right\vert ^{2}%
dx~,\label{functional fonseca-mantegazza}%
\end{equation}
which was studied in \cite{fonseca-mantegazza2000}. The case $q>0$ in was
treated in \cite{hilhorst-peletier-schatzle2002}, with $|\nabla^{2}u|^{2}$
replaced by $|\Delta u|^{2}$. The case $q<0$ is more delicate and was
considered in \cite{chermisi-dalmaso-fonseca-leoni2011} and \cite{cicalese-spadaro-zeppieri2011}. The original energy functional proposed
in \cite{kawakatsu-andelman-kawasaki-taniguchi1993},
\cite{taniguchi-kawasaki-andelman-kawakatsu1994},
\cite{leibler-andelman1987}, \cite{seul-andelman1995}) involved also a nonlocal
perturbation and was addressed in
\cite{fonseca-hayrapetyan-leoni-zwicknagl2016}.

A nonlocal local version of (\ref{functional modica mortola}) was studied in
\cite{alberti-bellettini-I-1998}, \cite{alberti-bellettini-II-1998},
\cite{alberti-bellettini-cassandro-presutti1996}, with the perturbation
$\varepsilon\int_{\Omega}\left\vert \nabla u(x)\right\vert ^{2}\,dx$ replaced
by a nonlocal term, leading to the energy
\begin{equation}
\frac{1}{\varepsilon}\int_{\Omega}W(u(x))\,dx+\varepsilon\int_{\Omega}%
\int_{\Omega}J_{\varepsilon}(x-y)|u(x)-u(y)|^{2}%
dxdy~,\label{functional alberti bellettini}%
\end{equation}
where
\begin{equation}
J_{\varepsilon}(x):=\frac{1}{\varepsilon^{n}}J\Bigl(\frac{x}{\varepsilon
}\Bigr)\label{J4}%
\end{equation}
and the kernel $J:\mathbb{R}^{n}\rightarrow\lbrack0,+\infty)$ is an even
measurable function such that%
\begin{equation}
\int_{\mathbb{R}^{n}}J(x)(|x|\wedge|x|^{2})~dx=:M_{J}<+\infty~,\label{J1}%
\end{equation}
with $a\wedge b:=\min\{a,b\}$. Functionals of the form
(\ref{functional alberti bellettini}) arise in equilibrium statistical
mechanics as free energies of continuum limits of Ising spin systems on
lattices. In that setting, $u$ is a macroscopic magnetization density and $J$
stands for a ferromagnetic Kac potential (see
\cite{alberti-bellettini-cassandro-presutti1996}). Note that (\ref{J1}) is
satisfied if $J$ is integrable and has compact support. Another important case
is when
\begin{equation}
J(x)=|x|^{-n-2s}\quad\text{with }\frac{1}{2}<s<1~,\label{gagliardo kernel}%
\end{equation}
so that $J_{\varepsilon}(x)=\varepsilon^{2s}|x|^{-n-2s}$, which leads to
Gagliardo's seminorm for the fractional Sobolev space $H^{s}(\mathbb{R}^{n})$
(see \cite{dinezza-palatucci-valdinoci2012}, \cite{gagliardo1957}
\cite{leoni2009}). A functional related to
(\ref{functional alberti bellettini}) with kernel (\ref{gagliardo kernel}) has
been studied in \cite{alberti-bouchitte-seppecher1994},
\cite{alberti-bouchitte-seppecher1998}, and \cite{savin-valdinoci2012} for
$0<s<1$ (see also \cite{garroni-palatucci2006} for an $L^{p}$ version in
dimension $n=1$).

The motivation in \cite{savin-valdinoci2012} was the renewed interest in the
fractional Laplacian (see, e.g., \cite{caffarelli-stinga} and the references
therein), and nonlocal characterizations of fractional Sobolev spaces
(\cite{ambrosio-de-philippis}, \cite{bourgain-brezis-mironescu},
\cite{brezis2015}, \cite{leoni-spector} and the references therein).

Another important application of this type of nonlocal singular perturbation
functionals is in the study of dislocations in elastic materials exhibiting
microstructure (see, e.g., \cite{cacace-garroni}, \cite{conti-garroni-mueller}%
, \cite{garroni-mueller}).

In this paper we consider a nonlocal version of
(\ref{functional fonseca-mantegazza}), to be precise, we study the functional%
\begin{equation}
\mathcal{F}_{\varepsilon}(u):=\frac{1}{\varepsilon}\int_{\Omega}%
W(u(x))\,dx+\varepsilon\int_{\Omega}\int_{\Omega}J_{\varepsilon}(x-y)|\nabla
u(x)-\nabla u(y)|^{2}dxdy \label{functional ours}%
\end{equation}
for $u\in W_{\operatorname*{loc}}^{1,2}(\Omega)$, where $\Omega\subset
\mathbb{R}^{n}$, $n\geq2$, is a bounded open set with Lipschitz boundary, the
double-well potential $W:\mathbb{R}\rightarrow\lbrack0,+\infty)$ is a
continuous function with $W^{-1}(\{0\})=\{-1,+1\}$ satisfying appropriate
coercivity and growth conditions, and $J_{\varepsilon}$ is given by
(\ref{J4}). We assume a non-degeneracy hypothesis (see (\ref{J2})) on the even
measurable kernel $J:\mathbb{R}^{n}\rightarrow\lbrack0,+\infty)$, and that
(\ref{J1}) holds.

We establish compactness in $L^{2}(\Omega)$ for energy bounded sequences, and
in order to study the asymptotic behavior of (\ref{functional ours}) as
$\varepsilon\rightarrow0^{+}$, we use the notion of $\Gamma$-convergence (see
\cite{dalmaso1993}) with respect to the metric in $L^{2}(\Omega)$ and we
identify the $\Gamma$-limit of $\mathcal{F}_{\varepsilon}$. As it is usual, we
extend $\mathcal{F}_{\varepsilon}(u)$ to be $+\infty$ for $u\in L^{2}%
(\Omega)\setminus W_{\operatorname*{loc}}^{1,2}(\Omega)$. Our first main
result is the following theorem.

\begin{theorem}
[Compactness]\label{theorem compactness n>1}Assume that $W$ and $J$ satisfy
(\ref{W1})--(\ref{W3b}) and (\ref{J1}), (\ref{J2}), respectively. Let
$\{u_{\varepsilon}\}\subset W_{\operatorname*{loc}}^{1,2}(\Omega)\cap
L^{2}\left(  \Omega\right)  $ be such that%
\begin{equation}
M:=\sup_{\varepsilon}\mathcal{F}_{\varepsilon}(u_{\varepsilon})<+\infty~.
\label{D7}%
\end{equation}
Then there exists a sequence $\varepsilon_{j}\rightarrow0^{+}$ such that
$\{u_{\varepsilon_{j}}\}$ converges in $L^{2}(\Omega)$ to some function $u\in
BV(\Omega;\{-1,1\})$.
\end{theorem}

The proof of this theorem is more involved than the corresponding one in
\cite{alberti-bellettini-II-1998} due to the presence of gradients in the
nonlocal term. This prevents us from using standard arguments in which
discontinuities in $u$ may be allowed. We first prove compactness in $n=1$,
and then use a slicing technique to treat the higher dimensional case.

To state the $\Gamma$ convergence result, we need to introduce some notation.
Given $n\geq2$ and $\nu\in\mathbb{S}^{n-1}:=\partial B_{1}(0)$, let $\nu_{1}$,
\ldots, $\nu_{n}$ be an orthonormal basis in $\mathbb{R}^{n}$ with $\nu
_{n}=\nu$. Here, and in what follows, we denote by $B_{r}(x)$ the open ball in
$\mathbb{R}^{n}$ centered at $x$ and with radius $r$. Let
\begin{align}
V^{\nu}  &  :=\{x\in\mathbb{R}^{n}:~|x\cdot\nu_{i}|<1/2\text{ for }%
i=1,\ldots,n-1\}~,\label{V ni}\\
Q^{\nu}  &  :=\{x\in\mathbb{R}^{n}:~|x\cdot\nu_{i}|<1/2
\text{ for }i=1,\ldots,n\}~,
\label{Q ni}%
\end{align}
 let $W_{\nu_{1},\ldots, \nu_{n-1}}^{1,2}$
be the set of all functions 
$v\in W_{\operatorname*{loc}}^{1,2}(\mathbb{R}^{n})$
such that $v(x+\nu_{i})=v(x)$ for a.e. $x\in\mathbb{R}^{n}$
and for every $i=1,\ldots,n-1$, and let
\begin{equation}
X^{\nu}:=\{  v\in  W_{\nu_{1},\ldots, \nu_{n-1}}^{1,2}:~v(x)=\pm1 
\text{ for a.e. }x\in\mathbb{R}^{n}\text{ with }\pm x\cdot\nu
\geq1/2\}\label{X ni}
\end{equation}
When $n=1$ take $\nu=\pm1$, $V^{\nu}:=\mathbb{R}$, $Q^{\nu}:=(-1/2,1/2)$, and
let $X^{\nu}$ be the space of all functions $v\in W_{\operatorname*{loc}%
}^{1,2}(\mathbb{R})$ such that $v(x)=\pm1$ for a.e. $x\in\mathbb{R}$ with $\pm
x\geq1/2$. We define the anisotropic surface energy density
\begin{equation}
\psi(\nu):=\inf_{0<\varepsilon<1}\inf_{v\in X^{\nu}}\mathcal{F}_{\varepsilon}^{\nu}(v)\,,\label{psi}\end{equation}
where
\[
\mathcal{F}_{\varepsilon}^{\nu}(u):= {}\frac{1}{\varepsilon}\int_{Q^{\nu}}W(u(x))\,dx
+\varepsilon\int_{V^{\nu}}\int_{\mathbb{R}^{n}}J_{\varepsilon}(x-y)|\nabla
u(x)-\nabla u(y)|^{2}dxdy~.
\]
Finally, we define $\mathcal{F}:L^{2}(\Omega)\rightarrow\lbrack0,+\infty]$ by%
\begin{equation}
\mathcal{F}(u):=\left\{
\begin{array}
[c]{ll}%
\displaystyle
\int_{S_{u}}\psi(\nu_{u})~d\mathcal{H}^{n-1} & \text{if }u\in BV(\Omega
;\{-1,1\})~,\\
+\infty & \text{otherwise in }L^{2}(\Omega)~,
\end{array}
\right.  \label{gamma limit}%
\end{equation}
where $S_{u}$ is the jump set of $u$, $\nu_{u}$ is the approximate normal to
$S_{u}$, and $\mathcal{H}^{n-1}$ is the $(n-1)$-dimensional Hausdorff measure
(see \cite{ambrosio-fusco-pallara2000} for a detailed description of these notions).

\begin{theorem}
[$\Gamma$-Limit]\label{theorem lim}Assume that $W$ and $J$ satisfy
(\ref{J2})--(\ref{W3b}) and (\ref{J1}), respectively. Then for every
$\varepsilon_{j}\rightarrow0^{+}$ the sequence $\{\mathcal{F}_{\varepsilon
_{j}}\}$ $\Gamma$-converges to $\mathcal{F}$ in $L^{2}(\Omega)$.
\end{theorem}

Although the general structure of the proof is standard, there are remarkable
technical difficulties due to the nonlocality of the perturbation and the
presence of gradients.

This paper is organized as follows. After a brief section on preliminaries, on
Section \ref{section compactness one} in order to establish compactness in
dimension $n=1$, we prove an interpolation result, which allows us to control the
$L^{2}$ norm of $u^{\prime}$ in terms of the full energy (see Lemma
\ref{F33}). Section \ref{section compactness n 2} is devoted to compactness in
higher dimensions, and here again we obtain the equivalent to the
interpolation Lemma \ref{F33} (see Lemma \ref{F53}). As it is classical in
this type of problems, it is important to be able to modify admissible
sequences near the boundary of their domain without increasing the limit
energy. We address this in Theorem \ref{modification} in Section
\ref{section modification}. Section \ref{section gammaliminf} concerns the
$\Gamma$-liminf inequality, and in Section \ref{section gamma limsup} we
construct the recovery sequence for the $\Gamma$-limsup inequality.

\section{Preliminaries}

\label{section preliminary}

In what follows, in addition to (\ref{J1}) we also assume that the kernel
$J:\mathbb{R}^{n}\rightarrow\lbrack0,+\infty)$ has the following property:
there exist $\gamma_{J}>0$, $\delta_{\!J}\in(0,1)$, $c_{\!J}>0$, such that for
all $\xi\in\mathbb{S}^{n-1}$ there are $\alpha(\xi)<\beta(\xi)$ satisfying%
\begin{equation}
-\gamma_{J}\leq\alpha(\xi)\leq\alpha(\xi)+\delta_{\!J}\leq\beta(\xi)\leq
\gamma_{J} \label{J1a}%
\end{equation}
and%
\begin{equation}
\int_{\alpha(\xi)}^{\beta(\xi)}\frac{1}{J(t\xi)|t|^{n-1}}~dt\leq c_{\!J}~.
\label{J2}%
\end{equation}

\begin{remark}
For example, condition (\ref{J2}) holds if there exist $0<r<R$ and $a>0$ such
that $J(x)\geq a$ for every $x\in\mathbb{R}^{n}$ with $r<|x|<R$. Indeed, it is
enough to set $\gamma_{J}=R$, $\delta_{\!J}=R-r$, $\alpha(\xi)=r$, $\beta
(\xi)=R$, and $c_{\!J}=(na)^{-1}(r^{-n}-R^{-n})$.
\end{remark}

We assume that the double-well potential is a continuous function
$W:\mathbb{R}\rightarrow\lbrack0,+\infty)$ such that
\begin{align}
&  W^{-1}(\{0\})=\{-1,1\}~,\label{W1}\\
&  (|s|-1)^{2}\leq c_{W}W(s)\quad\text{for all }s\in\mathbb{R}~,\label{W2}\\
&  W\text{ is increasing on }[1,+\infty)\text{ and on }[-1,-1+a_{W}%
]~,\label{W3}\\
&  W\text{ is decreasing on }(-\infty,-1]\text{ and on }[1-a_{W},1]~,
\label{W3b}%
\end{align}
for some constants $c_{W}>0$ and $a_{W}\in(0,1)$.

If $s\leq0$ and $|s+1|\geq\frac{1}{2}$, then $|s-1|=|s|-1+2$, hence
$(s-1)^{2}\leq2(|s|-1)^{2}+4\leq2c_{W}W(s)+\frac{4}{m_{W}}W(s)$, where
\begin{equation}
m_{W}:=\min_{\{||s|-1|\geq\frac{1}{2}\}}W(s)>0~. \label{F30}%
\end{equation}
Together with (\ref{W2}) this leads to the estimate
\begin{equation}
(s-1)^{2}\leq\hat{c}_{W}W(s)\quad\text{for all }s\in\mathbb{R}\text{ with
}|s+1|\geq\frac{1}{2}~, \label{F31}%
\end{equation}
where $\hat{c}_{W}:=2c_{W}+\frac{4}{m_{W}}$. Similarly, it can be shown that
\begin{equation}
(s+1)^{2}\leq\hat{c}_{W}W(s)\quad\text{for all }s\in\mathbb{R}\text{ with
}|s-1|\geq\frac{1}{2}~. \label{F41}%
\end{equation}

We recall that $\Omega\subset\mathbb{R}^{n}$ is a bounded open set with
Lipschitz boundary. For every $\varepsilon>0$ and $u\in L^{2}\left(
\Omega\right)  $ consider the functional
\begin{equation}
\mathcal{F}_{\varepsilon}(u):=\left\{
\begin{array}
[c]{ll}%
\mathcal{W}_{\varepsilon}(u)+\mathcal{J}_{\varepsilon}(u) & \text{if }u\in
W_{\operatorname*{loc}}^{1,2}(\Omega)\cap L^{2}\left(  \Omega\right)  ~,\\
+\infty & \text{otherwise,}%
\end{array}
\right.  \label{F}%
\end{equation}
where
\begin{equation}
\mathcal{W}_{\varepsilon}(u):=\frac{1}{\varepsilon}\int_{\Omega}%
W(u(x))~dx\quad\text{for }u\in L^{2}\left(  \Omega\right)  ~,
\label{W functional}%
\end{equation}
and
\begin{equation}
\mathcal{J}_{\varepsilon}(u):=\varepsilon\int_{\Omega}\int_{\Omega
}J_{\varepsilon}(x-y)|\nabla u(x)-\nabla u(y)|^{2}dxdy\quad\text{for }u\in
W_{\operatorname*{loc}}^{1,2}(\Omega). \label{J functional}%
\end{equation}

In the sequel, we will use a localized version of (\ref{F}). To be precise,
given two open sets $A$, $B\subset\mathbb{R}^{n}$ we define
\begin{equation}
\mathcal{W}_{\varepsilon}(u,A):=\frac{1}{\varepsilon}\int_{A}W(u(x))~dx
\label{W local}%
\end{equation}
for $u\in L^{2}(A)$, and%
\begin{equation}
\mathcal{J}_{\varepsilon}(u,A,B):=\varepsilon\int_{A}\int_{B}J_{\varepsilon
}(x-y)|\nabla u(x)-\nabla u(y)|^{2}dxdy \label{J and W local}%
\end{equation}
for $u\in W_{\operatorname*{loc}}^{1,2}(A\cup B)$. When $A=B$ we set
\begin{equation}
\mathcal{F}_{\varepsilon}(u,A):=\mathcal{W}_{\varepsilon}(u,A)+\mathcal{J}%
_{\varepsilon}(u,A,A)\quad\text{and}\quad\mathcal{J}_{\varepsilon
}(u,A):=\mathcal{J}_{\varepsilon}(u,A,A)
\label{214bis}
\end{equation}for $u\in W_{\operatorname*{loc}}^{1,2}(A)\cap L^{2}(A)$.

Since $J$ is even, by Fubini's theorem for all $u\in W_{\operatorname*{loc}%
}^{1,2}(A\cup B)$ we have that
\begin{equation}
\mathcal{J}_{\varepsilon}(u,A,B)=\mathcal{J}_{\varepsilon}(u,B,A)~.
\label{J symm}%
\end{equation}
Moreover, if $A\cap B=\emptyset$ we have
\begin{equation}
\mathcal{J}_{\varepsilon}(u,A\cup B)=\mathcal{J}_{\varepsilon}%
(u,A)+2\mathcal{J}_{\varepsilon}(u,A,B)+\mathcal{J}_{\varepsilon}(u,B)~.
\label{J add}%
\end{equation}

In the compactness theorem we use a slicing argument based on the following
preliminary result. Given a vector $\xi\in\mathbb{S}^{n-1}$, the hyperplane
through the origin orthogonal to $\xi$ is denoted by $\Pi^{\xi}$, that is,
\begin{equation}
\Pi^{\xi}:=\{x\in\mathbb{R}^{n}:~x\cdot\xi=0\}~. \label{PI xi}%
\end{equation}
If $E\subset\mathbb{R}^{n}$ and $y\in\Pi^{\xi}$, then we define
\begin{equation}
E_{y}^{\xi}:=\{t\in\mathbb{R}:~y+t\xi\in E\}~. \label{E slice}%
\end{equation}

The next result is a particular case of the affine Blaschke--Petkantschin
formula, for which we refer to \cite[Theorem 7.2.7]{Sch}.

\begin{proposition}
\label{proposition blaschke}Let $E\subset\mathbb{R}^{n}$ be a Borel set and
let $g:E\times E\rightarrow\lbrack0,+\infty]$ be a Borel function. Then
\begin{align*}
\int_{E}  &  \int_{E}g(x,y)~dxdy\\
&  =\frac{1}{2}\int_{\mathbb{S}^{n-1}}\int_{\Pi^{\xi}}\int_{E_{z}^{\xi}}%
\int_{E_{z}^{\xi}}g(z+s\xi,z+t\xi)|t-s|^{n-1}dsdtd\mathcal{H}^{n-1}%
(z)d\mathcal{H}^{n-1}(\xi)~.
\end{align*}

\end{proposition}

\begin{proof}
For the convenience of the reader we present a proof. We extend $g$ to be zero
outside $E\times E$. Using the change of variables $\tau=t-s$, we obtain%
\[
\int_{\mathbb{R}}g(z+s\xi,z+t\xi)|t-s|^{n-1}ds=\int_{\mathbb{R}}g(z+t\xi
-\tau\xi,z+t\xi)|\tau|^{n-1}d\tau~,
\]
and by Fubini's theorem we get%
\begin{align*}
\int_{\Pi^{\xi}}\int_{\mathbb{R}}\int_{\mathbb{R}}  &  g(z+s\xi,z+t\xi
)|t-s|^{n-1}dsdtd\mathcal{H}^{n-1}(z)\\
&  =\int_{\mathbb{R}^{n}}\int_{\mathbb{R}}g(y-\tau\xi,y)|\tau|^{n-1}d\tau dy~.
\end{align*}
Exchanging the order of integration and using integration in spherical
coordinates we have%
\begin{align*}
\frac{1}{2}\int_{\mathbb{S}^{n-1}}\int_{\Pi^{\xi}}\int_{\mathbb{R}}%
\int_{\mathbb{R}}  &  g(z+s\xi,z+t\xi)|t-s|^{n-1}dsdtd\mathcal{H}%
^{n-1}(z)d\mathcal{H}^{n-1}(\xi)\\
&  =\frac{1}{2}\int_{\mathbb{R}^{n}}\int_{\mathbb{S}^{n-1}}\int_{\mathbb{R}%
}g(y-\tau\xi,y)|\tau|^{n-1}d\tau d\mathcal{H}^{n-1}(\xi)dy\\
&  =\int_{\mathbb{R}^{n}}\int_{\mathbb{R}^{n}}g(x,y)~dxdy~,
\end{align*}
which concludes the proof.
\end{proof}

For $\xi\in\mathbb{S}^{n-1}$ and $\varepsilon>0$ define $J^{\xi}%
:\mathbb{R}\rightarrow\lbrack0,+\infty)$ by
\begin{equation}
J^{\xi}(t):=J(t\xi)|t|^{n-1}\quad\text{and}\quad J_{\varepsilon}^{\xi
}(t):=\frac{1}{\varepsilon}J^{\xi}\left(  \frac{t}{\varepsilon}\right)  ~.
\label{p1}%
\end{equation}

By (\ref{J1}) and using spherical coordinates, we have%
\begin{equation}
\int_{\mathbb{R}}J^{\xi}(t)(|t|\wedge|t|^{2})~dt<+\infty\label{p2}%
\end{equation}
for $\mathcal{H}^{n-1}$-a.e. $\xi\in\mathbb{S}^{n-1}$, and in view of
(\ref{J2}) we obtain
\begin{equation}
\int_{\alpha(\xi)}^{\beta(\xi)}\frac{1}{J^{\xi}(t)}~dt\leq c_{\!J}~.
\label{J2bis}%
\end{equation}
Moreover,
\begin{equation}
J_{\varepsilon}^{\xi}(t)=\frac{1}{\varepsilon}J^{\xi}\left(  \frac
{t}{\varepsilon}\right)  =\frac{1}{\varepsilon}J\left(  \frac{t\xi
}{\varepsilon}\right)  \left\vert \frac{t}{\varepsilon}\right\vert
^{n-1}=J_{\varepsilon}(t\xi)|t|^{n-1}~. \label{p3}%
\end{equation}

For $\xi\in\mathbb{S}^{n-1}$, $A\subset\mathbb{R}$, and $\varepsilon>0$, we
define
\begin{equation}
\mathcal{F}_{\varepsilon}^{\xi}(v,A):=\frac{1}{\sigma_{n-1}\varepsilon}%
\int_{A}W(v(t))~dt+\frac{\varepsilon}{2}\int_{A}\int_{A}J_{\varepsilon}^{\xi
}(s-t)(v^{\prime}(s)-v^{\prime}(t))^{2}dsdt \label{p5}%
\end{equation}
for $v\in W_{\operatorname*{loc}}^{1,2}(A)\cap L^{2}\left(  A\right)  $, where $\sigma_{n-1}:=
\mathcal{H}^{n-1}(\mathbb{S}^{n-1})$.

\section{Compactness and interpolation in dimension one}

\label{section compactness one}For a set $A$ contained in $\mathbb{R}^{n}$ and
for $\eta>0$ we define
\begin{equation}
\begin{array}[c]{ll}
(A)^{\eta}:=\{x\in\mathbb{R}^{n}:\operatorname*{dist}(x,A)<\eta\}~,
\\
(A)_{\eta}:=\{x\in\ A:~\operatorname*{dist}(x,\partial A)>\eta\}~. 
\end{array}
\label{F49}%
\end{equation}
The main result of this section is the following theorem.

\begin{theorem}
\label{theorem compactness}Let $\xi\in\mathbb{S}^{n-1}$, let $A\subset
\mathbb{R}$ be 
a bounded open set, and let $\{u_{\varepsilon}\}\subset
W_{\operatorname*{loc}}^{1,2}(A)\cap L^{2}\left(  A\right)  $ be such that%
\begin{equation}
M:=\sup_{\varepsilon}\mathcal{F}_{\varepsilon}^{\xi}(u_{\varepsilon
},A)<+\infty~, \label{c1}%
\end{equation}
where $\mathcal{F}_{\varepsilon}^{\xi}$ is defined in (\ref{p5}). Then there
exists a sequence $\varepsilon_{j}\rightarrow0^{+}$ such that
$\{u_{\varepsilon_{j}}\}$ converges in $L^{2}(A)$ to some function $u\in
BV(A;\{-1,1\})$. Moreover, there exists a constant $c_{J,W}>0$, independent of
$\xi$, $A$, and $\{u_{\varepsilon}\}$, such that
\begin{equation}
\#S_{u}\leq\frac{M}{c_{J,W}}~,\label{D10}%
\end{equation}
where $\#S_{u}$ denotes the number of jump points of $u$.
\end{theorem}

Next we introduce some auxiliary lemmas that will be used in the proof of
Theorem \ref{theorem compactness}.

\begin{lemma}
\label{lemma 1}Let $\xi\in\mathbb{S}^{n-1}$, let $A\subset\mathbb{R}$ be an
open set, let $\varepsilon>0$, let $\alpha<\beta$, and let $u\in
W_{\operatorname*{loc}}^{1,2}((A)^{\varepsilon\gamma_{J}})$, where $\gamma
_{J}$ is the constant in (\ref{J1a}). Then for a.e. $t\in A$,%
\begin{align}
\varepsilon &  \int_{t-\varepsilon\beta}^{t-\varepsilon\alpha}J_{\varepsilon
}^{\xi}(t-s)(u^{\prime}(t)-u^{\prime}(s))^{2}ds\nonumber\\
&  \geq\varepsilon(\beta-\alpha)^{2}\left(  \int_{\alpha}^{\beta}\frac
{1}{J^{\xi}(z)}~dz\right)  ^{-1}\left(  u^{\prime}(t)-\frac{u(t-\varepsilon
\alpha)-u(t-\varepsilon\beta)}{\varepsilon(\beta-\alpha)}\right)  ^{2},
\label{F32}%
\end{align}
where $J^{\xi}$ and $J_{\varepsilon}^{\xi}$ are defined in (\ref{p1}).
\end{lemma}

\begin{proof}
It is enough to show that for every $\lambda\in\mathbb{R}$ we have%
\begin{align}
\varepsilon &  \int_{t-\varepsilon\beta}^{t-\varepsilon\alpha}J_{\varepsilon
}^{\xi}(t-s)(\lambda-u^{\prime}(s))^{2}ds\nonumber\\
&  \geq\varepsilon(\beta-\alpha)^{2}\left(  \int_{\alpha}^{\beta}\frac
{1}{J^{\xi}(z)}~dz\right)  ^{-1}\left(  \lambda-\frac{u(t-\varepsilon
\alpha)-u(t-\varepsilon\beta)}{\varepsilon(\beta-\alpha)}\right)
^{2}~.\nonumber
\end{align}
This inequality follows by considering the Euler--Lagrange equation of the
minimum problem%
\[
\min\int_{t-\varepsilon\beta}^{t-\varepsilon\alpha}J_{\varepsilon}^{\xi
}(t-s)(\lambda-v^{\prime}(s))^{2}ds
\]
over all $v\in W^{1,2}((t-\varepsilon\beta,t-\varepsilon\alpha))$ satisfying
$v(t-\varepsilon\beta)=u(t-\varepsilon\beta)$ and $v(t-\varepsilon
\alpha)=u(t-\varepsilon\alpha)$.
\end{proof}

\begin{remark}
\label{remark 1}Under the same assumptions of Lemma \ref{lemma 1}, it follows
from (\ref{J1a}), (\ref{J2}), and (\ref{F32}) that
\begin{align*}
\varepsilon(u^{\prime}(t))^{2}\leq{}  &  \frac{2}{\delta_{\!J}^{2}}\frac
{1}{\varepsilon}\big(u(t-\varepsilon\alpha(\xi))-u(t-\varepsilon\beta
(\xi))\big)^{2}\\
&  +2c_{\!J}\varepsilon\int_{t-\varepsilon\gamma_{J}%
}^{t+\varepsilon\gamma_{J}}J_{\varepsilon}^{\xi}(t-s)(u^{\prime}(t)-u^{\prime
}(s))^{2}ds
\end{align*}
for a.e. $t\in A$.
\end{remark}

\begin{lemma}
\label{lemma compactness} Let $\gamma_{J}$ be the constant in (\ref{J1a}).
Then there exists a constant $c_{J,W}>0$ such that
\begin{equation}
\varepsilon\int_{\sigma}^{\tau}\int_{\sigma-\varepsilon\gamma_{J}}%
^{\tau+\varepsilon\gamma_{J}}J_{\varepsilon}^{\xi}(t-s)(u^{\prime
}(t)-u^{\prime}(s))^{2}dsdt+\frac{1}{\varepsilon}\int_{\sigma-\varepsilon
\gamma_{J}}^{\tau+\varepsilon\gamma_{J}}W(u(t))~dt\geq c_{J,W} \label{c2aa}%
\end{equation}
for every $\xi\in\mathbb{S}^{n-1}$, for every $\varepsilon>0$, for every
$\sigma$, $\tau$, with $\sigma<\tau$, and for every $u\in
W_{\operatorname*{loc}}^{1,2}((\sigma-\varepsilon\gamma_{J},\tau
+\varepsilon\gamma_{J}))$
such that
\begin{equation}
u(t)\in\left(  -\tfrac{1}{2},\tfrac{1}{2}\right)  \text{ for every }%
t\in(\sigma,\tau)~, \label{c2}%
\end{equation}
and either%
\begin{equation}
u(\sigma)=-\tfrac{1}{2}\quad\text{and}\quad u(\tau)=\tfrac{1}{2} \label{c2ab}%
\end{equation}
or
\begin{equation}
u(\sigma)=\tfrac{1}{2}\quad\text{and}\quad u(\tau)=-\tfrac{1}{2}~.
\label{c2ab-bis}%
\end{equation}

\end{lemma}

\begin{proof}
Fix $\xi$, $\varepsilon$, $\sigma$, $\tau$, and $u$ as in the statement of the
lemma, and let $\hat{\alpha}$ and $\hat{\beta}$ be such that $\alpha(\xi
)<\hat{\alpha}<\hat{\beta}<\beta(\xi)$, and
\begin{equation}
\alpha(\xi)-\hat{\alpha}>\frac{1}{4}\delta_{\!J},\quad\hat{\beta}-\hat{\alpha
}>\frac{1}{4}\delta_{\!J},\quad\beta(\xi)-\hat{\beta}>\frac{1}{4}\delta
_{\!J}~, \label{F51}%
\end{equation}
where $\delta_{\!J}$ is the constant in (\ref{J1a}). By (\ref{W2}) and
(\ref{c2}), we have $W(u(t))\geq\frac{1}{4C_{W}}$ for every $t\in(\sigma
,\tau)$. Therefore, if $\tau-\sigma>\varepsilon\delta_{\!J}/2^{6}$,
then
\begin{equation}
\frac{1}{\varepsilon}\int_{\sigma}^{\tau}W(u_{\varepsilon}(t))~dt>\frac
{\delta_{\!J}}{2^{8}C_{W}}~. \label{c3}%
\end{equation}
If $\tau-\sigma\leq\varepsilon\delta_{\!J}/2^{6}$, define%
\begin{equation}
A_0:=\left\{  t\in(\sigma,\tau):~|u^{\prime}(t)|\geq\frac{1}{2}\frac{1}%
{\tau-\sigma}\right\}  ~. \label{c4}%
\end{equation}
We consider now two cases.

\noindent\textbf{Case 1: }Assume that for every $t\in A_0$ there exist
$\alpha\in\lbrack\alpha(\xi),\hat{\alpha}]$ and $\beta\in\lbrack\hat{\beta
},\beta(\xi)]$ such that%
\[
\frac{|u(t-\varepsilon\alpha)-u(t-\varepsilon\beta)|}{\varepsilon(\beta
-\alpha)}<\frac{1}{2}|u^{\prime}(t)|~.
\]
Then
\[
\left(  u^{\prime}(t)-\frac{u(t-\varepsilon\alpha)-u(t-\varepsilon\beta
)}{\varepsilon(\beta-\alpha)}\right)  ^{2}\geq\frac{1}{4}(u^{\prime}%
(t))^{2}~.
\]
Therefore, by Lemma \ref{lemma 1},
\begin{align*}
\varepsilon\int_{t-\varepsilon\beta}^{t-\varepsilon\alpha}  &  J_{\varepsilon
}^{\xi}(t-s)(u^{\prime}(t)-u^{\prime}(s))^{2}ds\\
&  \geq\frac{\varepsilon(\beta-\alpha)^{2}}{4}\left(  \int_{\alpha}^{\beta
}\frac{1}{J^{\xi}(z)}~dz\right)  ^{-1}(u^{\prime}(t))^{2}~,
\end{align*}
and integrating over $A_0$, using (\ref{J2bis}) and (\ref{F51}), we obtain
\begin{equation}
\varepsilon\int_{A_0}\int_{t-\varepsilon\beta(\xi)}^{t-\varepsilon\alpha(\xi
)}J_{\varepsilon}^{\xi}(t-s)(u^{\prime}(t)-u^{\prime}(s))^{2}dsdt\geq
\frac{\varepsilon\delta_{\!J}^{2}}{2^{6}c_{\!J}}\int_{A_0}(u^{\prime}%
(t))^{2}dt~. \label{c5}%
\end{equation}
By (\ref{c2ab}), (\ref{c2ab-bis}), and (\ref{c4}) using Jensen's inequality
and $\tau-\sigma\leq\frac{\delta_{\!J}}{2^{6}}\varepsilon$, we have%
\[
\int_{A_0}(u^{\prime}(t))^{2}dt=\int_{\sigma}^{\tau}(u^{\prime}(t))^{2}%
dt-\int_{(\sigma,\tau)\setminus A_0}\!\!\!\!\!\!\!\!\!\!\!\!(u^{\prime}%
(t))^{2}dt\geq\frac{1}{\tau-\sigma}-\frac{1}{4}\frac{1}{\tau-\sigma}\geq
\frac{3\cdot2^{4}}{\varepsilon\delta_{\!J}}~.
\]
Hence, from (\ref{c5}) we deduce that
\begin{equation}
\varepsilon\int_{\sigma}^{\tau}\int_{\sigma-\varepsilon\beta(\xi)}%
^{\tau-\varepsilon\alpha(\xi)}J_{\varepsilon}^{\xi}(t-s)(u^{\prime
}(t)-u^{\prime}(s))^{2}dsdt\geq\frac{3}{4}\frac{\delta_{\!J}}{c_{\!J}}~.
\label{c6}%
\end{equation}

\noindent\textbf{Case 2: }It remains to study the case in which there exists
$t_{0}\in A_0$ such that
\[
\frac{|u(t_{0}-\varepsilon\alpha)-u(t_{0}-\varepsilon\beta)|}{\varepsilon
(\beta-\alpha)}\geq\frac{1}{2}|u_{\varepsilon}^{\prime}(t_{0})|
\]
for every $\alpha\in\lbrack\alpha(\xi),\hat{\alpha}]$ and for every $\beta
\in\lbrack\hat{\beta},\beta(\xi)]$. By (\ref{c4}) and the inequality
$\tau-\sigma\leq\varepsilon\delta_{\!J}/2^{6}$, we have
\[
\frac{|u(t_{0}-\varepsilon\alpha)-u(t_{0}-\varepsilon\beta)|}{\varepsilon
(\beta-\alpha)}\geq\frac{1}{4(\tau-\sigma)}\geq\frac{16}{\varepsilon\delta_{\!J}}\,,
\]
hence by (\ref{F51}),%
\[
|u(t_{0}-\varepsilon\alpha)-u(t_{0}-\varepsilon\beta)|\geq\frac{16(\hat{\beta
}-\hat{\alpha})}{\delta_{\!J}}\geq4~.
\]

If $|u(t_{0}-\varepsilon\alpha)|\geq2$ for every $\alpha\in\lbrack\alpha
(\xi),\hat{\alpha}]$, then by (\ref{W2}) we have $W(u(t_{0}-\varepsilon
\alpha))\geq\frac{1}{c_{W}}$ for every $\alpha\in\lbrack\alpha(\xi
),\hat{\alpha}]$. This leads to $W(u(t))\geq\frac{1}{c_{W}}$ for every
$t\in\lbrack t_{0}-\varepsilon\hat{\alpha},t_{0}-\varepsilon\alpha(\xi)]$,
hence
\begin{equation}
\frac{1}{\varepsilon}\int_{\sigma-\varepsilon\gamma_{J}}^{\tau+\varepsilon
\gamma_{J}}\!\!\!\! \!\!\!\! \!\!\!\! W(u(t))~dt\geq\frac{1}{\varepsilon}%
\int_{t_{0}-\varepsilon\hat{\alpha}}^{t_{0}-\varepsilon\alpha(\xi)}\!\!\!\!
\!\!\!\! \!\!\!\! W(u(t))~dt\geq\frac{\hat{\alpha}-\alpha(\xi)}{c_{W}}
\geq\frac{\delta_{\!J}}{4c_{W}}~, \label{c7bis}%
\end{equation}
where in the last inequality we used (\ref{F51}).

If there exists $\alpha\in\lbrack\alpha(\xi),\hat{\alpha}]$ such that
$|u(t_{0}-\varepsilon\alpha)|<2$, then $|u(t_{0}-\varepsilon\beta)|>2$ for
every $\beta\in\lbrack\hat{\beta},\beta_{J}]$ (if not, there exists $\beta
\in\lbrack\hat{\beta},\beta(\xi)]$ such that $|u(t_{0}-\varepsilon\beta
)|\leq2$, which gives $|u(t_{0}-\varepsilon\alpha)-u(t_{0}-\varepsilon
\beta)|<4$, a contradiction). Consequently, for every $\beta\in\lbrack
\hat{\beta},\beta(\xi)]$ we have $W(u(t_{0}-\varepsilon\beta))\geq\frac
{1}{c_{W}}$. This leads to $W(u(t))\geq\frac{1}{c_{W}}$ for every $t\in\lbrack
t_{0}-\varepsilon\beta(\xi),t_{0}-\varepsilon\hat{\beta}]$, hence
\begin{equation}
\frac{1}{\varepsilon}\int_{\sigma-\varepsilon\gamma_{J}}^{\tau+\varepsilon
\gamma_{J}}\!\!\!\!\!\!\!\!\!\!\!\!W(u(t))~dt\geq\frac{1}{\varepsilon}%
\int_{t_{0}-\varepsilon\beta(\xi)}^{t_{0}-\varepsilon\hat{\beta}%
}\!\!\!\!\!\!\!\!\!\!\!\!W(u(t))~dt\geq\frac{\beta(\xi)-\hat{\beta}}{c_{W}%
}\ \geq\frac{\delta_{\!J}}{4c_{W}}~, \label{c7}%
\end{equation}
where in the last inequality we used (\ref{F51}). The conclusion follows now
from (\ref{c3}), (\ref{c6}), (\ref{c7bis}), and (\ref{c7}).
\end{proof}

\begin{lemma}
[Interpolation inequality in dimension one]\label{F33} There exists a constant
$c_{J,W}^{(1)}$ such that
\begin{equation}
\varepsilon\int_{A}(u^{\prime}(t))^{2}dt\leq c_{J,W}^{(1)}\mathcal{F}%
_{\varepsilon}^{\xi}(u,({A)}^{2\varepsilon\gamma_{J}})~. \label{F35}%
\end{equation}
for every $\xi\in\mathbb{S}^{n-1}$, for every $\varepsilon>0$, for every open
set $A\subset\mathbb{R}$, and for every $u\in W_{\operatorname*{loc}}%
^{1,2}((A)^{2\varepsilon\gamma_{J}})$, where $\gamma_{J}$ is the constant in
(\ref{J1a}).
\end{lemma}

\begin{proof}
Fix $\xi$, $\varepsilon$, $A$, and $u$ as in the statement of the lemma, and
define
\begin{align}
&  U:=\{t\in A:u(t-\varepsilon\alpha(\xi)),u(t-\varepsilon\beta(\xi
))\notin\lbrack\tfrac{1}{2},\tfrac{3}{2}]\}~,\nonumber\\
&  V:=\{t\in A:u(t-\varepsilon\alpha(\xi)),u(t-\varepsilon\beta(\xi
))\notin\lbrack-\tfrac{3}{2},-\tfrac{1}{2}]\}~. \label{F34}%
\end{align}
If $t\in V$, then by (\ref{F31}),
\begin{align}
(u(t-\varepsilon\alpha(\xi))-u(t-\varepsilon\beta(\xi)))^{2}  &  {}%
\leq2(u(t-\varepsilon\alpha(\xi))-1)^{2}+2(u(t-\varepsilon\beta(\xi
))-1)^{2}\nonumber\\
&  \leq2\hat{c}_{W}\big(W(u(t-\varepsilon\alpha(\xi)))+W(u(t-\varepsilon
\beta(\xi)))\big)~.\nonumber
\end{align}
Using (\ref{F41}) we prove the same inequality for $t\in U$. Integrating and
using Remark \ref{remark 1}, we obtain
\begin{equation}
\varepsilon\int_{U\cup V}(u^{\prime}(t))^{2}dt\leq\big(8\frac{\hat{c}_{W}}{\delta_{\!J}^{2}}
+2c_{\!J}\big)\mathcal{F}_{\varepsilon}^{\xi}%
(u,({A)^{\varepsilon\gamma_{J}}})~. \label{F36}%
\end{equation}

If $t\in A\setminus(U\cup V)$, then either
\[
u(t-\varepsilon\alpha(\xi))\in\lbrack-\tfrac{3}{2},-\tfrac{1}{2}%
]\quad\text{and}\quad u(t-\varepsilon\beta(\xi))\in\lbrack\tfrac{1}{2}%
,\tfrac{3}{2}]
\]
or
\[
u(t-\varepsilon\beta(\xi))\in\lbrack-\tfrac{3}{2},-\tfrac{1}{2}]\quad
\text{and}\quad u(t-\varepsilon\alpha(\xi))\in\lbrack\tfrac{1}{2},\tfrac{3}%
{2}]~.
\]
Then
\begin{equation}
(u(t-\varepsilon\alpha(\xi))-u(t-\varepsilon\beta(\xi)))^{2}\leq9~.
\label{F39}%
\end{equation}
Moreover there exist $\sigma$ and $\tau$, satisfying
\begin{equation}
t-\varepsilon\gamma_{J}\leq t-\varepsilon\beta(\xi)\leq\sigma<\tau\leq
t-\varepsilon\alpha(\xi)\leq t+\varepsilon\gamma_{J} \label{F37}%
\end{equation}
and such that
\begin{equation}
u(t)\in\left(  -\tfrac{1}{2},\tfrac{1}{2}\right)  \text{ for every }%
t\in(\sigma,\tau)\nonumber
\end{equation}
and either%
\[
u(\sigma)=\tfrac{1}{2}\quad\text{and}\quad u(\tau)=-\tfrac{1}{2}%
\]
or
\[
u(\sigma)=-\tfrac{1}{2}\quad\text{and}\quad u(\tau)=\tfrac{1}{2}~.
\]
By Lemma \ref{lemma compactness} and by (\ref{F37}), there exists $c_{J,W}>0$
such that
\begin{equation}
c_{J,W}\leq\varepsilon\int_{t-\varepsilon\gamma_{J}}^{t+\varepsilon\gamma_{J}%
}\int_{t-2\varepsilon\gamma_{J}}^{t+2\varepsilon\gamma_{J}}%
\!\!\!\!\!\!\!\!\!J_{\varepsilon}^{\xi}(r-s)(u_{\varepsilon}^{\prime
}(r)-u_{\varepsilon}^{\prime}(s))^{2}dsdr+\frac{1}{\varepsilon}\int%
_{t-2\varepsilon\gamma_{J}}^{t+2\varepsilon\gamma_{J}}%
\!\!\!\!\!\!\!\!W(u_{\varepsilon}(r))~dr~.\nonumber
\end{equation}
Therefore by (\ref{F39}) we have
\begin{align}
\frac{1}{\varepsilon}\int_{A\setminus(U\cup V)}%
\!\!\!\!\!\!\!\!\!\!\!\!\!\!\!\!  &  (u(t-\varepsilon\alpha(\xi
))-u(t-\varepsilon\beta(\xi)))^{2}dt\nonumber\\
&  \leq\frac{9}{c_{J,W}}\int_{A}\int_{t-\varepsilon\gamma_{J}}^{t+\varepsilon
\gamma_{J}}\int_{t-2\varepsilon\gamma_{J}}^{t+2\varepsilon\gamma_{J}%
}\!\!\!\!\!\!\!\!\!J_{\varepsilon}^{\xi}(r-s)(u_{\varepsilon}^{\prime
}(r)-u_{\varepsilon}^{\prime}(s))^{2}dsdrdt\label{F43}\\
&  \quad+\frac{9}{c_{J,W}}\frac{1}{\varepsilon^{2}}\int_{A}\int%
_{t-2\varepsilon\gamma_{J}}^{t+2\varepsilon\gamma_{J}}%
\!\!\!\!\!\!\!\!W(u_{\varepsilon}(r))~drdt~.\nonumber
\end{align}
Since
\[
\frac{1}{2\eta}\int_{A}\int_{t-\eta}^{t+\eta}f(r)~drdt\leq\int_{(A)_{\eta}%
}f(t)~dt
\]
for every $\eta>0$ and for every integrable function $f\colon A\rightarrow
\lbrack0,+\infty]$, from (\ref{F43}) we obtain
\begin{equation}
\frac{1}{\varepsilon}\int_{A\setminus(U\cup V)}%
\!\!\!\!\!\!\!\!\!\!\!\!\!\!\!\!(u(t-\varepsilon\alpha(\xi))-u(t-\varepsilon
\beta(\xi)))^{2}dt\leq\tilde{c}_{J,W}\mathcal{F}_{\varepsilon
}^{\xi}(u,{(A)^{2\varepsilon\gamma_{J}}})~.\label{F44}%
\end{equation}
for a suitable constant $\tilde{c}_{J,W}$ depending only on $J$ and $W$. The conclusion follows from (\ref{F36}) and (\ref{F44}) using Remark
\ref{remark 1}.
\end{proof}

\medskip
\begin{proof}
[Proof of Theorem \ref{theorem compactness}]By (\ref{c1}) we have that%
\begin{equation}
\int_{A}W(u_{\varepsilon}(t))~dt\leq M\varepsilon~. \label{c7a}%
\end{equation}
By (\ref{W1}) and (\ref{W2}) this implies that $\{u_{\varepsilon}^{2}\}$
converges to $1$ in $L^{1}(A)$ and, up to a subsequence (not relabeled)
pointwise a.e. in $A$. 

Let $\gamma_{J}>0$ be the constant given in (\ref{J1a}). Consider the
collection $\mathcal{I}_{\varepsilon}$ of all intervals $(\sigma
-\varepsilon\gamma_{J},y_{\varepsilon}+\varepsilon\gamma_{J})$ such that
$(\sigma,\tau)$ is contained in $(A)^{\varepsilon\gamma_{J}}$, and
$u_{\varepsilon}$ satisfies (\ref{c2}) and either (\ref{c2ab}) or
(\ref{c2ab-bis}) in $(\sigma,\tau)$. Note that by the intermediate value
theorem for all $\varepsilon>0$ sufficiently small there exist such intervals.
Moreover, by construction, all intervals in $\mathcal{I}_{\varepsilon}$ are
contained in $A$. It follows from (\ref{W2}) and (\ref{c7a}) that%
\[
M\varepsilon\geq\int_{\sigma}^{\tau}W(u_{\varepsilon}(t))~dt\geq\frac
{\tau-\sigma}{4c_{W}}~,
\]
hence
\begin{equation}
\tau-\sigma\leq4c_{W}M\varepsilon\label{c7b}~.\end{equation}
In particular, for every $I\in\mathcal{I}_{\varepsilon}$ we have%
\begin{equation}
\operatorname*{diam}I\leq(4c_{W}M+2\gamma_{J})\varepsilon~. \label{c7c}%
\end{equation}
Moreover, by (\ref{c1}) and (\ref{c2aa}), if $I_{1}$, \ldots, $I_{k}$ are
pairwise disjoint intervals in $\mathcal{I}_{\varepsilon}$, then
\begin{equation}
k\leq\frac{M}{c_{J,W}}~. \label{c8}%
\end{equation}

Let $B_{\varepsilon}$ be the union of all intervals in $\mathcal{I}%
_{\varepsilon}$ and let $\mathcal{C}_{\varepsilon}$ be the collection of its
connected components. Observe that distinct elements of $\mathcal{C}%
_{\varepsilon}$ must contain disjoint intervals of $\mathcal{I}_{\varepsilon}%
$, and so by (\ref{c8}) the number of elements of $\mathcal{C}_{\varepsilon}$
is uniformly bounded. To be precise,%
\begin{equation}
\#\mathcal{C}_{\varepsilon}\leq\frac{M}{c_{J,W}}~. \label{c8a}%
\end{equation}
Next we claim that if $C\in\mathcal{C}_{\varepsilon}$, then%
\begin{equation}
\operatorname*{diam}C\leq2(4C_{W}M+2\gamma_{J})\left(  \frac{M}{c_{J,W}%
}+1\right)  \varepsilon~. \label{c9}%
\end{equation}
Assume by contradiction that (\ref{c9}) fails. Let $k$ be the integer such
that $\frac{M}{c_{J,W}}<k\leq\frac{M}{c_{J,W}}+1$ and partition $C$ into $k$
subintervals $C_{1}$, \ldots, $C_{k}$ of equal length larger that
$2(4C_{W}M+2\gamma_{J})\varepsilon$. The middle point of each $C_{i}$ belongs
to some interval $I_{i}\in\mathcal{I}_{\varepsilon}$. By (\ref{c7c}), we have
that $I_{i}\subset C_{i}$ and so $I_{1}$, \ldots, $I_{k}$ are pairwise
disjoint. In turn $k$ satisfies (\ref{c8}), which contradicts its definition.
This concludes the proof of (\ref{c9}).

In view of (\ref{c8a}) there exist a sequence $\varepsilon_{j}\rightarrow
0^{+}$ and a nonnegative integer $k\leq\frac{M}{c_{J,W}}$ such that
$\#\mathcal{C}_{\varepsilon_{j}}=k$ for all $j\in\mathbb{N}$. Write
$\mathcal{C}_{\varepsilon_{j}}=\{C_{j}^{1},\ldots,C_{j}^{k}\}$ and choose
$t_{j}^{i}\in C_{j}^{i}$. Up to a subsequence (not relabeled) we may assume
that $t_{j}^{i}\rightarrow t^{i}\in\overline{A}$ for all $i=1$, \ldots, $k$.
By (\ref{c9}) for every $\eta>0$ we have that $C_{j}^{i}\subset\lbrack
t^{i}-\eta,t^{i}+\eta]$ for all $j$ sufficiently large. Let $S:=\{t^{1}%
,\ldots,t^{k}\}$ and let $K$ be a closed interval contained in $A\setminus S$.
Then $B_{\varepsilon_{j}}\cap K=\emptyset$ for all $j$ sufficiently large. We
claim that for all such $j$ either $\inf_{K}u_{\varepsilon_{j}}\geq-\frac
{1}{2}$ or $\sup_{K}u_{\varepsilon_{j}}\leq\frac{1}{2}$. Indeed, if this does
not hold then we can find $\sigma_{j}$ and $\tau_{j}$ in $K$ for which
$u_{\varepsilon_{j}}$ satisfies (\ref{c2}) and either (\ref{c2ab}) or
(\ref{c2ab-bis}). On the one hand $(\sigma_{j},\tau_{j})\subset B_{\varepsilon
_{j}}$ by the definition of $B_{\varepsilon_{j}}$. On the other hand
$(\sigma_{j},\tau_{j})\subset K$ since $K$ an interval. Therefore $(\sigma_{j}%
,\tau_{j})\subset B_{\varepsilon_{j}}\cap K$ and this contradicts the fact
that $B_{\varepsilon_{j}}\cap K=\emptyset$.

We extract a subsequence, possibly depending on $K$, not relabelled, such
that, either $\inf_{K}u_{\varepsilon_{j}}\geq-\frac{1}{2}$ for all $j$ or
$\sup_{K}u_{\varepsilon_{j}}\leq\frac{1}{2}$ for all $j$. Since
$u_{\varepsilon_{j}}^{2}(t)\rightarrow1$ for a.e. $t\in K$, we conclude that
$u_{\varepsilon_{j}}(t)\rightarrow1$ for a.e. $t\in K$ in the former case
while $u_{\varepsilon_{j}}(t)\rightarrow-1$ for a.e. $t\in K$ in the latter.
By iterating this argument with an increasing sequence of compact intervals
$K$ whose union is a connected component of $A\setminus S$, it follows by a
diagonal argument that a subsequence $\{u_{\varepsilon_{j}}\}$ (not relabeled)
converges pointwise a.e in $A\setminus S$ to a function $u$ constantly equal
to $-1$ or $1$ in each connected component of $A\setminus S$. This implies
that $u\in BV(A;\{-1,1\})$ with $S_{u}\subset S$, hence $\#S_{u}\leq\#S\leq
k\leq\frac{M}{c_{J,W}}$. The $L^{2}$ convergence of $\{u_{\varepsilon_{j}}\}$
to $u$ now follows from (\ref{W2}) and (\ref{c7a}).
\end{proof}

\section{Compactness and interpolation for $n\geq2$}

\label{section compactness n 2}

Given $a\in\mathbb{R}$ we define
\begin{equation}
a^{(1)}:=(-1)\vee(a\wedge1)~. \label{C0}%
\end{equation}

\begin{lemma}
\label{lemma truncated}Let $\{u_{\varepsilon}\}\subset L^{2}\left(
\Omega\right)  $ be such that%
\begin{equation}
M:=\sup_{\varepsilon}\mathcal{W}_{\varepsilon}(u_{\varepsilon})<+\infty~.
\label{C1}%
\end{equation}
Then $u_{\varepsilon}-u_{\varepsilon}^{(1)}\rightarrow0$ strongly in
$L^{2}(\Omega)$.
\end{lemma}

\begin{proof}
By\ (\ref{W functional}) and (\ref{C1}) we have that
\begin{equation}
\int_{\Omega}W(u_{\varepsilon}(x))~dx\rightarrow0 \label{C2}%
\end{equation}
as $\varepsilon\rightarrow0^{+}$. By (\ref{W1}) and (\ref{W2}) this implies
that, up to a subsequence, $|u_{\varepsilon}(x)|\rightarrow1$ for a.e. $x\in\Omega$. Hence,
$u_{\varepsilon}(x)-u_{\varepsilon}^{(1)}(x)\rightarrow0$ for a.e. $x\in
\Omega$. On the other hand, by (\ref{W2}),%
\[
(u_{\varepsilon}(x)-u_{\varepsilon}^{(1)}(x))^{2}\leq(u_{\varepsilon}%
(x))^{2}\leq\frac{2}{c_{W}}W(u_{\varepsilon}(x))+2~,
\]
so that the conclusion follows from (\ref{C1}) and the (generalized) Lebesgue
dominated convergence theorem.
\end{proof}

In what follows, given a Borel set $E\subset\mathbb{R}^{n}$ and a function
$u:E\rightarrow\mathbb{R}$, for every $\xi\in\mathbb{S}^{n-1}$ and for every
$y\in\Pi^{\xi}$ (see (\ref{PI xi})) we define the one-dimensional function
\begin{equation}
u_{y}^{\xi}(t):=u(y+t\xi)~,\quad t\in E_{y}^{\xi}~, \label{u slice}%
\end{equation}
where $E_{y}^{\xi}$ is defined in (\ref{E slice}).

\begin{lemma}
\label{lemma slice energy}For every $A\subset\mathbb{R}^{n}$ open,
$\varepsilon>0$, and $u\in W_{\operatorname*{loc}}^{1,2}(A)\cap L^{2}(A)$, we
have%
\[
\mathcal{F}_{\varepsilon}(u,A)\geq\int_{\mathbb{S}^{n-1}}\int_{\Pi^{\xi}%
}\mathcal{F}_{\varepsilon}^{\xi}(u_{z}^{\xi},A_{z}^{\xi})~d\mathcal{H}%
^{n-1}(z)d\mathcal{H}^{n-1}(\xi)~.
\]

\end{lemma}

\begin{proof}
By Fubini's theorem, Proposition \ref{proposition blaschke}, (\ref{214bis}), (\ref{p3}), and (\ref{p5}), we obtain
\begin{align*}
&  \mathcal{F}_{\varepsilon}(u,A)\\
&  =\frac{1}{\sigma_{n-1}\varepsilon}\int_{\mathbb{S}^{n-1}}\int%
_{\Pi^{\xi}}\int_{A_{z}^{\xi}}W(u(z+t\xi))~dtd\mathcal{H}^{n-1}(z)d\mathcal{H}%
^{n-1}(\xi)\\
&  \quad+\frac{\varepsilon}{2}\!\int_{\mathbb{S}^{n-1}}\!\int_{\Pi^{\xi}}\!\int%
_{A_{z}^{\xi}}\!\int_{A_{z}^{\xi}}\!\!J_{\varepsilon}^{\xi}(t\!-\!s)|\nabla
u(z\!+\!t\xi)\!-\!\nabla u(z\!+\!s\xi)|^{2}dtdsd\mathcal{H}^{n-1}\!(z)d\mathcal{H}^{n-1}%
\!(\xi)\\
&  \geq\frac{1}{\sigma_{n-1}\varepsilon}\int_{\mathbb{S}^{n-1}}\int%
_{\Pi^{\xi}}\int_{A_{z}^{\xi}}W(u_{z}^{\xi}(t))~dtd\mathcal{H}^{n-1}%
(z)d\mathcal{H}^{n-1}(\xi)\\
&  \quad+\frac{\varepsilon}{2}\int_{\mathbb{S}^{n-1}}\int_{\Pi^{\xi}}\int%
_{A_{z}^{\xi}}\int_{A_{z}^{\xi}}J_{\varepsilon}^{\xi}(t-s)((u_{z}^{\xi
})^{\prime}(t)-(u_{z}^{\xi})^{\prime}(s))^{2}dtdsd\mathcal{H}^{n-1}%
(z)d\mathcal{H}^{n-1}(\xi)\\
&  =\int_{\mathbb{S}^{n-1}}\int_{\Pi^{\xi}}\mathcal{F}_{\varepsilon}^{\xi
}(u_{z}^{\xi},A_{z}^{\xi})~d\mathcal{H}^{n-1}(z)d\mathcal{H}^{n-1}(\xi)~.
\end{align*}

\end{proof}

\bigskip

\begin{proof}
[Proof of Theorem \ref{theorem compactness n>1}]Let $\varepsilon
_{j}\rightarrow0^{+}$ and, for simplicity, write $u_{j}:=u_{\varepsilon_{j}}$.
By Lemma \ref{lemma slice energy},%
\begin{equation}
\int_{\mathbb{S}^{n-1}}\int_{\Pi^{\xi}}\mathcal{F}_{\varepsilon_{j}}^{\xi
}((u_{j})_{z}^{\xi},\Omega_{z}^{\xi})~d\mathcal{H}^{n-1}(z)d\mathcal{H}%
^{n-1}(\xi)\leq M~. \label{D1}%
\end{equation}
We claim that there exist a collection $\xi_{1}$, \ldots, $\xi_{n}%
\in\mathbb{S}^{n-1}$ of linearly independent vectors and a subsequence (not
relabeled) such that
\begin{equation}
\lim_{j\rightarrow+\infty}\int_{\Pi^{\xi_{i}}}\mathcal{F}_{\varepsilon_{j}%
}^{\xi_{i}}((u_{j})_{z}^{\xi_{i}},\Omega_{z}^{\xi_{i}})~d\mathcal{H}%
^{n-1}(z)=:M_{i}<+\infty~,\label{C3}%
\end{equation}
for every $i=1$, \ldots, $n$.

Indeed, using Fatou's lemma by (\ref{D1}) we have that%
\begin{equation}
\int_{\mathbb{S}^{n-1}}\liminf_{j\rightarrow+\infty}\int_{\Pi^{\xi}%
}\mathcal{F}_{\varepsilon_{j}}^{\xi}((u_{j})_{z}^{\xi},\Omega_{z}^{\xi
})~d\mathcal{H}^{n-1}(z)d\mathcal{H}^{n-1}(\xi)\leq M~. \label{D2}%
\end{equation}
Hence, there exists $\xi_{1}\in\mathbb{S}^{n-1}$ such that%
\begin{equation}
\liminf_{j\rightarrow+\infty}\int_{\Pi^{\xi_{1}}}\mathcal{F}_{\varepsilon_{j}%
}^{\xi_{1}}((u_{j})_{z}^{\xi_{1}},\Omega_{z}^{\xi_{1}})~d\mathcal{H}%
^{n-1}(z)=:M_{1}<+\infty~, \label{C2A}%
\end{equation}
and we can extract a subsequence (not relabeled) such that (\ref{C3}) holds
for $i=1$.

We proceed by induction. Assume that we found a collection $\xi_{1}$, \ldots,
$\xi_{k}\in\mathbb{S}^{n-1}$, $1\leq k<n$, of linearly independent vectors and
a subsequence (not relabeled) such that (\ref{C3}) holds for every $i=1$,
\ldots, $k$. Note that this subsequence still satisfies (\ref{D1}), and hence
(\ref{D2}). Therefore we can find $\xi_{k+1}\in\mathbb{S}^{n-1}$, linearly
independent of $\xi_{1}$, \ldots, $\xi_{k}$, such that
\[
\liminf_{j\rightarrow+\infty}\int_{\Pi^{\xi_{k+1}}}\mathcal{F}_{\varepsilon
_{j}}^{\xi_{k+1}}((u_{j})_{z}^{\xi_{k+1}},\Omega_{z}^{\xi_{k+1}}%
)~d\mathcal{H}^{n-1}(z)=:M_{k+1}<+\infty~,
\]
and we can extract a subsequence (not relabeled) such that (\ref{C3}) holds
also for $i=k+1$. After $n$ steps we obtain that (\ref{C3}) is satisfied for
every $i=1$, \ldots, $n$.

Given $i=1$, \ldots, $n$ and $\delta>0$, for every $j$ let
\begin{equation}
A_{j}^{i}:=\Big\{z\in\Pi^{\xi_{i}}: \mathcal{F}_{\varepsilon_{j}}^{\xi_{i}
}((u_{j})_{z}^{\xi_{i}},\Omega_{z}^{\xi_{i}})>\frac{M_{i}}{\delta}\Big\}~,
\label{D3}%
\end{equation}
and let $v_{j}^{i}\in L^{2}(\Omega)$ be defined by
\begin{equation}%
\begin{cases}
(v_{j}^{i})_{z}^{\xi_{i}}:=(u_{j}^{(1)})_{z}^{\xi_{i}} & \hbox{if }z\in
\Pi^{\xi_{i}}\setminus A_{j}~,\\
(v_{j}^{i})_{z}^{\xi_{i}}:=0 & \hbox{if }z\in A_{j}~,
\end{cases}
\label{D4}%
\end{equation}
where $u_{j}^{(1)}$ is the truncated function defined using (\ref{C0}). By
(\ref{C3}) and (\ref{D3}) we have
\[\limsup_{j\rightarrow+\infty}\,\mathcal{H}^{n-1}(A_{j}^{i})\leq\delta~,
\]
hence (\ref{D4}) yields
\begin{equation}\limsup_{j\rightarrow+\infty}\,\|v_{j}^{i}-u_{j}^{(1)}\|_{L^{2}(\Omega)}^{2}\le\delta\operatorname*{diam}%
(\Omega)~. \label{D16}%
\end{equation}

By Theorem \ref{theorem compactness} for every $z\in \Pi^{\xi_{i}}$ the set
$\{(u_{j})_{z}^{\xi_{i}}(1-\chi_{A_{j}^{i}}(z)):j\in\mathbb{N}\}$ is relatively compact in
$L^{2}(\Omega_{z}^{\xi_{i}})$, where $\chi_{A_{j}^{i}}(z)=1$ for $z\in A_{j}^{i}$ and
 $\chi_{A_{j}^{i}}(z)=0$ for $z\not\in A_{j}^{i}$. Therefore the same property holds for the set
of truncated functions $\{(u_{j}^{(1)})_{z}^{\xi_{i}}(1-\chi_{A_{j}^{i}}(z)):j\in\mathbb{N}\}$. It
follows that for every $z\in\Pi^{\xi_{i}}$ the set $\{(v_{j}^{i})_{z}^{\xi
_{i}}:j\in\mathbb{N}\}$ is relatively compact in $L^{2}(\Omega_{z}^{\xi_{i}}%
)$. Since this property is valid for every $i=1$, \ldots, $n$, we can apply
the characterization by slicing of precompact sets of $L^{2}(\Omega)$ given by
\cite[Theorem 6.6]{alberti-bouchitte-seppecher1998} and we obtain that the set $\{u_{j}^{(1)}%
:j\in\mathbb{N}\}$ is relatively compact in $L^{2}(\Omega)$. In turn, by Lemma
\ref{lemma truncated} the set $\{u_{j}:j\in\mathbb{N}\}$ is relatively compact
in $L^{2}(\Omega)$, hence there exist a subsequence (not relabeled) , such
that $u_{j}$ converges in $L^{2}(\Omega)$ to some function $u$. By
(\ref{D7}),
\[
\lim_{j\rightarrow+\infty}\int_{\Omega}W(u_{j}(x))~dx=0~,
\]
which, together with (\ref{W1}) and (\ref{W2}), implies that $u(x)\in\{-1,1\}$
for a.e. $x\in\Omega$.

It remains to show that $u\in BV(\Omega)$. Using Fubini's theorem we find that
there exists a subsequence (not relabeled) such that
\begin{equation}
(u_{j})_{z}^{\xi_{i}}\rightarrow u_{z}^{\xi_{i}}\hbox{ in }L^{2}(\Omega
_{z}^{\xi_{i}})~. \label{D15}%
\end{equation}
Moreover, Fatou's lemma and (\ref{C3}) imply that%
\begin{equation}
\int_{\Pi^{\xi_{i}}}\liminf_{j\rightarrow+\infty}\mathcal{F}_{\varepsilon_{j}%
}^{\xi_{i}}((u_{j})_{z}^{\xi_{i}},\Omega_{z}^{\xi_{i}})~d\mathcal{H}%
^{n-1}(z)\leq M_{i}~, \label{D8}%
\end{equation}
hence
\begin{equation}
\liminf_{j\rightarrow+\infty}\mathcal{F}_{\varepsilon_{j}}^{\xi_{i}}%
((u_{j})_{z}^{\xi_{i}},\Omega_{z}^{\xi_{i}})<+\infty\label{D9}%
\end{equation}
for $\mathcal{H}^{n-1}$-a.e. $z\in\Pi^{\xi_{i}}$. Fix $z\in\Pi^{\xi_{i}}$
satisfying (\ref{D15}) and (\ref{D9}), and extract a subsequence $\{\hat
{u}_{j}\}$, depending on $z$, such that
\begin{equation}
\lim_{j\rightarrow+\infty}\mathcal{F}_{\varepsilon_{j}}^{\xi_{i}}((\hat{u}%
_{j})_{z}^{\xi_{i}},\Omega_{z}^{\xi_{i}})=\liminf_{j\rightarrow+\infty
}\mathcal{F}_{\varepsilon_{j}}^{\xi_{i}}((u_{j})_{z}^{\xi_{i}},\Omega_{z}%
^{\xi_{i}})~. \label{D19}%
\end{equation}
By (\ref{D10}),  (\ref{D15}), and (\ref{D19}) we have
\[
\#S_{u_{z}^{\xi_{i}}}\leq\frac{1}{c_{J,W}}\liminf_{j\rightarrow+\infty}\mathcal{F}_{\varepsilon_{j}}^{\xi_{i}}%
((u_{j})_{z}^{\xi_{i}},\Omega_{z}^{\xi_{i}})~.
\]
Since $u_{z}^{\xi_{i}}(t)\in\{-1,1\}$ for a.e. $t\in\Omega_{z}^{\xi_{i}}$, we
deduce that
\[
|Du_{z}^{\xi_{i}}|(\Omega_{z}^{\xi_{i}})\leq\frac{2}{c_{J,W}}%
\liminf_{j\rightarrow+\infty}\mathcal{F}_{\varepsilon_{j}}^{\xi_{i}}%
((u_{j})_{z}^{\xi_{i}},\Omega_{z}^{\xi_{i}})\]
for $\mathcal{H}^{n-1}$-a.e. $z\in\Pi^{\xi_{i}}$. This property holds for
every $i=1$, \ldots, $n$. Therefore, we can apply the characterization by
slicing of $BV$ functions given by \cite[Remark 3.104]%
{ambrosio-fusco-pallara2000} and we obtain from (\ref{D8}) that $u\in BV(\Omega)$.
\end{proof}

For $A\subset\mathbb{R}^{n}$ and $\eta>0$ we recall the notation (\ref{F49}).

\begin{lemma}
[Interpolation inequality]\label{F53} There exists a constant $c_{J,W}^{(n)}$
such that
\begin{equation}
\varepsilon\int_{A}|\nabla u(x)|^{2}dx\leq c_{J,W}^{(n)}\mathcal{F}%
_{\varepsilon}(u,{(A)}^{2\varepsilon\gamma_{J}})~. \label{F54}%
\end{equation}
for every $\varepsilon>0$, for every open set $A\subset\mathbb{R}^{n}$, and
for every $u\in W_{\operatorname*{loc}}^{1,2}((A)_{2\varepsilon\gamma_{J}})$,
where $\gamma_{J}$ is the constant in (\ref{J1a}).
\end{lemma}

\begin{proof}
Fix $\varepsilon$, $A$, and $u$ as in the statement of the lemma, and define
$B:=(A)^{2\varepsilon\gamma_{J}}$. Given $\xi\in\mathbb{S}^{n-1}$, for
$\mathcal{H}^{n-1}$ a.e. $z\in\Pi^{\xi}$ we have that $(A_{z}^{\xi
})^{2\varepsilon\gamma_{J}}\subset B_{z}^{\xi}$ and the sliced function
$u_{z}^{\xi}$ (see (\ref{u slice})) belongs to
$W_{\operatorname*{loc}}^{1,2}(B_{z}^{\xi})$. Hence by Lemma \ref{F33} we
have
\begin{equation}
\varepsilon\int_{A_{z}^{\xi}}((u_{z}^{\xi})^{\prime}(t))^{2}dt\leq
c_{J,W}^{(1)}\mathcal{F}_{\varepsilon}^{\xi}(u_{z}^{\xi},B_{z}^{\xi
})~.\nonumber
\end{equation}
Integrating this inequality in $z$ over $\Pi^{\xi}$ we obtain
\begin{equation}
\varepsilon\int_{A}(\nabla u(x)\cdot\xi)^{2}dx\leq c_{J,W}^{(1)}\int_{\Pi
^{\xi}}\mathcal{F}_{\varepsilon}^{\xi}(u_{z}^{\xi},B_{z}^{\xi})~d\mathcal{H}%
^{n-1}(z)~.\nonumber
\end{equation}
Integrating this inequality in $\xi$ over $\mathbb{S}^{n-1}$ and using Lemma
\ref{lemma slice energy}, together with the identity $\int_{\mathbb{S}^{n-1}%
}|a\cdot\xi|^{2}d\mathcal{H}^{n-1}(\xi)=\omega_{n}|a|^{2}$, we deduce
\begin{equation}
\omega_{n}\varepsilon\int_{A}|\nabla u(x)|^{2}dx\leq c_{J,W}^{(1)}%
\mathcal{F}_{\varepsilon}(u,B)~.\nonumber
\end{equation}
This concludes the proof.
\end{proof}

\section{The modification theorem}

\label{section modification}

In this section we prove that we can modify an admissible sequence to match a
mollification of its limit in a neighborhood of the boundary, without
increasing the limit energy.

Given $\nu\in\mathbb{S}^{n-1}$, let
\begin{equation}
w^{\nu}(x):=\left\{
\begin{array}
[c]{ll}%
\phantom{-}1 & \text{if }x\cdot\nu>0~,\\
-1 & \text{if }x\cdot\nu<0~.
\end{array}
\right.  \label{u nu}%
\end{equation}
When $\nu=e_{n}$, the superscript $\nu$ is omitted. Let $\theta\in
C_{c}^{\infty}\left(  \mathbb{R}^{n}\right)  $ be such that
$\operatorname*{supp}\theta\subset B_{1}\left(  0\right)  $, $\int%
_{\mathbb{R}^{n}}\theta\left(  x\right)  \,dx=1$, and for every $\sigma>0$
define the mollifier%
\begin{equation}
\theta_{\sigma}\left(  x\right)  :=\frac{1}{\sigma^{n}}\theta\left(  \frac
{x}{\sigma}\right)  ,\quad x\in\mathbb{R}^{n}~. \label{standard mollifier}%
\end{equation}
Note that $\operatorname*{supp}\theta_{\sigma}\subset B_{\sigma}\left(
0\right)  $. There exists a constant $C_{\theta}>1$, independent of $\sigma$,
such that%
\begin{align}
\vphantom{\frac{C_{\psi}}{\varepsilon^{2}}}  &  \sup_{\mathbb{R}^{n}%
}\left\vert \left(  w^{\nu}\!\ast\theta_{\sigma}\right)  -w^{\nu}\right\vert
\leq1~,\label{F75}\\
\vphantom{\frac{C_{\psi}}{\varepsilon^{2}}}  &  \left(  w^{\nu}\!\ast
\theta_{\sigma}\right)  (x)=1\quad\text{if }x\cdot\nu>\sigma,\quad\left(
w^{\nu}\!\ast\theta_{\sigma}\right)  (x)=-1\quad\text{if }x\cdot\nu
<-\sigma~,\label{admonbou}\\
\vphantom{\frac{C_{\psi}}{\varepsilon^{2}}}  &  \nabla\!\left(  w^{\nu}%
\!\ast\theta_{\sigma}\right)  (x)=0\quad\text{if }\left\vert x\cdot
\nu\right\vert >\sigma~,\label{mollifier support}\\
&  \sup_{\mathbb{R}^{n}}|\nabla\!\left(  w^{\nu}\!\ast\theta_{\sigma}\right)
|\leq\frac{C_{\theta}}{\sigma}\quad\text{and}\quad\sup_{\mathbb{R}^{n}}%
|\nabla^{2}\!\left(  w^{\nu}\!\ast\theta_{\sigma}\right)  |\leq\frac
{C_{\theta}}{\sigma^{2}}~. \label{derivatives convolution}%
\end{align}

Let $P$ be a bounded polyhedron of dimension $n-1$  containing $0$ 
 and let $\nu\in
\mathbb{S}^{n-1}$ be a normal to $P$. For every $\rho>0$ we set%
\begin{equation}
P_{\rho}:=\{x+t\nu:~x\in P~,~t\in(-\rho/2,\rho/2)\}~. \label{sigma rho}%
\end{equation}

\begin{theorem}
[Modification Theorem]\label{modification} Let $P$ be a bounded polyhedron of dimension $n-1$  containing $0$, 
let $\rho>0$, let $\varepsilon_{j}\rightarrow0^{+}$, and
let $\{u_{j}\}$ be a sequence in $W_{\operatorname*{loc}}^{1,2}(P_{\rho})\cap
L^{2}(P_{\rho})$ such that $u_{j}\rightarrow w^{\nu}$ in $L^{2}(P_{\rho})$.
Then there exists a constant $\delta_{P_{\rho}}>0$ depending only on $P_{\rho
}$ such that for every $0<\delta<\delta_{P_{\rho}}$ there exists a sequence
$\{v_{j}\}\subset W_{\operatorname*{loc}}^{1,2}(P_{\rho})\cap L^{2}(P_{\rho})$
such that $v_{j}\rightarrow w^{\nu}$ in $L^{2}(P_{\rho})$, $v_{j}=u_{j}$ in
$(P_{\rho})_{2\delta}$, $v_{j}=w^{\nu}\!\ast\theta_{\varepsilon_{j}}$ on
$P_{\rho}\setminus(P_{\rho})_{\delta}$, and%
\begin{equation}
\limsup_{j\rightarrow+\infty}\mathcal{F}_{\varepsilon_{j}}(v_{j},P_{\rho}%
)\leq\limsup_{j\rightarrow+\infty}\mathcal{F}_{\varepsilon_{j}}(u_{j},P_{\rho
})+\kappa_{1}\delta~, \label{G7}%
\end{equation}
where $\kappa_{1}>0$ is a constant independent of $j$, $\delta$, and $P_{\rho
}$.
\end{theorem}

\begin{remark}
\label{G8} By choosing a suitable subsequence, under the same assumptions of
Theorem \ref{modification} we obtain that
\begin{equation}
\liminf_{j\rightarrow+\infty}\mathcal{F}_{\varepsilon_{j}}(v_{j},P_{\rho}%
)\leq\liminf_{j\rightarrow+\infty}\mathcal{F}_{\varepsilon_{j}}(u_{j},P_{\rho
})+\kappa_{1}\delta~. \label{G9}%
\end{equation}

\end{remark}

To prove Theorem \ref{modification} we use the estimate of the following lemma.

\begin{lemma}
\label{lemma g}Let $\varepsilon>0$, let $y\in\mathbb{R}^{n}$, let $A$ be a
measurable subset of $\mathbb{R}^{n}$, and let $g:A\rightarrow\mathbb{R}$ be a
measurable function such that
\begin{equation}
0\leq g(x)\leq(a|x-y|)^{2}\wedge b^{2}\quad\text{for every }x\in A~,
\label{E1}%
\end{equation}
for some constants $a$ and $b$. Then
\begin{equation}
\int_{A}J_{\varepsilon}(x-y)g(x)~dx\leq M_{J}\big((\varepsilon a)\vee
b\big)^{2}~, \label{E2}%
\end{equation}
where $M_{J}$ is the constant given in (\ref{J1}) and $\alpha\vee\beta
:=\max\{\alpha,\beta\}$.
\end{lemma}

\begin{proof}
Using (\ref{J4}) and the change of variables $z=(x-y)/\varepsilon$, we
obtain
\begin{align}
\int_{A}J_{\varepsilon}(x-y)g(x)~dx  &  \leq a^{2}\int_{A\cap B_{\varepsilon
}(y)}J_{\varepsilon}(x-y)|x-y|^{2}~dx\nonumber\\
&  \quad+b^{2}\int_{A\setminus B_{\varepsilon}(y)}J_{\varepsilon}%
(x-y)\frac{|x-y|}{\varepsilon}~dx\nonumber\\
&  \leq\varepsilon^{2}a^{2}\int_{B_{1}(0)}J(z)|z|^{2}~dz+b^{2}\int%
_{\mathbb{R}^{n}\setminus B_{1}(0)}J(z)|z|~dz~.\nonumber
\end{align}
The conclusion follows from (\ref{J1}).
\end{proof}

\begin{lemma}
\label{lemma separated}Let $0<\varepsilon<\delta$, let $A$ and $B$ be open
sets in $\mathbb{R}^{n}$, with $\operatorname*{dist}(A,B)\geq\delta$, and let
$u\in W_{\operatorname*{loc}}^{1,2}(A\cup B)$. Then
\begin{equation}
\mathcal{J}_{\varepsilon}(u,A,B)\leq\varepsilon\omega_{1}\Bigl(\frac
{\varepsilon}{\delta}\Bigr)\int_{A\cup B}|\nabla u(x)|^{2}dx~,
\label{estimate separated}%
\end{equation}
where
\begin{equation}
\omega_{1}(t):=2\int_{\mathbb{R}^{n}\setminus B_{1/t}(0)}%
J(z)|z|~dz\rightarrow0 \label{little omega}%
\end{equation}
as $t\rightarrow0^{+}$.
\end{lemma}

\begin{proof}
Using a change of variables we obtain
\begin{align*}
\mathcal{J}_{\varepsilon}  &  (u,A,B)=\varepsilon\int_{A}\int_{B}%
J_{\varepsilon}(x-y)|\nabla u(x)-\nabla u(y)|^{2}dxdy\\
&  \leq2\varepsilon\int_{B}\Big(\int_{A}J_{\varepsilon}(x-y)~dy\Big)|\nabla
u(x)|^{2}dx\\
&  \ +2\varepsilon\int_{A}\Big(\int_{B}J_{\varepsilon}(x-y)~dx\Big)|\nabla
u(y)|^{2}dy\\
&  \leq2\varepsilon\int_{B}\Big(\int_{\mathbb{R}^{n}\setminus B_{\delta}%
(x)}J_{\varepsilon}(x-y)~dy\Big)|\nabla u(x)|^{2}dx\\
&  \ +2\varepsilon\int_{A}\Big(\int_{\mathbb{R}^{n}\setminus B_{\delta}%
(y)}J_{\varepsilon}(x-y)~dx\Big)|\nabla u(y)|^{2}dy\\
\end{align*}
\begin{align*}
&  \leq2\varepsilon\int_{\mathbb{R}^{n}\setminus B_{\frac{\delta}{\varepsilon
}}(0)}J(z)~dz\int_{A\cup B}|\nabla u(x)|^{2}dx\\
&  \leq2\varepsilon\int_{\mathbb{R}^{n}\setminus B_{\frac{\delta}{\varepsilon
}}(0)}J(z)|z|~dz\int_{A\cup B}|\nabla u(x)|^{2}dx~.
\end{align*}
This leads to (\ref{estimate separated}). The fact that $\omega_{1}%
(t)\rightarrow0^{+}$ as $t\rightarrow0^{+}$ follows from (\ref{J1}).
\end{proof}

\medskip
\begin{proof}
[Proof of Theorem \ref{modification}]It is not restrictive to assume that
$\delta<\frac{1}{4}$, $\varepsilon_{j}<\delta^{2}$, and $8\varepsilon
_{j}\gamma_{J}<\delta$ for every $j$. To simplify the notation, set
$\widetilde{u}_{j}:=w^{\nu}\!\ast\theta_{\varepsilon_{j}}$. From
(\ref{mollifier support}) and (\ref{derivatives convolution}) it follows that%
\begin{equation}
\varepsilon_{j}\int_{P_{\rho}}|\nabla\widetilde{u}_{j}(x)|^{2}dx\leq
C_{\theta,P}\quad\text{for every }j~, \label{G1}%
\end{equation}
for some constant $C_{\theta,P}>0$ depending only on $P$ and $\theta$.

If the right-hand side of (\ref{G7}) is infinite, then there is nothing to
prove. Thus, by extracting a subsequence (not relabeled), without loss of
generality we may assume that
\begin{equation}
\mathcal{F}_{\varepsilon_{j}}(u_{j},P_{\rho})\leq M<+\infty\quad\text{for
every }j~, \label{exilim}%
\end{equation}
for a suitable constant $M>0$.

The functions $v_{j}$ will be constructed as
\begin{equation}
v_{j}:=\varphi_{j}u_{j}+(1-\varphi_{j})\widetilde{u}_{j}~, \label{vj}%
\end{equation}
where $\varphi_{j}\in C_{c}^{\infty}(\mathbb{R}^{n})$ are suitable cut-off
functions satisfying $\varphi_{j}(x)=1$ for $x\in(P_{\rho})_{\delta}$ and
$\varphi_{j}(x)=0$ for $x\notin(P_{\rho})_{\delta/2}$. Introduce the set
\begin{equation}
S:=\Bigl\{x\in P_{\rho}:~\,\frac{\delta}{2}<\operatorname*{dist}%
\big(x,\partial P_{\rho}\big)\leq\delta\Bigr\}\,. \label{F70}%
\end{equation}
To construct the cut-off functions we divide $S$ into $m_{j}$ pairwise
disjoint layers of width $\frac{\delta}{2m_{j}}$.

Consider the sequence $\{\eta_{j}\}$ defined by
\begin{equation}
\eta_{j}:=\int_{P_{\rho}}(u_{j}(x)-\widetilde{u}_{j}(x))^{2}dx+\int_{P_{\rho}%
}\int_{P_{\rho}\setminus B_{\varepsilon_{j}}(y)}J_{\varepsilon_{j}}%
(x-y)(u_{j}(x)-\widetilde{u}_{j}(x))^{2}dxdy~. \label{F18}%
\end{equation}
By Fubini's theorem, a change of variables, (\ref{J1}), and (\ref{F18}), we
obtain%
\begin{align*}
&  \int_{P_{\rho}}\int_{P_{\rho}\setminus B_{\varepsilon_{j}}(y)}%
J_{\varepsilon_{j}}(x-y)(u_{j}(x)-\widetilde{u}_{j}(x))^{2}dxdy\\
&  =\int_{P_{\rho}}\biggl(\int_{P_{\rho}\setminus B_{\varepsilon_{j}}%
(x)}J_{\varepsilon_{j}}(x-y)~dy\biggr)(u_{j}(x)-\widetilde{u}_{j}(x))^{2}dx\\
&  \leq\int_{P_{\rho}}(u_{j}(x)-\widetilde{u}_{j}(x))^{2}dx\int_{\mathbb{R}%
^{n}\setminus B_{1}(0)}J(z)~dz\leq M_{J}\int_{P_{\rho}}(u_{j}(x)-\widetilde{u}%
_{j}(x))^{2}dx~.
\end{align*}
Hence, $\eta_{j}\rightarrow0^{+}$ as $j\rightarrow+\infty$, because
$\{u_{j}\}$ and $\{\widetilde{u}_{j}\}$ converge to $w^{\nu}$ in
$L^{2}(P_{\rho})$. Without loss of generality, we assume that $\eta_{j}%
<\frac{1}{4}$ for every $j$. Let $m_{j}$ be the unique integer such that
\begin{equation}
\frac{\sqrt{\varepsilon_{j}}+\sqrt{\eta_{j}}}{\varepsilon_{j}}<m_{j}\leq
\frac{\sqrt{\varepsilon_{j}}+\sqrt{\eta_{j}}}{\varepsilon_{j}}+1~. \label{F79}%
\end{equation}
Since $\varepsilon_{j}<1$ we have
\begin{equation}
\frac{1}{m_{j}}<\sqrt{\varepsilon_{j}}\quad\text{and}\quad m_{j}<2\frac
{\sqrt{\varepsilon_{j}}+\sqrt{\eta_{j}}}{\varepsilon_{j}} \label{F80}%
\end{equation}
and
\begin{equation}
\frac{\eta_{j}}{m_{j}\varepsilon_{j}}\leq\sqrt{\varepsilon_{j}}+\sqrt{\eta
_{j}}\quad\text{and}\quad m_{j}\varepsilon_{j}\leq2(\sqrt{\varepsilon_{j}%
}+\sqrt{\eta_{j}})~. \label{F65}%
\end{equation}


Divide $S$ into $m_{j}$ pairwise disjoint layers of width $\frac{\delta
}{2m_{j}}$\thinspace,%
\begin{equation}
S_{j}^{i}:=\bigg\{x\in P_{\rho}:~\,\frac{\delta}{2}+\frac{(i-1)\delta}{2m_{j}%
}<\operatorname*{dist}\big(x,\partial P_{\rho}\big)<\frac{\delta}{2}%
+\frac{i\delta}{2m_{j}}\bigg\}~, \label{strip S 1 j}%
\end{equation}
$i=1,\dots,m_{j}$.

For every open set $A\subset\mathbb{R}^{d}$ define
\begin{align}
\mathcal{G}_{j}(A)  &  :=\mathcal{J}_{\varepsilon_{j}}(u_{j},A,P_{\rho
})+\mathcal{W}_{\varepsilon_{j}}(u_{j},A)\nonumber\\
&\quad+\,\varepsilon_{j}\int_{A}|\nabla
u_{j}(x)|^{2}dx\label{G j}+\frac{1}{\varepsilon_{j}}\int_{A}(u_{j}(x)-\widetilde{u}_{j}%
(x))^{2}dx\\
&\quad+\frac{1}{\varepsilon_{j}}\int_{A}\int_{P_{\rho}\setminus
B_{\varepsilon_{j}}(y)}J_{\varepsilon_{j}}(x-y)(u_{j}(x)-\widetilde{u}%
_{j}(x))^{2}dxdy~.\nonumber
\end{align}
Hence, using (\ref{exilim}), (\ref{F18}), and Lemma \ref{F53}, we obtain%
\[
\sum_{i=1}^{m_{j}}\mathcal{G}_{j}(S_{j}^{i})\leq\mathcal{G}_{j}(S)\leq K-1+\frac{\eta_{j}}{\varepsilon_{j}}~,
\]
where $K:=M+c_{J,W}^{(n)}M+1$, and so there exists $i_{j}\in\{1,\dots,m_{j}\}$
such that, setting%
\[
S_{j}:=S_{j}^{i_{j}}~,
\]
we have
\begin{equation}
\mathcal{G}_{j}(S_{j})\leq\frac{K-1}{m_{j}}+\frac{\eta_{j}}{m_{j}%
\varepsilon_{j}}\leq K\sqrt{\varepsilon_{j}}+\sqrt{\eta_{j}}\leq K~,
\label{Ennio}%
\end{equation}
where in the last inequalities we used (\ref{F80}), (\ref{F65}), and the fact
that $\varepsilon_{j}<\frac{1}{4}$, $\eta_{j}<\frac{1}{4}$, and $K\geq 1$.
Define%
\begin{align}
A_{j}  &  :=\biggl\{x\in P_{\rho}:~\operatorname*{dist}(x,\partial P_{\rho
})>\frac{\delta}{2}+\frac{i_{j}\delta}{2m_{j}}\biggr\}~,\nonumber\\
A_{j}^{\ast}  &  :=\biggl\{x\in P_{\rho}:~\operatorname*{dist}(x,\partial
P_{\rho})>\frac{\delta}{2}+\frac{i_{j}\delta}{2m_{j}}-\frac{\delta}{4m_{j}%
}\biggr\}~,\label{sets}\\
B_{j}  &  :=\biggl\{x\in P_{\rho}:~\operatorname*{dist}(x,\partial P_{\rho
})<\frac{\delta}{2}+\frac{(i_{j}-1)\delta}{2m_{j}}\biggr\}~,\nonumber
\end{align}
and let
\[
\varphi_{j}(x):=\int_{A_{j}^{\ast}}\theta_{\!\!\!_{\frac{\delta}{4m_{j}}}%
}\!\!(x-y)\,dy~.
\]
Then $\varphi_{j}\in C_{c}^{\infty}(\mathbb{R}^{n})$ and the following
properties hold, thanks to (\ref{derivatives convolution}) and (\ref{F80}):
\begin{align}
&  \varphi_{j}=1\text{ in }A_{j},\quad0\leq\varphi_{j}\leq1\text{ in }%
S_{j},\quad\varphi_{j}=0\text{ in }B_{j}~,\label{cut-offc}\\
&  \sup|\nabla\varphi_{j}|\leq8\frac{C_{\theta}}{\delta}\frac{\sqrt
{\varepsilon_{j}}+\sqrt{\eta_{j}}}{\varepsilon_{j}}\leq\frac{8C_{\theta}%
}{\delta\varepsilon_{j}}~,\ \sup|\nabla^{2}\varphi_{j}|\leq2^{7}%
\frac{C_{\theta}}{\delta^{2}}\frac{\varepsilon_{j}+\eta_{j}}{\varepsilon
_{j}^{2}}~, \label{deriphi}%
\end{align}
where $C_{\theta}$ is the constant given in (\ref{derivatives convolution}).

Let $v_{j}$ be the function defined by (\ref{vj}). Since $(P_{\rho})_{\delta
}\subset A_{j}$ and $P_{\rho}\setminus(P_{\rho})_{\delta/2}\subset B_{j}$, we
have that $v_{j}=u_{j}$ in $(P_{\rho})_{\delta}$ and $v_{j}=\widetilde{u}_{j}$
on $P_{\rho}\setminus(P_{\rho})_{\delta/2}$. Moreover, since $u_{j}$ and
$\widetilde{u}_{j}$ converge to $w^{\nu}$ in $L^{2}(P_{\rho})$, we have that
$v_{j}\rightarrow w^{\nu}$ in $L^{2}(P_{\rho})$. Note that
\begin{equation}
\nabla v_{j}:=\varphi_{j}\nabla u_{j}+(1-\varphi_{j})\nabla\widetilde{u}%
_{j}+(u_{j}-\widetilde{u}_{j})\nabla\varphi_{j}~. \label{E10}%
\end{equation}
Fix $0<\eta<\frac{1}{2}$. Using the inequality $|a+b|^{2}\leq\frac{|a|^{2}%
}{1-\eta}+\frac{|b|^{2}}{\eta}$, we obtain
\begin{align}
|\nabla v_{j}(x)-{}  &  \nabla v_{j}(y)|^{2}\leq\frac{1}{1-\eta}%
\big\vert\varphi_{j}(x)\nabla u_{j}(x)-\varphi_{j}(y)\nabla u_{j}%
(y)\nonumber\\
&  \qquad+(1-\varphi_{j}(x))\nabla\widetilde{u}_{j}(x)-(1-\varphi
_{j}(y))\nabla\widetilde{u}_{j}(y)\big\vert^{2}\label{E11}\\
&  +\frac{1}{\eta}\big\vert(u_{j}(x)-\widetilde{u}_{j}(x))\nabla\varphi
_{j}(x)-(u_{j}(y)-\widetilde{u}_{j}(y))\nabla\varphi_{j}(y)\big\vert^{2}%
~.\nonumber
\end{align}
In view of the same inequality and the convexity of $|\cdot|^{2}$, we get
\begin{align}
\big\vert\varphi_{j}(x)  &  \nabla u_{j}(x)-\varphi_{j}(y)\nabla
u_{j}(y)+(1-\varphi_{j}(x))\nabla\widetilde{u}_{j}(x)-(1-\varphi_{j}%
(y))\nabla\widetilde{u}_{j}(y)\big\vert^{2}\nonumber\\
&  =\big\vert\varphi_{j}(x)(\nabla u_{j}(x)-\nabla u_{j}(y))+(\varphi
_{j}(x)-\varphi_{j}(y))\nabla u_{j}(y)\nonumber\\
&  \quad+(1-\varphi_{j}(x))(\nabla\widetilde{u}_{j}(x)-\nabla\widetilde{u}%
_{j}(y))-(\varphi_{j}(x)-\varphi_{j}(y))\nabla\widetilde{u}_{j}%
(y)\big\vert^{2}\nonumber\\
&  \leq\frac{1}{1-\eta}\big\vert\varphi_{j}(x)(\nabla u_{j}(x)-\nabla
u_{j}(y))+(1-\varphi_{j}(x))(\nabla\widetilde{u}_{j}(x)-\nabla\widetilde{u}%
_{j}(y))\big\vert^{2}\nonumber\\
&  \quad+\frac{1}{\eta}\big\vert(\varphi_{j}(x)-\varphi_{j}(y))(\nabla
u_{j}(y)-\nabla\widetilde{u}_{j}(y))\big\vert^{2}\nonumber\\
&  \leq\frac{\varphi_{j}(x)}{1-\eta}\big\vert\nabla u_{j}(x)-\nabla
u_{j}(y)\big\vert^{2}+\frac{1-\varphi_{j}(x)}{1-\eta}\big\vert\nabla
\widetilde{u}_{j}(x)-\nabla\widetilde{u}_{j}(y)\big\vert^{2}\nonumber\\
&  \quad+\frac{1}{\eta}(\varphi_{j}(x)-\varphi_{j}(y))^{2}\big\vert\nabla
u_{j}(y)-\nabla\widetilde{u}_{j}(y)\big\vert^{2}~.\nonumber
\end{align}
This inequality and (\ref{E11}) yield
\begin{align}
|\nabla v_{j}(x)-\nabla v_{j}(y)|^{2}  &  {}\leq\frac{\varphi_{j}(x)}%
{(1-\eta)^{2}}\big\vert\nabla u_{j}(x)-\nabla u_{j}(y)\big\vert^{2}\nonumber\\
&  \quad+\frac{1-\varphi_{j}(x)}{(1-\eta)^{2}}\big\vert\nabla\widetilde{u}%
_{j}(x)-\nabla\widetilde{u}_{j}(y)\big\vert^{2}\nonumber\\
&  \quad+\frac{2}{\eta}(\varphi_{j}(x)-\varphi_{j}(y))^{2}\big\vert\nabla
u_{j}(y)-\nabla\widetilde{u}_{j}(y)\big\vert^{2}\nonumber\\
&  \quad+\frac{1}{\eta}\big\vert(u_{j}(x)-\widetilde{u}_{j}(x))\nabla
\varphi_{j}(x)-(u_{j}(y)-\widetilde{u}_{j}(y))\nabla\varphi_{j}%
(y)\big\vert^{2}~,\nonumber
\end{align}
hence for every pair of open sets $A$, $B\subset P_{\rho}$ we obtain by (\ref{J and W local})\begin{align}
&  \mathcal{J}_{\varepsilon_{j}}(v_{j},A,B)\leq\frac{\mathcal{J}%
_{\varepsilon_{j}}(u_{j},A,B\cap(A_{j}\cup S_{j}))}{(1-\eta)^{2}}%
+\frac{\mathcal{J}_{\varepsilon_{j}}(\widetilde{u}_{j},A,B\cap(S_{j}\cup
B_{j}))}{(1-\eta)^{2}}\nonumber\\
&  +\frac{2\varepsilon_{j}}{\eta}\int_{A}\Big(\int_{B}J_{\varepsilon_{j}%
}(x-y)(\varphi_{j}(x)-\varphi_{j}(y))^{2}dx\Big)|\nabla u_{j}(y)-\nabla
\widetilde{u}_{j}(y)|^{2}dy\label{E15}\\
&  +\frac{\varepsilon_{j}}{\eta}\int_{A}\Big(\int_{B}\!J_{\varepsilon_{j}%
}(x\!-\!y)\big\vert(u_{j}(x)\!-\!\widetilde{u}_{j}(x))\nabla\varphi
_{j}(x)-(u_{j}(y)\!-\!\widetilde{u}_{j}(y))\nabla\varphi_{j}(y)\big\vert^{2}%
dxdy.\nonumber
\end{align}

By (\ref{J add}) we have
\begin{align}
\mathcal{J}_{\varepsilon_{j}}(v_{j},P_{\rho})={}  &  \mathcal{J}%
_{\varepsilon_{j}}(u_{j},A_{j})+\mathcal{J}_{\varepsilon_{j}}(v_{j}%
,S_{j})+\mathcal{J}_{\varepsilon_{j}}(\widetilde{u}_{j},B_{j})\nonumber\\
&  +2\mathcal{J}_{\varepsilon_{j}}(v_{j},S_{j},A_{j}\cup B_{j})+2\mathcal{J}%
_{\varepsilon_{j}}(v_{j},A_{j},B_{j})~. \label{E9}%
\end{align}
We now estimate all the terms but the first on the right-hand side of
(\ref{E9}).

By (\ref{E15}),
\begin{align}
&  \mathcal{J}_{\varepsilon_{j}}(v_{j},S_{j})\leq\frac{\mathcal{J}%
_{\varepsilon_{j}}(u_{j},S_{j})}{(1-\eta)^{2}}+\frac{\mathcal{J}%
_{\varepsilon_{j}}(\widetilde{u}_{j},S_{j})}{(1-\eta)^{2}}\label{E16}\\
&  +\frac{2\varepsilon_{j}}{\eta}\int_{S_{j}}\Big(\int_{S_{j}}J_{\varepsilon
_{j}}(x-y)(\varphi_{j}(x)-\varphi_{j}(y))^{2}dx\Big)|\nabla u_{j}%
(y)-\nabla\widetilde{u}_{j}(y)|^{2}dy\nonumber\\
&  +\frac{\varepsilon_{j}}{\eta}\int_{S_{j}}\Big(\int_{S_{j}}\!J_{\varepsilon
_{j}}(x\!-\!y)\big\vert(u_{j}(x)\!-\!\widetilde{u}_{j}(x))\nabla\varphi
_{j}(x)-(u_{j}(y)\!-\!\widetilde{u}_{j}(y))\nabla\varphi_{j}(y)\big\vert^{2}%
dxdy.\nonumber
\end{align}
From (\ref{J add}) and (\ref{mollifier support}) it follows that
\begin{align}
\mathcal{J}_{\varepsilon_{j}}(  &  \widetilde{u}_{j},S_{j}\cup B_{j}%
)=\mathcal{J}_{\varepsilon_{j}}(\widetilde{u}_{j},(S_{j}\cup B_{j})\cap
P_{2\varepsilon_{j}})\nonumber\\
&  +2\mathcal{J}_{\varepsilon_{j}}(\widetilde{u}_{j},(S_{j}\cup B_{j})\cap
P_{2\varepsilon_{j}},(S_{j}\cup B_{j})\setminus P_{2\varepsilon_{j}})~.
\label{E4}%
\end{align}
By the mean value theorem and by (\ref{derivatives convolution}), for every
$y\in P_{\rho}$ the function $g(x):=|\nabla\widetilde{u}_{j}(x)-\nabla
\widetilde{u}_{j}(y)|^{2}$ satisfies (\ref{E1}) with $a=\frac{C_{\theta}%
}{\varepsilon_{j}^{2}}$ and $b=\frac{2C_{\theta}}{\varepsilon_{j}}$, hence by
Lemma~\ref{E1} we obtain
\[
\int_{P_{\rho}}J_{\varepsilon_{j}}(x-y)|\nabla\widetilde{u}_{j}
(x)-\nabla\widetilde{u}_{j}(y)|^{2}dx\leq 4C_{\theta}^{2}M_{J}\frac
{1}{\varepsilon_{j}^{2}}~.
\]
Therefore by (\ref{J and W local}) and (\ref{E4}) we have
\begin{align*}
\mathcal{J}_{\varepsilon_{j}}(\widetilde{u}_{j},S_{j},S_{j}\cup B_{j}%
)+\mathcal{J}_{\varepsilon_{j}}(\widetilde{u}_{j},B_{j})  &  \leq
\mathcal{J}_{\varepsilon_{j}}(\widetilde{u}_{j},S_{j}\cup B_{j})\\
\leq\mathcal{L}^{n}((S_{j}\cup B_{j})\cap P_{2\varepsilon_{j}})\,4C_{\theta}^{2}M_{J}\frac{1}{\varepsilon_{j}}  &  .
\end{align*}
We now use the fact that there exist two constants $C_{P_{\rho}}>0$ and
$\delta_{P_{\rho}}>0$, depending only on $P_{\rho}$, such that
\begin{equation}
\mathcal{L}^{n}(((P_{\rho})_{\delta_{1}}\setminus(P_{\rho})_{\delta_{2}})\cap
P_{\varepsilon})\leq C_{P_{\rho}}\varepsilon(\delta_{2}-\delta_{1})
\label{E4a}%
\end{equation}
for every $0<\varepsilon<\delta_{1}<\delta_{2}<\delta_{P_{\rho}}$. Therefore
\begin{equation}
\mathcal{J}_{\varepsilon_{j}}(\widetilde{u}_{j},S_{j},S_{j}\cup B_{j}%
)+\mathcal{J}_{\varepsilon_{j}}(\widetilde{u}_{j},B_{j})\leq 4C_{P_{\rho}%
}C_{\theta}^{2}M_{J}\delta~. \label{E5}%
\end{equation}

By the mean value theorem, (\ref{F80}), and (\ref{deriphi}), for every $y\in
S_{j}$ the function $g(x)=(\varphi_{j}(x)-\varphi_{j}(y))^{2}$ satisfies
(\ref{E1}) with $a=\frac{8C_{\theta}}{\delta\varepsilon_{j}}$ and
$b=1\leq\frac{8C_{\theta}}{\delta}$, where we used the inequalities
$C_{\theta}\geq1$ and $\delta\leq1$. Hence, by Lemma \ref{lemma g} we have%
\begin{equation}
\int_{P_{\rho}}J_{\varepsilon_{j}}(x-y)(\varphi_{j}(x)-\varphi_{j}%
(y))^{2}dx\leq2^{6}\frac{C_{\theta}^{2}}{\delta^{2}}M_{J}~.\nonumber
\end{equation}
In turn, by (\ref{mollifier support}), (\ref{derivatives convolution}),
(\ref{G j}), and (\ref{Ennio}),%
\begin{align}
\frac{2\varepsilon_{j}}{\eta}  &  \int_{S_{j}}\Big(\int_{P_{\rho}%
}J_{\varepsilon_{j}}(x-y)(\varphi_{j}(x)-\varphi_{j}(y))^{2}dx\Big)|\nabla
u_{j}(y)-\nabla\widetilde{u}_{j}(y)|^{2}dy\nonumber\\
&  \leq2^{8}\frac{C_{\theta}^{2}M_{J}}{\eta\delta^{2}}\varepsilon_{j}%
\int_{S_{j}}|\nabla u_{j}(y)|^{2}dy+2^{8}\frac{C_{\theta}^{4}M_{J}}{\eta
\delta^{2}}\frac{1}{\varepsilon_{j}}\mathcal{L}^{n}(S_{j}\cap P_{2\varepsilon
_{j}})\label{E18}\\
&  \leq2^{8}\frac{C_{\theta}^{2}M_{J}}{\eta\delta^{2}}\big(K\sqrt
{\varepsilon_{j}}+\sqrt{\eta_{j}}\big)+2^{8}C_{P_{\rho}}\frac{C_{\theta}%
^{4}M_{J}}{\eta\delta}\sqrt{\varepsilon_{j}}~,\nonumber
\end{align}
where in the last inequality we used the estimate
\begin{equation}
\mathcal{L}^{n}(S_{j}\cap P_{\varepsilon_{j}})\leq C_{P_{\rho}}\delta
\frac{\varepsilon_{j}}{m_{j}}\leq C_{P_{\rho}}\delta\varepsilon_{j}%
\sqrt{\varepsilon_{j}}~, \label{F77}%
\end{equation}
which follows fron (\ref{F80}) and (\ref{E4a}).

To treat the last term on the right-hand side of (\ref{E16}) we observe that
\begin{align}
\big\vert(u_{j}(x)-{}  &  \widetilde{u}_{j}(x))\nabla\varphi_{j}%
(x)-(u_{j}(y)-\widetilde{u}_{j}(y))\nabla\varphi_{j}(y)\big\vert^{2}%
\nonumber\\
&  =\big\vert(u_{j}(x)-{}\widetilde{u}_{j}(x))(\nabla\varphi_{j}%
(x)-\nabla\varphi_{j}(y))+\nonumber\\
&  \ \quad+(u_{j}(x)-\widetilde{u}_{j}(x)-u_{j}(y)+\widetilde{u}_{j}%
(y))\nabla\varphi_{j}(y)\big\vert^{2}\nonumber\\
&  \leq2(u_{j}(x)-{}\widetilde{u}_{j}(x))^{2}\big\vert\nabla\varphi
_{j}(x)-\nabla\varphi_{j}(y)\big\vert^{2}\nonumber\\
&  \ +2(u_{j}(x)-\widetilde{u}_{j}(x)-u_{j}(y)+\widetilde{u}_{j}%
(y))^{2}\big\vert\nabla\varphi_{j}(y)\big\vert^{2}~.\nonumber
\end{align}
Integrating and using the symmetry of $J$, we obtain%
\begin{align}
&  \frac{\varepsilon_{j}}{\eta}\int_{S_{j}}\Big(\int_{S_{j}}\!J_{\varepsilon
_{j}}(x\!-\!y)\big\vert(u_{j}(x)\!-\!\widetilde{u}_{j}(x))\nabla\varphi
_{j}(x)-(u_{j}(y)\!-\!\widetilde{u}_{j}(y))\nabla\varphi_{j}(y)\big\vert^{2}%
dxdy\nonumber\\
&  \leq\frac{2\varepsilon_{j}}{\eta}\int_{S_{j}}\Big(\int_{S_{j}%
}J_{\varepsilon_{j}}(x-y)|\nabla\varphi_{j}(x)-\nabla\varphi_{j}%
(y)|^{2}dx\Big)(u_{j}(y)-{}\widetilde{u}_{j}(y))^{2}dy\label{E18a}\\
&  +\frac{2\varepsilon_{j}}{\eta}\int_{S_{j}}\Big(\int_{S_{j}}J_{\varepsilon
_{j}}(x-y)(u_{j}(x)-\widetilde{u}_{j}(x)-u_{j}(y)+\widetilde{u}_{j}%
(y))^{2}dx\Big)|\nabla\varphi_{j}(y)|^{2}dy~.\nonumber
\end{align}
By the mean value theorem and (\ref{deriphi}), for every $y\in S_{j}$ the
function $g(x)=|\nabla\varphi_{j}(x)-\nabla\varphi_{j}(y)|^{2}$ satisfies
(\ref{E1}) for every $x\in\mathbb{R}^{n}$, with $a=\frac{2^{7}C_{\theta}%
}{\delta^{2}}\frac{\varepsilon_{j}+\eta_{j}}{\varepsilon_{j}^{2}}\leq
\frac{2^{6}C_{\theta}}{\delta^{2}}\frac{\sqrt{\varepsilon_{j}}+\sqrt{\eta_{j}%
}}{\varepsilon_{j}^{2}}$ and $b=\frac{2^{4}C_{\theta}}{\delta}\frac{\sqrt
{\varepsilon_{j}}+\sqrt{\eta_{j}}}{\varepsilon_{j}}\leq\frac{2^{6}C_{\theta}%
}{\delta^{2}}\frac{\sqrt{\varepsilon_{j}}+\sqrt{\eta_{j}}}{\varepsilon_{j}}$,
where we used the inequalities $\delta\leq1$, $\varepsilon_{j}\leq\frac{1}{4}%
$, and $\eta_{j}\leq\frac{1}{4}$. Hence, by Lemma \ref{lemma g} we have%
\[
\int_{P_{\rho}}\!\!\!J_{\varepsilon_{j}}(x-y)|\nabla\varphi_{j}(x)-\nabla
\varphi_{j}(y)|^{2}dx\leq2^{13}\frac{C_{\theta}^{2}M_{J}}{\delta^{4}}%
\frac{\varepsilon_{j}+\eta_{j}}{\varepsilon_{j}^{2}}~.
\]
In turn, by (\ref{G j}) and (\ref{Ennio}),
\begin{align}
&  \frac{2\varepsilon_{j}}{\eta}\int_{S_{j}}\Big(\int_{P_{\rho}}%
J_{\varepsilon_{j}}(x-y)|\nabla\varphi_{j}(x)-\nabla\varphi_{j}(y)|^{2}%
dx\Big)(u_{j}(y)-{}\widetilde{u}_{j}(y))^{2}dy\nonumber\\
&  \leq2^{14}\frac{C_{\theta}^{2}M_{J}}{\eta\delta^{4}}(\varepsilon_{j}%
+\eta_{j})\frac{1}{\varepsilon_{j}}\int_{S_{j}}(u_{j}(y)-{}\widetilde{u}%
_{j}(y))^{2}dy\label{E19}\\
&  \leq2^{14}\frac{C_{\theta}^{2}M_{J}K}{\eta\delta^{4}}(\varepsilon_{j}%
+\eta_{j})~.\nonumber
\end{align}

Since $J$ is even, by Fubini's theorem, a change of variables, and
(\ref{deriphi}),%
\begin{align}
&  \frac{2\varepsilon_{j}}{\eta}\int_{S_{j}}\Big(\int_{P_{\rho}}%
\!\!J_{\varepsilon_{j}}(x-y)(u_{j}(x)-\widetilde{u}_{j}(x)-u_{j}%
(y)+\widetilde{u}_{j}(y))^{2}dx\Big)|\nabla\varphi_{j}(y)|^{2}dy\nonumber\\
&  \leq\frac{2^{8}C_{\theta}^{2}}{\eta\delta^{2}}\frac{\varepsilon_{j}%
+\eta_{j}}{\varepsilon_{j}}\int_{S_{j}}\Big(\int_{P_{\rho}\cap B_{\varepsilon
_{j}}\!(y)}\!\!\!\!\!\!\!\!\!\!\!\!\!\!\!\!\!\!\!\!J_{\varepsilon_{j}%
}(x-y)(u_{j}(x)-\widetilde{u}_{j}(x)-u_{j}(y)+\widetilde{u}_{j}(y))^{2}%
dx\Big)dy\nonumber\\
&  +\frac{2^{8}C_{\theta}^{2}}{\eta\delta^{2}}\frac{\varepsilon_{j}+\eta_{j}%
}{\varepsilon_{j}}\int_{S_{j}}\Big(\int_{P_{\rho}\setminus B_{\varepsilon_{j}%
}\!(y)}\!\!\!\!\!\!\!\!\!\!\!\!\!\!\!\!\!\!\!\!J_{\varepsilon_{j}}%
(x-y)(u_{j}(x)-\widetilde{u}_{j}(x)-u_{j}(y)+\widetilde{u}_{j}(y))^{2}%
dx\Big)dy\nonumber\\
&  \leq\frac{2^{8}C_{\theta}^{2}}{\eta\delta^{2}}\frac{\varepsilon_{j}%
+\eta_{j}}{\varepsilon_{j}}\!\int_{B_{\varepsilon_{j}}\!(0)}%
\!\!\!\!\!\!\!\!\!J_{\varepsilon_{j}}(z)\Big(\int_{S_{j}}(u_{j}%
(y+z)-\widetilde{u}_{j}(y+z)-u_{j}(y)+\widetilde{u}_{j}(y))^{2}%
dy\Big)dz\nonumber\\
&  +\frac{2^{9}C_{\theta}^{2}}{\eta\delta^{2}}\frac{\varepsilon_{j}+\eta_{j}%
}{\varepsilon_{j}}\int_{S_{j}}\Big(\int_{P_{\rho}\setminus B_{\varepsilon_{j}%
}\!(y)}\!\!\!\!\!\!\!\!\!\!\!\!\!\!\!\!\!\!J_{\varepsilon_{j}}(x-y)(u_{j}%
(x)-\widetilde{u}_{j}(x))^{2}dx\Big)dy\label{E20}\\
&  +\frac{2^{9}C_{\theta}^{2}}{\eta\delta^{2}}\frac{\varepsilon_{j}+\eta_{j}%
}{\varepsilon_{j}}\int_{S_{j}}\Big(\int_{P_{\rho}\setminus B_{\varepsilon_{j}%
}\!(y)}\!\!\!\!\!\!\!\!\!\!\!\!\!\!\!\!\!\!J_{\varepsilon_{j}}%
(x-y)dx\Big)(u_{j}(y)-\widetilde{u}_{j}(y))^{2}dy~.\nonumber
\end{align}
Since $\varepsilon_{j}<\delta/4$, by (\ref{F80}) and
(\ref{strip S 1 j}) for $y\in S_{j}$ and $|z|\leq\varepsilon_{j}$ the segment
joning $y$ and $y+z$ is contained in $(P_{\rho})_{\delta/4}$, and so by the mean value
theorem for $|z|\leq\varepsilon_{j}$,%
\[
\int_{S_{j}}(u_{j}(y+z)-\widetilde{u}_{j}(y+z)-u_{j}(y)+\widetilde{u}%
_{j}(y))^{2}dy\leq|z|^{2}\int_{(P_{\rho})_{\delta/4}}\!\!\!\!\!\!\!\!|\nabla u_{j}(y)-\nabla\widetilde{u}%
_{j}(y)|^{2}dy~.
\]
Therefore, recalling that $2\varepsilon_{j}\gamma_{J}<\delta/4$, it follows from (\ref{J4}), (\ref{J1}), (\ref{G1}), and Lemma \ref{F53}, that
\begin{align}
&  \frac{2^{8}C_{\theta}^{2}}{\eta\delta^{2}}\frac{\varepsilon_{j}%
\!+\!\eta_{j}}{\varepsilon_{j}}\int_{B_{\varepsilon_{j}}(0)}%
\!\!\!\!\!\!\!\!\!\!J_{\varepsilon_{j}}(z)\Big(\int_{S_{j}}(u_{j}%
(y\!+\!z)\!-\!\widetilde{u}_{j}(y\!+\!z)\!-\!u_{j}(y)\!+\!\widetilde{u}%
_{j}(y))^{2}dy\Big)dz\nonumber\\
&  \leq\frac{2^{8}C_{\theta}^{2}}{\eta\delta^{2}}\frac{\varepsilon_{j}%
+\eta_{j}}{\varepsilon_{j}}\int_{B_{\varepsilon_{j}}(0)}%
\!\!\!\!\!\!\!\!J_{\varepsilon_{j}}(z)|z|^{2}dz
\int_{(P_{\rho})_{\delta/4}}\!\!\!\!\!\!\!\!|\nabla
u_{j}(y)-\nabla\widetilde{u}_{j}(y)|^{2}dy\nonumber\\
&  \leq\frac{2^{9}C_{\theta}^{2}}{\eta\delta^{2}}(\varepsilon_{j}+\eta
_{j})\varepsilon_{j}\int_{B_{1}(0)}\!\!\!\!\!\!\!\!J(z)|z|^{2}dz
\int_{(P_{\rho})_{\delta/4}}\!\!\!\!\!\!\!\!|\nabla u_{j}(y)|^{2}dy\label{E21}\\
&  \quad+\frac{2^{9}C_{\theta}^{2}}{\eta\delta^{2}}(\varepsilon_{j}+\eta
_{j})\varepsilon_{j}\int_{B_{1}(0)}\!\!\!\!\!\!\!\!J(z)|z|^{2}dz
\int_{(P_{\rho})_{\delta/4}}\!\!\!\!\!\!\!\!|\nabla\widetilde{u}_{j}(y)|^{2}dy\nonumber\\
&  \leq\frac{2^{9}C_{\theta}^{2}M_{J}c_{J,W}^{(n)}M}{\eta\delta^{2}%
}(\varepsilon_{j}+\eta_{j})+\frac{2^{9}C_{\theta}^{2}C_{\theta,P}M_{J}}{\eta\delta^{2}%
}(\varepsilon_{j}+\eta_{j})~.\nonumber
\end{align}

By (\ref{G j}) and (\ref{Ennio})%
\begin{align}
&  \frac{2^{9}C_{\theta}^{2}}{\eta\delta^{2}}\frac{\varepsilon_{j}+\eta_{j}%
}{\varepsilon_{j}}\int_{S_{j}}\Big(\int_{P_{\rho}\setminus B_{\varepsilon_{j}%
}\!(y)}\!\!\!\!\!\!\!\!\!\!\!\!\!\!\!\!\!\!J_{\varepsilon_{j}}(x-y)(u_{j}%
(x)-\widetilde{u}_{j}(x))^{2}dx\Big)dy\label{E21a}\\
&  \leq\frac{2^{9}C_{\theta}^{2}K}{\eta\delta^{2}}(\varepsilon_{j}+\eta
_{j})~.\nonumber
\end{align}

Using (\ref{J1}), (\ref{G j}), and (\ref{Ennio}) we obtain%
\begin{align}
&  \frac{2^{9}C_{\theta}^{2}}{\eta\delta^{2}}\frac{\varepsilon_{j}+\eta_{j}%
}{\varepsilon_{j}}\int_{S_{j}}\Big(\int_{P_{\rho}\setminus B_{\varepsilon_{j}%
}(y)}J_{\varepsilon_{j}}(x-y)dx\Big)(u_{j}(y)-\widetilde{u}_{j}(y))^{2}%
dy\nonumber\\
&  \leq\frac{2^{9}C_{\theta}^{2}M_{J}}{\eta\delta^{2}}\frac{\varepsilon
_{j}+\eta_{j}}{\varepsilon_{j}}\int_{S_{j}}(u_{j}(y)-\widetilde{u}_{j}%
(y))^{2}dy\leq\frac{2^{9}C_{\theta}^{2}M_{J}K}{\eta\delta^{2}}(\varepsilon
_{j}+\eta_{j})~. \label{E22}%
\end{align}

Combining (\ref{E16}), (\ref{E5}), (\ref{E18}), (\ref{E18a}), (\ref{E19}),
(\ref{E20}), (\ref{E21}), (\ref{E21a}), and (\ref{E22}), we have
\begin{equation}
\mathcal{J}_{\varepsilon_{j}}(v_{j},S_{j})+\mathcal{J}_{\varepsilon_{j}%
}(\widetilde{u}_{j},B_{j})\leq\frac{\mathcal{J}_{\varepsilon_{j}}(u_{j}%
,S_{j})}{(1-\eta)^{2}}+\frac{4C_{P_{\rho}}C_{\theta}^{2}M_{J}}{(1-\eta)^{2}%
}\delta+\sigma_{j}^{(1)}~, \label{E23}%
\end{equation}
where $\sigma_{j}^{(1)}\rightarrow0^{+}$ as $j\rightarrow+\infty$.

Next we consider the term $\mathcal{J}_{\varepsilon_{j}}(v_{j},S_{j},A_{j}\cup
B_{j})$ in (\ref{E9}) . By (\ref{E15}), using (\ref{cut-offc}),%
\begin{align}
\mathcal{J}_{\varepsilon_{j}}  &  (v_{j},S_{j},A_{j}\cup B_{j})\leq
\frac{\mathcal{J}_{\varepsilon_{j}}(u_{j},S_{j},A_{j})}{(1-\eta)^{2}}%
+\frac{\mathcal{J}_{\varepsilon_{j}}(\widetilde{u}_{j},S_{j},B_{j})}%
{(1-\eta)^{2}}\nonumber\\
&  +\frac{2\varepsilon_{j}}{\eta}\int_{S_{j}}\Big(\int_{A_{j}\cup B_{j}%
}J_{\varepsilon_{j}}(x-y)(\varphi_{j}(x)-\varphi_{j}(y))^{2}dx\Big)|\nabla
u_{j}(y)-\nabla\widetilde{u}_{j}(y)|^{2}dy\nonumber\\
&  +\frac{\varepsilon_{j}}{\eta}\int_{S_{j}}\int_{A_{j}\cup B_{j}%
}\!J_{\varepsilon_{j}}(x\!-\!y)(u_{j}(y)\!-\!\widetilde{u}_{j}(y))^{2}%
|\nabla\varphi_{j}(y)|^{2}dxdy~. \label{E24}%
\end{align}
 Since $\eta<1/2$, by (\ref{G j}) and (\ref{Ennio}) we have
\begin{equation}
\frac{\mathcal{J}_{\varepsilon_{j}}(u_{j},S_{j},A_{j})}{(1-\eta)^{2}}\leq
4\mathcal{J}_{\varepsilon_{j}}(u_{j},S_{j},A_{j})\leq
4K\sqrt{\varepsilon_{j}}+ 4\sqrt{\eta_{j}}~.
\label{545bis}
\end{equation}
The second and third terms on the right-hand side of (\ref{E24}) can be
estimated using (\ref{E5}) and (\ref{E18}). For the last term, we use the fact
that $\nabla\varphi_{j}(x)=0$ if $x\in A_{j}\cup B_{j}$. Hence, by a change of
variables, from (\ref{J1}), (\ref{G j}), (\ref{Ennio}), (\ref{deriphi}) and from the inequalities $\delta\leq1$, $\varepsilon_{j}\leq1$, and $\eta_{j}%
\leq1$, we obtain%
\begin{align}
&  \frac{\varepsilon_{j}}{\eta}\int_{S_{j}}\int_{A_{j}\cup B_{j}%
}\!J_{\varepsilon_{j}}(x\!-\!y)(u_{j}(y)\!-\!\widetilde{u}_{j}(y))^{2}%
|\nabla\varphi_{j}(y)|^{2}dxdy\nonumber\\
&  \leq\frac{\varepsilon_{j}}{\eta}\int_{S_{j}}\int_{B_{\varepsilon_{j}}%
\!(y)}\!J_{\varepsilon_{j}}(x\!-\!y)(u_{j}(y)\!-\!\widetilde{u}_{j}%
(y))^{2}|\nabla\varphi_{j}(y)-\nabla\varphi_{j}(x)|^{2}dxdy\nonumber\\
&  +\frac{\varepsilon_{j}}{\eta}\int_{S_{j}}\int_{P_{\rho}\setminus
B_{\varepsilon_{j}}\!(y)}\!J_{\varepsilon_{j}}(x\!-\!y)(u_{j}%
(y)\!-\!\widetilde{u}_{j}(y))^{2}|\nabla\varphi_{j}(y)|^{2}dxdy\nonumber\\
&  \leq2^{14}\frac{C_{\theta}^{2}}{\eta\delta^{4}}\frac{(\varepsilon_{j}%
+\eta_{j})^{2}}{\varepsilon_{j}^{3}}\int_{S_{j}}\Big(\int_{B_{\varepsilon_{j}%
}\!(y)}\!J_{\varepsilon_{j}}(x\!-\!y)|x-y|^{2}dx\Big)(u_{j}%
(y)\!-\!\widetilde{u}_{j}(y))^{2}dy\nonumber\\
&  +\frac{2^{7}C_{\theta}^{2}}{\eta\delta^{2}}\frac{\varepsilon_{j}+\eta_{j}%
}{\varepsilon_{j}}\int_{S_{j}}\Big(\int_{P_{\rho}\setminus B_{\varepsilon_{j}%
}\!(y)}\!J_{\varepsilon_{j}}(x\!-\!y)~dx\Big)(u_{j}(y)\!-\!\widetilde{u}%
_{j}(y))^{2}dxdy\label{E25}\\
&  \leq2^{14}\frac{C_{\theta}^{2}M_{J}}{\eta\delta^{4}}\frac{\varepsilon
_{j}\!+\!\eta_{j}}{\varepsilon_{j}}\!\!\int_{S_{j}}\!\!(u_{j}%
(y)\!-\!\widetilde{u}_{j}(y))^{2}\!dy\leq2^{14}\frac{C_{\theta}^{2}M_{J}%
K}{\eta\delta^{4}}(\varepsilon_{j}\!+\!\eta_{j})~.\nonumber
\end{align}
Therefore, by (\ref{E5}), (\ref{E18}), (\ref{E24}), (\ref{545bis}), and (\ref{E25}) we get%
\begin{equation}
\mathcal{J}_{\varepsilon_{j}}(v_{j},S_{j},A_{j}\cup B_{j})\leq
\frac{4C_{P_{\rho}}C_{\theta}^{2}M_{J}}{(1-\eta)^{2}}\delta
+\sigma_{j}^{(2)}~,\!\!\! \label{E26}%
\end{equation}
where $\sigma_{j}^{(2)}\rightarrow0^{+}$ as $j\rightarrow+\infty$.

We now estimate the term $\mathcal{J}_{\varepsilon_{j}}(v_{j},A_{j},B_{j})$ in
(\ref{E9}). Since $v_{j}=u_{j}$ in $A_{j}$, $v_{j}=\widetilde{u}_{j}=1$ in
$B_{j}$, and $\text{dist}(A_{j},B_{j})=\frac{\delta}{2m_{j}}$, by a change of
variables and in view of (\ref{G1}), (\ref{F65}), and Lemmas
\ref{F53} and \ref{lemma separated}, for $j$ large enough we obtain
\begin{align}
\mathcal{J}_{\varepsilon_{j}}  &  (v_{j},A_{j},B_{j})
\leq2\omega_{1}\bigl(2\frac{m_{j}\varepsilon_{j}}{\delta}\bigr)\bigg(\varepsilon_{j}%
\int_{B_{j}}|\nabla\widetilde{u}_{j}(x)|^{2}dx+\varepsilon_{j}\int_{A_{j}%
}|\nabla u_{j}(y)|^{2}dy\bigg)\nonumber\\
&  \leq2\omega_{1}\Big(4\frac{\sqrt{\varepsilon_{j}}+ \sqrt{\eta_{j}}}{\delta}
\Big)(C_{\theta,P}\!+\!c_{J,W}^{(n)}M)~. \label{G2}%
\end{align}

Combining (\ref{E9}), (\ref{E5}), (\ref{E23}), (\ref{E26}), and (\ref{G2}) we
deduce
\begin{equation}
\mathcal{J}_{\varepsilon_{j}}(v_{j},P_{\rho})\leq\frac{\mathcal{J}%
_{\varepsilon_{j}}(u_{j},P_{\rho})}{(1-\eta)^{2}}+\frac{12 C_{P_{\rho}}C_{\theta}%
^{2}M_{J}}{(1-\eta)^{2}}\delta
+\sigma_{j}^{(3)}~, \label{G33}%
\end{equation}
where $\sigma_{j}^{(3)}\rightarrow0^{+}$ as $j\rightarrow+\infty$.

Next we consider the term $\mathcal{W}_{\varepsilon_{j}}(v_{j},P_{\rho})$. Fix
$x\in S_{j}$ with $x\cdot\nu>\varepsilon_{j}$, so that $\widetilde{u}%
_{j}(x)=1$. By (\ref{W3}) and (\ref{W3b}) we have $W(v_{j}(x))\leq
W(u_{j}(x))$ if $u_{j}(x)\geq1-a_{W}$. Let $s_{0}<-1$ be such that
\begin{equation}
W(s_{0})=\max_{[-1,1]}W=:M_{W}~. \label{G5}%
\end{equation}
If $u_{j}(x)\leq s_{0}$, then either $u_{j}(x)\leq v_{j}(x)\leq-1$ or $-1\leq
v_{j}(x)\leq1$. In both cases we get $W(v_{j}(x))\leq W(u_{j}(x))$, either by
(\ref{W3b}) or by (\ref{G5}). If $s_{0}<u_{j}(x)<1-a_{W}$, then $s_{0}%
<v_{j}(x)<1$ and we have
\[
W(v_{j}(x))\leq W(s_{0})=M_{W}%
\]
by (\ref{W3b}) and (\ref{G5}). We conclude that
\[
W(v_{j}(x))\leq W(u_{j}(x))+M_{W}%
\]
for every $x\in S_{j}$ with $x\cdot\nu>\varepsilon_{j}$. Integrating we
obtain
\begin{align}
\frac{1}{\varepsilon_{j}}  &  \int_{S_{j}\cap\{x\cdot\nu>\varepsilon_{j}%
\}}W(v_{j}(x))~dx\leq\frac{1}{\varepsilon_{j}}\int_{S_{j}\cap\{x\cdot
\nu>\sigma_{j}\}}W(u_{j}(x))~dx\nonumber\\
&  \ +\frac{M_{W}}{\varepsilon_{j}}\mathcal{L}^{n}(S_{j}\cap\{|u_{j}%
-1|>a_{W}\}\cap\{x\cdot\nu>\varepsilon_{j}\})\nonumber\\
&  {}\leq\frac{1}{\varepsilon_{j}}\int_{S_{j}\cap\{x\cdot\nu>\varepsilon
_{j}\}}W(u_{j}(x))~dx+\frac{M_{W}}{\varepsilon_{j}a_{W}^{2}}\int_{S_{j}%
\cap\{x\cdot\nu>\varepsilon_{j}\}}(u_{j}(x)-1)^{2}dx\nonumber
\end{align}
A similar inequality can be obtained for $S_{j}\cap\{x\cdot\nu<-\varepsilon
_{j}\}$, and adding these two inequalities we conclude that
\begin{align}
\frac{1}{\varepsilon_{j}}\int_{S_{j}\setminus P_{\varepsilon_{j}}}%
W(v_{j}(x))~dx  &  \leq\frac{1}{\varepsilon_{j}}\int_{S_{j}\setminus
P_{\varepsilon_{j}}}W(u_{j}(x))~dx\nonumber\\
&  \quad+\frac{M_{W}}{a_{W}^{2}}\frac{1}{\varepsilon_{j}}\int_{S_{j}\setminus
P_{\varepsilon_{j}}}(u_{j}(x)-\widetilde{u}_{j}(x))^{2}dx~, \label{F62}%
\end{align}
where in the last inequality we used the fact that $\widetilde{u}_{j}=w^{\nu}$
on $P_{\rho}\setminus P_{\varepsilon_{j}}$.

On the other hand, since $W(v_{j}(x))\leq W(u_{j}(x))+M_{W}$ for every $x\in
P_{\rho}$, integrating over $S_{j}\cap P_{\varepsilon_{j}}$ and using
(\ref{F77}), we obtain
\begin{align}
\frac{1}{\varepsilon_{j}}  &  \int_{S_{j}\cap P_{\varepsilon_{j}}}%
W(v_{j}(x))~dx\leq\frac{1}{\varepsilon_{j}}\int_{S_{j}\cap P_{\varepsilon_{j}%
}}W(u_{j}(x))~dx+\frac{M_{W}}{\varepsilon_{j}}\mathcal{L}^{n}(S_{j}\cap
P_{\varepsilon_{j}})\nonumber\\
&  {}\leq\frac{1}{\varepsilon_{j}}\int_{S_{j}\cap P_{\varepsilon_{j}}}%
W(u_{j}(x))~dx+C_{P_{\rho}}M_{W}\delta\sqrt{\varepsilon_{j}}~. \label{F63}%
\end{align}
Adding (\ref{F62}) and (\ref{F63}) gives
\begin{align}
\frac{1}{\varepsilon_{j}}\int_{S_{j}}  &  W(v_{j}(x))~dx\leq\frac
{1}{\varepsilon_{j}}\int_{S_{j}}W(u_{j}(x))~dx\nonumber\\
&  {}+\frac{M_{W}}{a_{W}^{2}}\frac{1}{\varepsilon_{j}}\int_{S_{j}}%
(u_{j}(x)-\widetilde{u}_{j}(x))^{2}dx+C_{P_{\rho}}M_{W}\delta\sqrt
{\varepsilon_{j}}~,\nonumber
\end{align}
hence by (\ref{G j}) and (\ref{Ennio}) we have
\begin{align}
\frac{1}{\varepsilon_{j}}\int_{S_{j}}  &  W(v_{j}(x))~dx\leq\frac
{1}{\varepsilon_{j}}\int_{S_{j}}W(u_{j}(x))~dx\nonumber\\
&  {}+\frac{M_{W}}{a_{W}^{2}}(K\sqrt{\varepsilon_{j}%
}+\sqrt{\eta_{j}})+C_{P_{\rho}}M_{W}\delta\sqrt{\varepsilon_{j}}~.
\label{F19}%
\end{align}

By (\ref{F75}), (\ref{admonbou}), (\ref{E4a}), and (\ref{G5}) we get
\begin{align}
\frac{1}{\varepsilon_{j}}\int_{B_{j}}  &  W(v_{j}(x))~dx=\frac{1}%
{\varepsilon_{j}}\int_{B_{j}}W(\widetilde{u}_{j}(x))~dx\nonumber\\
&  \leq\frac{M_{W}}{\varepsilon_{j}}\mathcal{L}^{n}(B_{j}\cap P_{\varepsilon
_{j}})\leq C_{P_{\rho}}M_{W}\delta~. \label{G6}%
\end{align}

From (\ref{F19}) and (\ref{G6}) it follows that
\begin{equation}
\frac{1}{\varepsilon_{j}}\int_{P_{\rho}}W(v_{j}(x))~dx\leq\frac{1}%
{\varepsilon_{j}}\int_{P_{\rho}}W(u(x))~dx+C_{P_{\rho}}M_{W}\delta+\sigma
_{j}^{(4)}~, \label{G12}%
\end{equation}
where $\sigma_{j}^{(4)}\rightarrow0^{+}$ as $j\rightarrow+\infty$.

Adding (\ref{G33}) and (\ref{G12}) we obtain
\[
\mathcal{F}_{\varepsilon_{j}}(v_{j},P_{\rho})\leq\frac{\mathcal{F}%
_{\varepsilon_{j}}(u_{j},P_{\rho})}{(1-\eta)^{2}}+
C_{P_{\rho}}(48\,C_{\theta
}^{2}M_{J}+M_{W})\delta+\sigma_{j}^{(5)}%
\]
where $\sigma_{j}^{(5)}\rightarrow0^{+}$ as $j\rightarrow+\infty$. This
implies that
\[
\limsup_{j\rightarrow+\infty}\mathcal{F}_{\varepsilon_{j}}(v_{j},P_{\rho}%
)\leq\frac{1}{(1-\eta)^{2}}\limsup_{j\rightarrow+\infty}\mathcal{F}%
_{\varepsilon_{j}}(u_{j},P_{\rho})+\kappa_{1}\delta~,
\]
where $\kappa_{1}$ is a constant independent of $j$, $\delta$,
and $P_{\rho}$. Passing to the limit as $\eta\rightarrow0^{+}$ we obtain (\ref{G7}).
\end{proof}

\section{Gamma Liminf Inequality}

\label{section gammaliminf}

In this section we prove the $\Gamma$-liminf inequality.

\begin{theorem}
[$\Gamma$-Liminf]\label{theorem liminf}Let $\varepsilon_{j}\rightarrow0^{+}$
and let $\{u_{j}\}$ be a sequence in $W_{\operatorname*{loc}}^{1,2}%
(\Omega)\cap L^{2}(\Omega)$ such that $u_{j}\rightarrow u$ in $L^{2}(\Omega)$
and
\begin{equation}
\liminf_{j\rightarrow+\infty}\mathcal{F}_{\varepsilon_{j}}(u_{j}%
,\Omega)<+\infty~. \label{DM2}
\end{equation}
Then $u\in BV(\Omega;\{-1,1\})$ and%
\begin{equation}
\liminf_{j\rightarrow+\infty}\mathcal{F}_{\varepsilon_{j}}(u_{j},\Omega
)\geq\int_{S_{u}}\psi(\nu_{u})~d\mathcal{H}^{n-1}~, \label{liminf inequality}%
\end{equation}
where $\psi$ is defined by (\ref{psi}).
\end{theorem}

Given $\nu\in\mathbb{S}^{n-1}$, let $\nu_{1}$, \ldots, $\nu_{n}$ be an
orthonormal basis in $\mathbb{R}^{n}$ with $\nu_{n}=\nu$, let
\begin{equation}
Q_{\rho}^{\nu}:=\{x\in\mathbb{R}^{n}:~|x\cdot\nu_{i}|<\rho/2~,~i=1,\ldots
,n\}~,\quad\hat{Q}_{\rho}^{\nu}:=\mathbb{R}^{n}\setminus Q_{\rho}^{\nu},
\label{Q ni rho}%
\end{equation}
and let
\[
S_{\rho}^{\nu}:=\{x\in\mathbb{R}^{n}:~|x\cdot\nu|<\rho/2\}~,\quad\hat{S}%
_{\rho}^{\nu}:=\mathbb{R}^{n}\setminus S_{\rho}^{\nu}~.
\]
When $\nu_{1}$, \ldots, $\nu_{n}$ is the canonical basis $e_{1}$, \ldots,
$e_{n}$ in $\mathbb{R}^{n}$ we omit the superscript $\nu$ in the above notation.

We recall the definition of the sets $V^{\nu}$ and $X^{\nu}$ in (\ref{V ni})
and in (\ref{X ni}), respectively. We will use these sets in what follows.
Further, as in\ Section \ref{section modification}, $\theta_{\varepsilon}$ is
the standard mollifier (see (\ref{standard mollifier})), and we set
\begin{equation}
\tilde{u}_{\varepsilon}:=w^{\nu}\ast\theta_{\varepsilon}~, \label{u tilde}%
\end{equation}
where $w^{\nu}$ is the function defined in (\ref{u nu}), with $\nu
\in\mathbb{S}^{n-1}$.

\begin{lemma}
\label{lemma S S}Let $0<\varepsilon<\delta<1/3$, let $C_{\delta
}:=Q_{1+\delta}\setminus Q_{1-\delta}$, and let $\tilde{u}_{\varepsilon}$ be
the function in (\ref{u tilde}), with $\nu=e_{n}$. Then
\[
\mathcal{J}_{\varepsilon}(\tilde{u}_{\varepsilon},C_{\delta})\leq\kappa
_{2}\delta
\]
for some constant $\kappa_{2}>0$ independent of $\varepsilon$ and $\delta$.
\end{lemma}

\begin{proof}
For every $\sigma>0$ define $C_{\delta}^{\sigma}:=C_{\delta}\cap
\{|x_{n}|<\sigma\}$, $\hat{C}_{\delta}^{\sigma}:=C_{\delta}\cap\{|x_{n}%
|\geq\sigma\}$, and write%
\[
C_{\delta}\times C_{\delta}=(C_{\delta}^{2\varepsilon}\times C_{\delta
}^{2\varepsilon})\cup(C_{\delta}^{\varepsilon}\times\hat{C}_{\delta
}^{2\varepsilon})\cup(\hat{C}_{\delta}^{2\varepsilon}\times C_{\delta
}^{\varepsilon})\cup(\hat{C}_{\delta}^{\varepsilon}\times\hat{C}_{\delta
}^{\varepsilon})~.
\]
Since $J$ is even, we have
\begin{equation}
\mathcal{J}_{\varepsilon}(\tilde{u}_{\varepsilon},C_{\delta})\leq
\mathcal{J}_{\varepsilon}(\tilde{u}_{\varepsilon},C_{\delta}^{2\varepsilon
})+2\mathcal{J}_{\varepsilon}(\tilde{u}_{\varepsilon},C_{\delta}^{\varepsilon
},\hat{C}_{\delta}^{2\varepsilon})+\mathcal{J}_{\varepsilon}(\tilde
{u}_{\varepsilon},\hat{C}_{\delta}^{\varepsilon})~. \label{ls 1}%
\end{equation}
By (\ref{standard mollifier}) we have that $\nabla\tilde{u}_{\varepsilon}=0$
on $\hat{C}_{\delta}^{\varepsilon}$ and so%
\begin{equation}
\mathcal{J}_{\varepsilon}(\tilde{u}_{\varepsilon},\hat{C}_{\delta
}^{\varepsilon})=0~. \label{ls 2}%
\end{equation}
We now estimate the first term on the right-hand side of (\ref{ls 1}). Since
$\varepsilon\nabla\tilde{u}_{\varepsilon}$ and $\varepsilon^{2}\nabla
^{2}\tilde{u}_{\varepsilon}$ are bounded in $L^{\infty}$ uniformly with
respect to $\varepsilon$, there exists a constant $c>0$ such that%
\[
|\nabla\tilde{u}_{\varepsilon}(x)-\nabla\tilde{u}_{\varepsilon}(y)|^{2}%
\leq\frac{c}{\varepsilon^{2}}\Bigl(\Bigl\vert\frac{x-y}{\varepsilon
}\Bigr\vert\wedge\Bigl\vert\frac{x-y}{\varepsilon}\Bigr\vert^{2}\Bigr)
\]
for every $x$, $y\in\mathbb{R}^{n}$. Therefore, by the change of variables
$z=(x-y)/\varepsilon$ and (\ref{J1}) we get%
\begin{align}
\mathcal{J}_{\varepsilon}(\tilde{u}_{\varepsilon},C_{\delta}^{2\varepsilon})
&  \leq\frac{c}{\varepsilon}\int_{C_{\delta}^{2\varepsilon}}\int_{C_{\delta
}^{2\varepsilon}}J_{\varepsilon}(x-y)\Bigl(\Bigl\vert\frac{x-y}{\varepsilon
}\Bigr\vert\wedge\Bigl\vert\frac{x-y}{\varepsilon}\Bigr\vert^{2}%
\Bigr)~dxdy\label{ls 3}\\
&  \leq\frac{cM_{J}}{\varepsilon}\mathcal{L}^{n}(C_{\delta}^{2\varepsilon
})\leq2^{n+1}cM_{J}\delta~.\nonumber
\end{align}
Next we study the second term on the right-hand side of (\ref{ls 1}). Since
$\nabla\tilde{u}_{\varepsilon}=0$ on $\hat{C}_{\delta}^{2\varepsilon}$ and
$\varepsilon\nabla\tilde{u}_{\varepsilon}$ is bounded in $L^{\infty}$
uniformly with respect to $\varepsilon$, there exists a constant $c>0$ such
that
\begin{align}
\mathcal{J}_{\varepsilon}(\tilde{u}_{\varepsilon},C_{\delta}^{\varepsilon
},\hat{C}_{\delta}^{2\varepsilon})  &  =\varepsilon\int_{C_{\delta
}^{\varepsilon}}\Bigl(\int_{\hat{C}_{\delta}^{2\varepsilon}}J_{\varepsilon
}(x-y)dx\Bigr)|\nabla\tilde{u}_{\varepsilon}(y)|^{2}dy\label{ls 4}\\
&  \leq\frac{c}{\varepsilon}\mathcal{L}^{n}(C_{\delta}^{\varepsilon}%
)\int_{\mathbb{R}^{n}\setminus B_{1}(0)}J(z)~dz\leq2^{n}cM_{J}\delta
~,\nonumber
\end{align}
where we used again the change of variables $z=(x-y)/\varepsilon$ and
(\ref{J1}). The conclusion follows by combining (\ref{ls 1})--(\ref{ls 4}).
\end{proof}

The following result will be crucial in the proof of the $\Gamma$-liminf inequality.

\begin{lemma}
\label{lemma extra terms}Let $0<\varepsilon<\delta<1/3$, let $u\in
X^{\nu}$ be such $u=\tilde{u}_{\varepsilon}$ in $Q_{1}^{\nu}\setminus
Q_{1-\delta}^{\nu}$, where $\tilde{u}_{\varepsilon}$ is the function defined
in (\ref{u tilde}). Then there exist two constants $\kappa_{3}$ and
$\kappa_{4}$, depending only on the dimension $n$ of the space, such that%
\[
\mathcal{J}_{\varepsilon}(u,V^{\nu},\mathbb{R}^{n})-\mathcal{J}_{\varepsilon
}(u,Q_{1}^{\nu})\leq\kappa_{2}\delta+\Bigl(\kappa_{3}\omega_{1}\Bigl(\frac
{\varepsilon}{\delta}\Bigr)+\kappa_{4}\omega_{1}(\varepsilon)\Bigr)\varepsilon
\int_{Q_{1}^{\nu}}|\nabla u(x)|^{2}dx~,
\]
where $\kappa_{2}$ is the constant in Lemma \ref{lemma S S}, and $\omega_{1}$
is the function defined in (\ref{little omega}).
\end{lemma}

\begin{proof}
Without loss of generality, we may assume that $\nu=e_{n}$, the $n$-th vector
of the canonical basis. For simplicity we omit the superscript $\nu$ in the
notation for $Q_{\rho}^{\nu}$, $\hat{Q}_{\rho}^{\nu}$, $S_{\rho}^{\nu}$,
$\hat{S}_{\rho}^{\nu}$, $V^{\nu}$, $X^{\nu}$, $w^{\nu}$, and the subscript
$\rho$ when $\rho=1$. Write%
\begin{align}
V&\times\mathbb{R}^{n}=((V\setminus Q)\times Q)\cup((V\setminus Q)\times\hat
{Q})\cup(Q\times Q)\cup(Q\times\hat{Q})\label{li 100}\\
&\subset(\hat{S}{\times} Q)\cup((V{\setminus}Q){\times} S)
\cup(\hat{S}{\times}\hat{S})\cup(Q{\times} Q)\cup (Q{\times}(S{\setminus}Q))\cup (Q{\times}\hat{S})~.
\nonumber
\end{align}
Since $J$ is even we have
\begin{align}
\mathcal{J}_{\varepsilon}(u,V,\mathbb{R}^{n})-\mathcal{J}_{\varepsilon}(u,Q)
&  \leq 2\varepsilon\int_{\hat{S}}\Bigl(\int_{Q_{1-\delta}}J_{\varepsilon
}(x-y)|\nabla u(x)|^{2}dx\Bigr)dy\nonumber\\
&  \quad+\varepsilon\int_{V\setminus Q}\Bigl(\int_{S_{1-\delta}}%
J_{\varepsilon}(x-y)|\nabla u(x)|^{2}dx\Bigr)dy\label{li 101}\\
&  \quad+\varepsilon\int_{Q}\Bigl(\int_{S\setminus Q}J_{\varepsilon
}(x-y)|\nabla u(x)-\nabla u(y)|^{2}dx\Bigr)dy~,\nonumber
\end{align}
where we have used the equalities $u=\pm1$ and $\nabla u=0$ in
$\hat{S}_{1-\delta}$,  which follow from the facts that $u\in X$ and
$u=\tilde{u}_{\varepsilon}$ on $Q_{1}\setminus Q_{1-\delta}$ (see (\ref{admonbou}), (\ref{mollifier support}), and the inequalities $0<\varepsilon<\delta<1/3$).

We now estimate the first term on the right-hand side of (\ref{li 101}). By
Lemma \ref{lemma separated} and because $\nabla u=0$ in $\hat{S}$, we
have%
\begin{equation}
\varepsilon\int_{\hat{S}}\Bigl(\int_{Q_{1-\delta}}J_{\varepsilon
}(x-y)|\nabla u(x)|^{2}dx\Bigr)dy\leq\varepsilon\omega_{1}\Bigl(\frac
{\varepsilon}{\delta}\Bigr)\int_{Q_{1-\delta}}|\nabla u(x)|^{2}dx~.
\label{li 102}%
\end{equation}

To estimate the second term on the right-hand side of (\ref{li 101}), we
identify $\mathbb{Z}^{n}$ with $\mathbb{Z}^{n-1}\times\mathbb{Z}$ so that for
$\alpha=(\alpha_{1},\ldots,\alpha_{n-1})\in\mathbb{Z}^{n-1}$ and $\beta
\in\mathbb{Z}$ we have $(\alpha,\beta)=(\alpha_{1},\ldots,\alpha_{n-1}%
,\beta)\in\mathbb{Z}^{n}$. Write%
\[
S\setminus Q_{3}=\bigcup_{\alpha\in\mathbb{Z}^{n-1},~|\alpha|_{\infty}\geq
2}((\alpha,0)+Q)~,\quad V=\bigcup_{\beta\in\mathbb{Z}}((0,\beta)+Q)~,
\]
where $|\alpha|_{\infty}:=\max\{|\alpha_{1}|,\ldots,|\alpha_{n-1}|\}$. Then%
\begin{align}
\varepsilon\int_{V\setminus Q}\Bigl(  &  \int_{S_{1-\delta}}J_{\varepsilon
}(x-y)|\nabla u(x)|^{2}dx\Bigr)dy\nonumber\\
&  \leq\varepsilon\int_{V\setminus Q}\Bigl(\int_{S_{1-\delta}\cap Q_{3}%
}J_{\varepsilon}(x-y)|\nabla u(x)|^{2}dx\Bigr)dy\label{li 103}\\
&  \quad+\sum_{\alpha\in\mathbb{Z}^{n-1},~|\alpha|_{\infty}\geq2}\sum
_{\beta\in\mathbb{Z}}\varepsilon\int_{(0,\beta)+Q}\Bigl(\int_{(\alpha
,0)+Q}J_{\varepsilon}(x-y)|\nabla u(x)|^{2}dx\Bigr)dy~.\nonumber
\end{align}
By Lemma \ref{lemma separated} and because $\nabla u=0$ in $V\setminus Q$, we
have%
\[
\varepsilon\int_{V\setminus Q}\Bigl(\int_{S_{1-\delta}\cap Q_{3}%
}J_{\varepsilon}(x-y)|\nabla u(x)|^{2}dx\Bigr)dy\leq\varepsilon\omega
_{1}\Bigl(\frac{\varepsilon}{\delta}\Bigr)\int_{S_{1-\delta}\cap Q_{3}}|\nabla
u(x)|^{2}dx~.\]
To estimate the second term on the right-hand side of (\ref{li 103}), we use
the change of variables $\zeta=x-y$ and observe that for $x\in(\alpha,0)+Q$
and $y\in(0,\beta)+Q$ we have $\zeta\in(\alpha,-\beta)+Q_{2}$. Therefore, we
obtain%
\begin{align*}
\int_{(0,\beta)+Q}\Bigl(  &  \int_{(\alpha,0)+Q}J_{\varepsilon}(x-y)|\nabla
u(x)|^{2}dx\Bigr)dy\\
&  =\int_{(\alpha,0)+Q}|\nabla u(x)|^{2}\Bigl(\int_{(0,\beta)+Q}%
J_{\varepsilon}(x-y)~dy\Bigr)dx\\
&  \leq\int_{(\alpha,0)+Q}|\nabla u(x)|^{2}dx\int_{(\alpha,-\beta)+Q_{2}%
}J_{\varepsilon}(\zeta)~d\zeta\\
&  =\int_{Q}|\nabla u(x)|^{2}dx\int_{(\alpha,-\beta)+Q_{2}}J_{\varepsilon
}(\zeta)~d\zeta~,
\end{align*}
where in the last equality we used the periodicity of $u\in X$. Hence%
\begin{align*}
\sum_{\alpha\in\mathbb{Z}^{n-1},~|\alpha|_{\infty}\geq2}\sum_{\beta
\in\mathbb{Z}}  &  \varepsilon\int_{(0,\beta)+Q}\Bigl(\int_{(\alpha
,0)+Q}J_{\varepsilon}(x-y)|\nabla u(x)|^{2}dx\Bigr)dy\\
&  \leq\varepsilon\int_{Q}|\nabla u(x)|^{2}dx\sum_{\alpha\in\mathbb{Z}%
^{n-1},~|\alpha|_{\infty}\geq2}\sum_{\beta\in\mathbb{Z}}\int_{(\alpha
,-\beta)+Q_{2}}J_{\varepsilon}(\zeta)~d\zeta\\
&  \leq2^{n}\varepsilon\int_{Q}|\nabla u(x)|^{2}dx
\int_{\hat{Q}_{2}}J_{\varepsilon}(\zeta)~d\zeta~.
\end{align*}
In the last inequality we used the fact that each point of $\hat{Q}_{2}$
belongs to at most $2^{n}$ cubes of the form
$(\alpha,-\beta)+Q_{2}$ for $\alpha\in\mathbb{Z}^{n-1}$, with $|\alpha
|_{\infty}\geq2$, and $\beta\in\mathbb{Z}$. After the change of variables
$z=\zeta/\varepsilon$ we obtain (see (\ref{little omega}))
\[
\int_{\hat{Q}_{2}}J_{\varepsilon}(\zeta)~d\zeta\leq
\int_{\mathbb{R}^{n}\setminus B_{1/\varepsilon}(0)}J(z)~dz\leq
\omega_{1}(\varepsilon)~.
\]
Combining the last five inequalities and using the periodicity of $u$, 
 from (\ref{li 103}) we
obtain%
\begin{align}
\varepsilon\int_{V\setminus Q}\Bigl(\int_{S_{1-\delta}}  &  J_{\varepsilon
}(x-y)|\nabla u(x)|^{2}dx\Bigr)dy\label{li 10}\\
&  \leq\Bigl(\omega_{1}\Bigl(\frac{\varepsilon}{\delta}\Bigr)+2^{n}\omega
_{1}(\varepsilon)\Bigr)\varepsilon\int_{S\cap Q_{3}}|\nabla u(x)|^{2}%
dx\nonumber\\
&  =3^{n-1}\Bigl(\omega_{1}\Bigl(\frac{\varepsilon}{\delta}\Bigr)+2^{n}%
\omega_{1}(\varepsilon)\Bigr)\varepsilon\int_{Q}|\nabla u(x)|^{2}dx~.\nonumber
\end{align}

Finally, to estimate the last term on the right-hand side of (\ref{li 101}),
we use the inclusion
\begin{align*}
Q&\times(S\setminus Q)   
\subset\big(Q\times(S\setminus Q_{3})\big)\cup\big(Q_{1-\delta
}\times(S\cap(Q_{3}\setminus Q_{1})\big)\\
&  \cup\big((Q_{1}\setminus Q_{1-\delta})\times 
(Q_{1+\delta}\setminus Q_{1})\big)\cup\big((Q_{1}\setminus Q_{1-\delta})\times
(S\cap(Q_{3}\setminus Q_{1+\delta}))\big)
\end{align*}
and we write
\begin{align}
\varepsilon\int_{Q}\Bigl(\int_{S\setminus Q}  &  J_{\varepsilon}(x-y)|\nabla
u(x)-\nabla u(y)|^{2}dx\Bigr)dy\nonumber\\
&  \leq\varepsilon\int_{Q}\Bigl(\int_{S\setminus Q_{3}}J_{\varepsilon
}(x-y)|\nabla u(x)-\nabla u(y)|^{2}dx\Bigr)dy\nonumber\\
&  \quad+\varepsilon\int_{Q_{1-\delta}}\Bigl(\int_{S\cap(Q_{3}\setminus
Q_{1})}J_{\varepsilon}(x-y)|\nabla u(x)-\nabla u(y)|^{2}%
dx\Bigr)dy\label{li 104}\\
&  \quad+\varepsilon\int_{Q_{1}\setminus Q_{1-\delta}
}\Bigl(\int_{Q_{1+\delta}\setminus Q_{1}
}J_{\varepsilon}(x-y)|\nabla u(x)-\nabla u(y)|^{2}dx\Bigr)dy\nonumber\\
&  \quad+\varepsilon\int_{Q_{1}\setminus Q_{1-\delta}
}\Bigl(\int_{S\cap(Q_{3}\setminus
Q_{1+\delta})}J_{\varepsilon}(x-y)|\nabla u(x)-\nabla u(y)|^{2}%
dx\Bigr)dy~.\nonumber
\end{align}
By Lemma \ref{lemma separated},
\begin{align}
\varepsilon\int_{Q_{1-\delta}}\Bigl(  &  \int_{S\cap(Q_{3}\setminus Q_{1}%
)}J_{\varepsilon}(x-y)|\nabla u(x)-\nabla u(y)|^{2}dx\Bigr)dy\nonumber\\
&  +\varepsilon\int_{Q_{1}\setminus Q_{1-\delta}
}\Bigl(\int_{S\cap(Q_{3}\setminus
Q_{1+\delta})}J_{\varepsilon}(x-y)|\nabla u(x)-\nabla u(y)|^{2}%
dx\Bigr)dy\label{li 105}\\
&  \leq2\varepsilon\omega_{1}\Bigl(\frac{\varepsilon}{\delta}\Bigr)\int_{S\cap
Q_{3}}|\nabla u(x)|^{2}dx=2{\cdot}3^{n-1}\varepsilon\omega_{1}\Bigl(\frac
{\varepsilon}{\delta}\Bigr)\int_{Q}|\nabla u(x)|^{2}dx~,\nonumber
\end{align}
where in the last equality we used the periodicity of $u$. On the other hand,
by Lemma \ref{lemma S S}%
\begin{equation}
\varepsilon\int_{Q_{1}\setminus Q_{1-\delta}
}\Bigl(\int_{Q_{1+\delta}\setminus Q_{1})
}J_{\varepsilon
}(x-y)|\nabla u(x)-\nabla u(y)|^{2}dx\Bigr)dy\leq\kappa_{2}\delta~.
\label{li 106}%
\end{equation}
It remains to study the first term on the right-hand side of (\ref{li 104}).
We have
\begin{align}
\varepsilon\int_{Q}\Bigl(  &  \int_{S\setminus Q_{3}}J_{\varepsilon
}(x-y)|\nabla u(x)-\nabla u(y)|^{2}dx\Bigr)dy\nonumber\\
&  \leq2\varepsilon\int_{Q}\Bigl(\int_{S\setminus Q_{3}}J_{\varepsilon
}(x-y)|\nabla u(x)|^{2}dx\Bigr)dy\label{li 107}\\
&  \quad+2\varepsilon\int_{Q}\Bigl(\int_{S\setminus Q_{3}}J_{\varepsilon
}(x-y)~dx\Bigr)|\nabla u(y)|^{2}dy~.\nonumber
\end{align}
To estimate the first term on the right-hand side of (\ref{li 107}) we write%
\begin{align*}
2\varepsilon\int_{Q}\Bigl(\int_{S\setminus Q_{3}}  &  J_{\varepsilon
}(x-y)|\nabla u(x)|^{2}dx\Bigr)dy\\
&  =2\varepsilon\sum_{\alpha\in\mathbb{Z}^{n}\cap(S\setminus Q_{3})}\int%
_{Q}\Bigl(\int_{\alpha+Q}J_{\varepsilon}(x-y)|\nabla u(x)|^{2}dx\Bigr)dy~.
\end{align*}
By Fubini's theorem and the change of variables $\zeta=x-y$, we get%
\begin{align*}
\int_{Q}&\Bigl(\int_{\alpha+Q}J_{\varepsilon}(x-y)|\nabla u(x)|^{2}dx\Bigr)dy
  =\int_{\alpha+Q}\Bigl(\int_{Q}J_{\varepsilon}(x-y)~dy\Bigr)|\nabla
u(x)|^{2}dx\\
&
  \leq\int_{\alpha+Q}\Bigl(\int_{x-Q}J_{\varepsilon}(\zeta)~d\zeta
\Bigr)|\nabla u(x)|^{2}dx
  \leq\int_{Q}|\nabla u(x)|^{2}dx\int_{\alpha-Q_{2}}J_{\varepsilon}%
(\zeta)~d\zeta~,
\end{align*}
where in the last inequality we have used the periodicity of $u$ and the
inclusion $x-Q\subset\alpha-Q_{2}$ for $x\in\alpha+Q$. Hence,%
\begin{align*}
2\varepsilon\sum_{\alpha\in\mathbb{Z}^{n}\cap(S\setminus Q_{3})}  &  \int%
_{Q}\Bigl(\int_{\alpha+Q}J_{\varepsilon}(x-y)|\nabla u(x)|^{2}dx\Bigr)dy\\
&  \leq2\varepsilon\int_{Q}|\nabla u(x)|^{2}dx\sum_{\alpha\in\mathbb{Z}%
^{n}\cap(S\setminus Q_{3})}\int_{\alpha-Q_{2}}J_{\varepsilon}(\zeta)~d\zeta\\
&  \leq2^{n}\varepsilon\int_{Q}|\nabla u(x)|^{2}dx\int_{\hat{Q}_{2}%
}J_{\varepsilon}(\zeta)~d\zeta~,
\end{align*}
where in the last inequality we used the fact that each point of $\hat{Q}_{2}$
belongs to at most $2^{n-1}$ cubes of the form $\alpha-Q_{2}$ for $\alpha
\in\mathbb{Z}^{n}\cap(S\setminus Q_{3})$. After the change of variables
$z=\zeta/\varepsilon$, we obtain%
\begin{equation}
2\varepsilon\int_{Q}\Bigl(\int_{S\setminus Q_{3}}\!\!\!\!\!\!J_{\varepsilon
}(x-y)|\nabla u(x)|^{2}dx\Bigr)dy\leq2^{n}\varepsilon\int_{Q}|\nabla
u(x)|^{2}dx\int_{\mathbb{R}^{n}\setminus B_{1/\varepsilon}%
(0)}\!\!\!\!\!\!\!\!\!\!\!\!\!\!\!\!\!\!\!\!J(z)|z|~dz~. \label{li 2}%
\end{equation}

We now estimate the second term on the right-hand side of (\ref{li 107}). With
the change of variables $z=(x-y)/\varepsilon$, we have%
\begin{equation}
2\varepsilon\int_{Q}\Bigl(\int_{S\setminus Q_{3}}\!\!\!\!\!\!\!J_{\varepsilon}%
(x-y)~dx\Bigr)|\nabla u(y)|^{2}dy\leq2\varepsilon\int_{\mathbb{R}^{n}\setminus
B_{1/\varepsilon}(0)}\!\!\!\!\!\!\!\!\!\!\!\!\!\!\!\!\!\!\!\!
J(z)|z|~dz\int_{Q}|\nabla u(y)|^{2}dy~.
\label{li 3}%
\end{equation}
Combining the inequalities (\ref{li 107})--(\ref{li 3}), we obtain
\begin{equation}
2\varepsilon\int_{Q}\Bigl(\int_{S\setminus Q_{3}}J_{\varepsilon}(x-y)|\nabla
u(x)|^{2}dx\Bigr)dy\leq2^{n}\varepsilon\omega_{1}(\varepsilon)\int%
_{Q}|\nabla u(x)|^{2}dx~. \label{li 108}%
\end{equation}
The conclusion follows from (\ref{li 102}), (\ref{li 10}), (\ref{li 104}),
(\ref{li 105}), (\ref{li 106}), and (\ref{li 108}).
\end{proof}

\bigskip\bigskip

\begin{proof}
[Proof of Theorem \ref{theorem liminf}]By Theorem
\ref{theorem compactness n>1} we deduce that $u\in BV(\Omega;\{-1,1\})$. 
Let $\mu_{j}$ be the nonnegative Radon measure on $\Omega$ defined by
\begin{equation}
\mu_{j}(B):=\frac{1}{\varepsilon}\int_{B}W(u_{j}(x))\,dx
+\varepsilon\int_{B}\int_{\Omega}J_{\varepsilon}(x-y)|\nabla
u_{j}(x)-\nabla u_{j}(y)|^{2}dxdy\label{DM1}
\end{equation}
for every Borel set $B\subset\Omega$.
Since $\mu_{j}(\Omega)=\mathcal{F}_{\varepsilon_{j}%
}(u_{j},\Omega)$, by (\ref{DM2}) $\mu_{j}(\Omega)$ is bounded uniformly with respect to~$j$.
Extracting a subsequence (not relabeled), we may assume that the liminf in
(\ref{liminf inequality}) is a limit and that
$\mu_{j}\overset{\ast}{\rightharpoonup}\mu$ weakly$^*$ in the space
$\mathcal{M}_{b}(\Omega)$ of bounded Radon measures on $\Omega$, considered, 
as usual, as the dual of the space $C_{0}(\Omega)$ of continuous functions on 
$\overline\Omega$ vanishing on $\partial\Omega$. 
Let $g$ be the density of the absolutely continuous
part of $\mu$ with respect to $\mathcal{H}^{n-1}$ restricted to $S_{u}$. Then the
inequality (\ref{liminf inequality}) will follow from%
\begin{equation}
g(x_{0})\geq\psi(\nu_{u}(x_{0}))\text{ for }\mathcal{H}^{n-1}\text{ a.e.
}x_{0}\in S_{u}~. \label{liminf1}%
\end{equation}
To prove this inequality, fix $x_{0}\in S_{u}$ such that, setting $\nu
:=\nu_{u}(x_{0})$, we have%
\begin{align}
&  \lim_{\rho\rightarrow0^{+}}\frac{1}{\rho^{n}}\int_{Q_{\rho}^{\nu}%
}|u(x+x_{0})-w^{\nu}(x+x_{0})|~dx=0~,\label{liminf2}\\
&  g(x_{0})=\lim_{\rho\rightarrow0^{+}}\frac{\mu(x_{0}+\overline{Q_{\rho}%
^{\nu}})}{\rho^{n-1}}<+\infty~. \label{liminf3}%
\end{align}
It is well-known (see \cite[Theorem 3 in Section 5.9]{evans-gariepy}) that
(\ref{liminf2}) and (\ref{liminf3}) hold for $\mathcal{H}^{n-1}$ a.e.
$x_{0}\in S_{u}$. Since $\mu_{j}\overset{\ast}{\rightharpoonup}\mu$ weakly$^*$ in $\mathcal{M}_{b}(\Omega)$,
by (\ref{214bis}) and (\ref{DM1}), using a change of variables, we get%
\begin{align*}
g(x_{0})  &  =\lim_{\rho\rightarrow0^{+}}\frac{\mu(x_{0}+\overline{Q_{\rho
}^{\nu}})}{\rho^{n-1}}\geq\limsup_{\rho\rightarrow0^{+}}\,\limsup_{j\rightarrow
+\infty}\frac{\mu_{j} (x_{0}+Q_{\rho}^{\nu})}{\rho^{n-1}}\\
&\geq  \limsup_{\rho\rightarrow0^{+}}\,\limsup_{j\rightarrow
+\infty}\frac{\mathcal{F}_{\varepsilon_{j}}(u_{j},x_{0}+Q_{\rho}^{\nu})}%
{\rho^{n-1}}=\limsup_{\rho\rightarrow0^{+}}\,\limsup_{j\rightarrow+\infty}\mathcal{F}%
_{\eta_{j,\rho}}(v_{j,\rho},Q_{1}^{\nu})~,
\end{align*}
where $\eta_{j,\rho}:=\varepsilon_{j}/\rho$ and $v_{j,\rho}(y):=u_{j}%
(x_{0}+\rho y)$. On the other hand, since $u_{j}\rightarrow u$ in
$L^{2}(\Omega)$, by (\ref{liminf2}) we obtain%
\begin{align*}
0  &  =\lim_{\rho\rightarrow0^{+}}\lim_{j\rightarrow+\infty}\frac{1}{\rho^{n}%
}\int_{Q_{\rho}^{\nu}}|u_{j}(x+x_{0})-w^{\nu}(x+x_{0})|~dx\\
&  =\lim_{\rho\rightarrow0^{+}}\lim_{j\rightarrow+\infty}\int_{Q_{1}^{\nu}%
}|v_{j,\rho}(x)-w^{\nu}(x)|~dx~.
\end{align*}
Since for every $\rho>0$
\[
\lim_{j\rightarrow+\infty}\eta_{j,\rho}=0~,
\]
by a diagonal argument we can choose $\rho_{j}\rightarrow0^{+}$ such that,
setting $\eta_{j}:=\eta_{j,\rho_{j}}$ and $v_{j}:=v_{j,\rho_{j}}$, we have
$\eta_{j}\rightarrow0^{+}$, $v_{j}\rightarrow w^{\nu}$ in $L^{1}(Q_{1}^{\nu}%
)$, and
\begin{equation}
g(x_{0})\geq\limsup_{j\rightarrow+\infty}\mathcal{F}_{\eta_{j}}(v_{j}%
,Q_{1}^{\nu})~. \label{liminf4}%
\end{equation}
The finiteness of $g(x_{0})$ and Theorem \ref{theorem compactness n>1} yield
that $v_{j}\rightarrow w^{\nu}$ in $L^{2}(Q_{1}^{\nu})$. We can now apply the
modification Theorem \ref{modification}: there exists $\delta_{\nu}>0$ such that 
for every $0<\delta<\delta_{\nu}$ we obtain a sequence $\{w_{j}\}\subset W_{\operatorname*{loc}}%
^{1,2}(Q_{1}^{\nu})\cap L^{2}(Q_{1}^{\nu})$ with $w_{j}\rightarrow w^{\nu}$ in
$L^{2}(Q_{1}^{\nu})$, $w_{j}=w^{\nu}\!\ast\theta_{\varepsilon_{j}}$ in
$Q_{1}^{\nu}\setminus Q_{1-\delta}^{\nu}$, and%
\begin{equation}
\limsup_{j\rightarrow+\infty}\mathcal{F}_{\eta_{j}}(v_{j},Q_{1}^{\nu}%
)\geq\limsup_{j\rightarrow+\infty}\mathcal{F}_{\eta_{j}}(w_{j},Q_{1}^{\nu})
-\kappa_{1}\delta~, \label{liminf5}%
\end{equation}
where, we recall, the constant $\kappa_{1}$ is independent of $\delta$. Extend
$w_{j}$ to $\mathbb{R}^{n}$ in such a way that $w_{j}(x)=\pm1$ for $\pm
x\cdot\nu\geq\frac{1}{2}$ and $w(x+\nu_{i})=w(x)$ for all $x\in\mathbb{R}^{n}$
and for all $i=1$, \ldots, $n-1$, where $\nu_{i}$ are the vectors in
(\ref{Q ni}). Then $w_{j}\in X^{\nu}$ and so we can apply Lemma
\ref{lemma extra terms} to obtain
\begin{align}
\limsup_{j\rightarrow+\infty}\, &\mathcal{F}_{\eta_{j}}(w_{j},Q_{1}^{\nu}) 
\geq\limsup_{j\rightarrow+\infty}(\mathcal{W}_{\eta_{j}}(w_{j},Q_{1}^{\nu
})+\mathcal{J}_{\eta_{j}}(w_{j},V^{\nu},\mathbb{R}^{n}))+{}\label{liminf6}\\
&  -\kappa_{2}\delta-\limsup_{j\rightarrow+\infty}\Bigl(\kappa_{3}\omega
_{1}\Bigl(\frac{\eta_{j}}{\delta}\Bigr)+\kappa_{4}\omega_{1}(\eta
_{j})\Bigr)\eta_{j}\int_{Q_{1}^{\nu}}|\nabla w_{j}(x)|^{2}dx~,\nonumber
\end{align}
where we recall that $\mathcal{W}_{\eta_{j}}$ is defined in (\ref{W local}).
By (\ref{psi}),
\begin{equation}
\mathcal{W}_{\eta_{j}}(w_{j},Q_{1}^{\nu})+\mathcal{J}_{\eta_{j}}(w_{j},V^{\nu
},\mathbb{R}^{n})\geq\psi(\nu) \label{li 7}%
\end{equation}
for every $j$ with $\eta_{j}<1$. By (\ref{liminf4}) and (\ref{liminf4})
the finiteness of $g(x_{0})$ implies that 
$\mathcal{F}_{\eta_{j}}(w_{j},Q_{1}^{\nu})$ is bounded uniformly with respect to $j$.
Therefore Lemma \ref{F53}, together with
the periodicity of $w_{j}$, proves that the same property holds for $\eta_{j}\int_{Q_{1}^{\nu}}|\nabla
w_{j}(x)|^{2}dx$. Together with
(\ref{little omega}), (\ref{liminf4}), (\ref{liminf5}), (\ref{liminf6}), and
(\ref{li 7}), this shows that $g(x_{0})\geq\psi(\nu)-\kappa_{1}\delta
-\kappa_{2}\delta$ for every $0<\delta<\delta_{\nu}$. Taking the limit as $\delta
\rightarrow0^{+}$ we obtain (\ref{liminf1}). This concludes the proof of the theorem.
\end{proof}

\bigskip

\section{Gamma Limsup Inequality}

\label{section gamma limsup}

In this section we prove the $\Gamma$-limsup inequality. Fix $\varepsilon
_{j}\rightarrow0^{+}$. For every $u\in BV(\Omega;\{-1,1\})$ we define%
\begin{equation}
\mathcal{F}^{^{\prime\prime}}(u,\Omega):=\inf\bigl\{\limsup_{j\rightarrow
+\infty}\mathcal{F}_{\varepsilon_{j}}(u_{j},\Omega):~u_{j}\rightarrow u\text{
in }L^{2}(\Omega)\bigr\}~. \label{limsup}%
\end{equation}

\begin{theorem}
[$\Gamma$-Limsup]\label{theorem limsup}For every $u\in BV(\Omega;\{-1,1\})$ we
have%
\begin{equation}
\mathcal{F}^{^{\prime\prime}}(u,\Omega)\leq\int_{S_{u}}\psi(\nu_{u}%
)~d\mathcal{H}^{n-1}~. \label{limsup inequality}%
\end{equation}

\end{theorem}

\bigskip

To prove the $\Gamma$-limsup inequality we need the results proved in the following lemmas.

\begin{lemma}
\label{lemma n-2}Let $u\in BV_{\operatorname*{loc}}(\mathbb{R}^{n};\{-1,1\})$ and, for every $\varepsilon>0$, 
let $\tilde
{u}_{\varepsilon}$ be as in (\ref{u tilde}). Assume that there exists a bounded polyhedral set $\Sigma$ of
dimension $n-1$ such that $S_{u}=\Sigma$, 
let $\Sigma^{n-2}$ the union of all its
$n-2$ dimensional faces, 
and let $(\Sigma^{n-2})^{\delta
}$ be defined as in 
(\ref{F49}). Then  there exists $\delta_{\Sigma}>0$ such that for $0<\varepsilon<\delta<\delta_{\Sigma}$ we have 
\[
\mathcal{J}_{\varepsilon}(\tilde{u}_{\varepsilon},(\Sigma^{n-2})^{\delta
})
\leq c_{1}
\delta\mathcal{H}^{n-2}(\Sigma^{n-2})
\]
for some constant $c_{1}
>0$ independent of $\varepsilon$, $\delta$, and $\Sigma$.
\end{lemma}

\begin{proof}
It is enough to repeat the proof of Lemma \ref{lemma S S} with $C_{\delta
}^{\sigma}$ and $\hat{C}_{\delta}^{\sigma}$ replaced by $\{x\in(\Sigma
^{n-2})^{\delta
}:~\operatorname*{dist}(x,\Sigma)<\varepsilon\}$ and
$\{x\in(\Sigma^{n-2})^{\delta
}:~\operatorname*{dist}(x,\Sigma){} \geq {}
\varepsilon\}$.
\end{proof}

\begin{lemma}
\label{lemma one face}Let $P$ be a bounded polyhedron of dimension $n-1$
containing $0$ with normal $\nu$, let $\rho>0$, 
and let
$P_{\rho}$ be the $n$-dimensional prism defined in (\ref{sigma rho}). Then for
every $\eta>0$ there exists a sequence $\{u_{\varepsilon}\}\subset
W^{1,2}(P_{\rho})$ such that $u_{\varepsilon}\rightarrow w^{\nu}$ in
$L^{2}(P_{\rho})$ and%
\[
\limsup_{\varepsilon\rightarrow0^{+}}\big(
\mathcal{W}_{\varepsilon}(u_{\varepsilon
},P_{\rho})+\mathcal{J}_{\varepsilon}(u_{\varepsilon},P_{\rho},\mathbb{R}%
^{n})\big)
\leq(\psi(\nu)+\eta)\mathcal{H}^{n-1}(P)~.
\]
\end{lemma}

\begin{proof}
Without loss of generality, we assume that $\nu=e_{n}$. For simplicity, we
omit the superscript $\nu$ in the notation for $w^{\nu}$, $X^{\nu}$, $V^{\nu}%
$, $Q_{1}^{\nu}$, and the subscript $\rho$ when $\rho=1$. By the definition of
$\psi$ (see (\ref{psi})), given $\eta>0$ there exist $\varepsilon_{\ast}%
\in(0,1)$ and $u_{\ast}\in X$ such that%
\begin{equation}
\mathcal{W}_{\varepsilon_{\ast}}(u_{\ast},Q)+\mathcal{J}_{\varepsilon_{\ast}%
}(u_{\ast},V,\mathbb{R}^{n})\leq\psi(e_{n})+\eta~. \label{ls 100}%
\end{equation}
Define $u_{\varepsilon}(x):=u_{\ast}(\frac{\varepsilon_{\ast
}}{\varepsilon}x)$
for $x\in\mathbb{R}^{n}$. Since $u_{\ast}(x)=\pm1$ for $\pm x_{n}\geq1/2$, the
sequence $\{u_{\varepsilon}\}$ converges to $w$ 
in $L_{\operatorname*{loc}%
}^{2}(\mathbb{R}^{n})$.

To estimate $\mathcal{W}_{\varepsilon}(u_{\varepsilon},P_{\rho})$ and
$\mathcal{J}_{\varepsilon}(u_{\varepsilon},P_{\rho},\mathbb{R}^{n})$, we
consider the $(n-1)$-dimensional cube $Q^{(n-1)}:=Q\cap\{x_{n}=0\}$ and we set%
\[
Z_{\varepsilon}:=\Bigl\{\{\alpha\in\mathbb{Z}^{n}:~\alpha_{n}=0~,~(\alpha
+Q^{(n-1)})\cap\Bigl(\frac{\varepsilon_{\ast}}{\varepsilon}P\Bigr)\neq\emptyset
\Bigr\}~.
\]
Observe that
\begin{equation}
\Bigl(\frac{\varepsilon}{\varepsilon_{\ast}}\Bigr)^{n-1}\#
 Z_{\varepsilon
}\rightarrow\mathcal{H}^{n-1}(P)\quad\text{as }\varepsilon\rightarrow0^{+}~,
\label{ls 101}%
\end{equation}
where $\#
 Z_{\varepsilon}$ is the number of elements of $Z_{\varepsilon}$.

Let $S:=\{x\in\mathbb{R}^{n}:~|x_{n}|<1/2\}$. Since $u_{\ast}(x)=\pm1$ for
$\pm x_{n}\geq1/2$, by (\ref{W1}) we have $W(u_{\ast}(x))=0$ for $x\in\mathbb{R}^{n}\setminus S$.
Therefore 
a change of variables and the
periodicity of $u_{\ast}$ give  
\begin{align}
\mathcal{W}_{\varepsilon}(u_{\varepsilon}&,P_{\rho})  =\Bigl(\frac
{\varepsilon}{\varepsilon_{\ast}}\Bigr)^{n-1}\mathcal{W}_{\varepsilon_{\ast}%
}\Bigl(u_{\ast},\frac{\varepsilon_{\ast}}{\varepsilon}P_{\rho}%
\Bigr)=\Bigl(\frac{\varepsilon}{\varepsilon_{\ast}}\Bigr)^{n-1}\mathcal{W}%
_{\varepsilon_{\ast}}\Bigl(u_{\ast},\Bigl(\frac{\varepsilon_{\ast}%
}{\varepsilon}P_{\rho}\Bigr)\cap S\Bigr)\nonumber\\
&  \leq\Bigl(\frac{\varepsilon}{\varepsilon_{\ast}}\Bigr)^{n-1}\sum_{\alpha\in
Z_{\varepsilon}}\mathcal{W}_{\varepsilon_{\ast}}(u_{\ast},\alpha
+Q)=\Bigl(\frac{\varepsilon}{\varepsilon_{\ast}}\Bigr)^{n-1}\#
Z_{\varepsilon}\mathcal{W}_{\varepsilon_{\ast}}(u_{\ast},Q)~.\label{ls 101a}
\end{align}
Similarly,%
\begin{align}
\mathcal{J}_{\varepsilon}(u_{\varepsilon}&,P_{\rho},\mathbb{R}^{n})
=\Bigl(\frac{\varepsilon}{\varepsilon_{\ast}}\Bigr)^{n-1}\!\!\mathcal{J}%
_{\varepsilon_{\ast}}\Bigl(u_{\ast},\frac{\varepsilon_{\ast}}{\varepsilon
}P_{\rho},\mathbb{R}^{n}\Bigr)\nonumber\\
&\leq\Bigl(\frac{\varepsilon}{\varepsilon_{\ast}%
}\Bigr)^{n-1}\!\!\sum_{\alpha\in Z_{\varepsilon}}\mathcal{J}_{\varepsilon_{\ast}%
}(u_{\ast},\alpha+V,\mathbb{R}^{n})
=\Bigl(\frac{\varepsilon}{\varepsilon_{\ast}}\Bigr)^{n-1}\#
Z_{\varepsilon}\mathcal{J}_{\varepsilon_{\ast}}(u_{\ast},V,\mathbb{R}%
^{n})~.\label{ls 102}
\end{align}
The result now follows from (\ref{ls 100})--(\ref{ls 102}).
\end{proof}

\bigskip

\begin{lemma}
\label{lemma sigma rho}Let $u\in BV_{\operatorname*{loc}}(\mathbb{R}%
^{n};\{-1,1\})$. Assume that there exists a bounded polyhedral set $\Sigma$ of
dimension $n-1$ such that $S_{u}=\Sigma$. 
 For every $\rho>0$ let $\Sigma
_{\rho}:=\{x\in\mathbb{R}^{n}:~\operatorname*{dist}(x,\Sigma)<\rho/2\}$. 
Then for every $\sigma>0$ there exist
$\rho>0$ and 
$\delta\in(0,\rho)$ 
with the following property: for every
$\varepsilon_{j}\rightarrow0^{+}$ there exists $v_{j}\in W^{1,2}(\Sigma
_{\rho})$ such that $v_{j}=u$ on $\Sigma_{\rho}\setminus\Sigma_{\rho-\delta}$
and%
\[
\limsup_{j\rightarrow+\infty}\mathcal{F}_{\varepsilon_{j}}(v_{j},\Sigma_{\rho
})\leq\int_{\Sigma}\psi(\nu_{u})~d\mathcal{H}^{n-1}+\sigma
~.
\]

\end{lemma}

\begin{proof}
Let $\delta_{\Sigma}>0$ be as in Lemma \ref{lemma n-2}. 
Fix $\sigma$ and $\hat{\sigma}$ with
$\hat{\sigma}\in(0,\min\{\sigma,\delta_{\Sigma}\})$. 
There exist $\rho\in(0,\hat{\sigma})$ 
and a finite number of bounded
polyhedra $P^{1}$, \ldots, $P^{k}$ of dimension $n-1$ and contained in the
$n-1$ dimensional faces of $\Sigma$ such that $\overline{P}{}_{\rho}^{i}\cap
\overline{P}{}_{\rho}^{j}=\emptyset$ 
for $i\neq j$ and%
\begin{equation}
\Sigma_{\rho}\setminus\bigcup_{i=1}^{k}P_{\rho}^{i}\subset(\Sigma
^{n-2})^{\hat{\sigma}
}, \label{ls 103}%
\end{equation}
where $P_{\rho}^{i}$ and $(\Sigma^{n-2})^{\hat{\sigma}
}$ are defined as in 
(\ref{sigma rho}) and Lemma 
\ref{lemma n-2}, respectively. Find $R^{1}$, \ldots,
$R^{k}$, bounded polyhedra of dimension $n-1$ contained in the $n-1$
dimensional faces of $\Sigma$, such that $P^{i}\Subset R^{i}$ and
$\overline{R}{}_{\rho}^{i}\cap\overline{R}{}_{\rho}^{j}=\emptyset$ 
for $i\neq j$.

Fix $\eta>0$ 
such that $\eta\mathcal{H}^{n-1}(\Sigma)<\sigma/2$. 
By Lemma \ref{lemma one face} for every $i=1$, \ldots, $k$,
there exists a sequence $\{u_{j}^{i}\}\subset W^{1,2}(R_{\rho}^{i})$ such that
$u_{j}^{i}\rightarrow u$ in $L^{2}(R_{\rho}^{i})$, and%
\begin{equation}
\limsup_{j\rightarrow+\infty}\big(
\mathcal{W}_{\varepsilon_{j}}(u_{j}^{i},R_{\rho
}^{i})+\mathcal{J}_{\varepsilon_{j}}(u_{j}^{i},R_{\rho}^{i},\mathbb{R}%
^{n})\big)
\leq(\psi(\nu^{i})+\eta)\mathcal{H}^{n-1}(R^{i})~.\label{77bis}
\end{equation}
By Theorem \ref{modification} there exist $\delta\in(0,\min\{\hat{\sigma},\rho/2\})$ 
and $\{v_{j}%
^{i}\}\subset W^{1,2}(R_{\rho}^{i})$ such that $v_{j}^{i}\rightarrow u$ in
$L^{2}(R_{\rho}^{i})$ as $j\rightarrow+\infty$, 
$v_{j}^{i}=u\!\ast\theta_{\varepsilon_{j}}$ on
$R_{\rho}^{i}\setminus (R_{\rho}^{i})_{\delta}$, 
and%
\begin{align}
\limsup_{j\rightarrow+\infty}\mathcal{F}_{\varepsilon_{j}}(v_{j}^{i},R_{\rho
}^{i})  &  \leq\limsup_{j\rightarrow+\infty}\mathcal{F}_{\varepsilon_{j}%
}(u_{j}^{i},R_{\rho}^{i})+\kappa_{1}\delta
\label{ls 104b}\\
&  \leq(\psi(\nu^{i})+\eta)\mathcal{H}^{n-1}(R^{i})+\kappa_{1}\hat\sigma
~,\nonumber
\end{align}
where, we recall, the costant $\kappa_{1}>0$ is independent of $j$, $\hat\sigma$, 
and $R_{\rho}^{i}$. Define $v_{j}:=v_{j}^{i}$ on $R_{\rho}^{i}$ and
$v_{j}:=u\!\ast\theta_{\varepsilon_{j}}$ on $A_{\rho}:={}
\Sigma_{\rho}\setminus
\bigcup_{i=1}^{k}R_{\rho}^{i}$. Then $v_{j}\in W^{1,2}(\Sigma_{\rho})$ and
$v_{j}\rightarrow u$ in $L^{2}(\Sigma_{\rho})$. \ Moreover $v_{j}=u$ on
$\Sigma_{\rho}\setminus\Sigma_{\rho-\delta}$ for all $j$ sufficiently large.

By additivity we obtain 
\begin{equation}
\mathcal{W}_{\varepsilon_{j}}(v_{j},\Sigma_{\rho})\leq\sum_{i=1}%
^{k}\mathcal{W}_{\varepsilon_{j}}(v_{j},R_{\rho}^{i})+\mathcal{W}%
_{\varepsilon_{j}}(v_{j},A_{\rho}
)~. \label{ls 104c}%
\end{equation}
Since $(u\!\ast\theta_{\varepsilon_{j}})
(x)=\pm1$ for $x\notin\Sigma_{2\varepsilon_{j}}$ and $-1\le (u\!\ast\theta_{\varepsilon_{j}})(x)\leq 1$, 
by (\ref{W1}) and (\ref{ls 103}) 
we
have%
\begin{align*}
\mathcal{W}_{\varepsilon_{j}}(v_{j},A_{\rho}
)  &
 {} \leq {}
\mathcal{W}_{\varepsilon_{j}}( u\!\ast\theta_{\varepsilon_{j}}
,(\Sigma^{n-2})^{\hat{\sigma}
}\cap
\Sigma_{2\varepsilon_{j}})\\
&\leq\frac{1}{\varepsilon_{j}}M_{W}\mathcal{L}%
^{n}((\Sigma^{n-2})^{\hat{\sigma}
}\cap\Sigma_{2\varepsilon_{j}})  \leq M_{W}c_{\Sigma
}\hat{\sigma}
\mathcal{H}^{n-2}%
(\Sigma^{n-2})~,
\end{align*}
where $M_{W}$ is the constant in (\ref{G5}) and $c_{\Sigma
}>0$ is a constant
depending only on the geometry of $\Sigma$. The previous inequality together
with (\ref{ls 104c}) gives%
\begin{equation}
\mathcal{W}_{\varepsilon_{j}}(v_{j},\Sigma_{\rho})\leq\sum_{i=1}%
^{k}\mathcal{W}_{\varepsilon_{j}}(v_{j},R_{\rho}^{i})+M_{W}c_{\Sigma
}%
\hat{\sigma}
\mathcal{H}^{n-2}(\Sigma^{n-2})~. \label{ls 104d}%
\end{equation}

To estimate $\mathcal{J}_{\varepsilon_{j}}(v_{j},\Sigma_{\rho})$ we use the
inclusion%
\begin{align*}
\Sigma_{\rho}\times\Sigma_{\rho}  &  \subset\bigcup_{i=1}^{k}(R_{\rho}%
^{i}\times R_{\rho}^{i})\cup\bigcup_{i=1}^{k}(P_{\rho}^{i}\times(\Sigma_{\rho
}\setminus R_{\rho}^{i}))\cup\bigcup_{i=1}^{k}((\Sigma_{\rho}\setminus
R_{\rho}^{i})\times P_{\rho}^{i})\\
&  \cup\biggl(\Bigl(\Sigma_{\rho}\setminus\bigcup_{i=1}^{k}P_{\rho}%
^{i}\Bigr)\times\Bigl(\Sigma_{\rho}\setminus\bigcup_{i=1}^{k}P_{\rho}%
^{i}\Bigr)\biggr)\cup\bigcup_{i\neq j}(R_{\rho}^{i}\times R_{\rho}^{j})~,
\end{align*}
which, together with (\ref{ls 103}), gives%
\begin{align}
\mathcal{J}_{\varepsilon_{j}}  &  (v_{j},\Sigma_{\rho})\leq\sum_{i=1}%
^{k}\mathcal{J}_{\varepsilon_{j}}(v_{j},R_{\rho}^{i})+\sum_{i=1}%
^{k}\mathcal{J}_{\varepsilon_{j}}(v_{j},P_{\rho}^{i},\Sigma_{\rho}\setminus
R_{\rho}^{i})\label{ls 104}\\
&  +\sum_{i=1}^{k}\mathcal{J}_{\varepsilon_{j}}(v_{j},\Sigma_{\rho}\setminus
R_{\rho}^{i},P_{\rho}^{i})+\mathcal{J}_{\varepsilon_{j}}%
(v_{j},(\Sigma^{n-2})^{\hat{\sigma}
})+\sum_{i\neq j}\mathcal{J}_{\varepsilon_{j}%
}(v_{j},R_{\rho}^{i},R_{\rho}^{j})~.\nonumber
\end{align}
By Lemma \ref{F53} and (\ref{ls 104b}) the sequence $\{\varepsilon_{j}%
\int_{R_{\rho}^{i}
}|\nabla v_{j}^{i
}|^{2}dx\}$ is uniformly bounded with respect
to $j$. Taking into account (\ref{mollifier support}) 
and (\ref{derivatives convolution}) we see that the same property holds for 
$\{\varepsilon_{j}%
\int_{\Sigma_{\rho}}|\nabla v_{j}|^{2}dx\}$. Hence, by Lemma \ref{lemma separated}, the second, third, and fifth
terms on the right-hand side of (\ref{ls 104}) tend to zero as $j\rightarrow
+\infty$. By Lemma \ref{lemma n-2},
\begin{equation}
\mathcal{J}_{\varepsilon_{j}}(v_{j},(\Sigma^{n-2})^{\hat{\sigma}
})\leq
c_{1
}\hat{\sigma}
\mathcal{H}^{n-2}(\Sigma^{n-2})~.
 \label{ls 106}%
\end{equation}
Combining 
(\ref{ls 104b}), (\ref{ls 104d}), (\ref{ls 104}), and
(\ref{ls 106})  we get
\begin{align}
\limsup_{j\rightarrow+\infty}\mathcal{F}_{\varepsilon_{j}}(v_{j},\Sigma_{\rho
})&\leq\int_{\Sigma}\psi(\nu_{u})~d\mathcal{H}^{n-1}
+\eta\mathcal{H}^{n-1}(\Sigma)\nonumber
\\
&+\kappa_{1}\hat{\sigma}+M_{W}c_{\Sigma
}%
\hat{\sigma}\mathcal{H}^{n-2}(\Sigma^{n-2})
+c_{1}\hat{\sigma}
\mathcal{H}^{n-2}(\Sigma^{n-2})~.\nonumber%
\end{align}
Since $\eta\mathcal{H}^{n-1}(\Sigma)<\sigma/2$, the 
conclusion
 can be obtained by taking $\hat{\sigma}$ sufficiently small.
\end{proof}

\medskip
We are now ready to prove Theorem \ref{theorem limsup}.
\medskip

\begin{proof}
[Proof of Theorem \ref{theorem limsup}]By \cite[Lemma 3.1]{baldo1990} for
every $u\in BV(\Omega;\{-1,1\})$ there exists a sequence $\{z_{k}\}$ in
$BV(\Omega;\{-1,1\})$ converging to $u$ in $L^{2}(\Omega)$ such that
$S_{z_{k}}$ is given by the intersection with $\Omega$ with a bounded
polyhedral set $\Sigma_{k}$ of dimension $n-1$ and $\mathcal{H}^{n-1}%
(S_{z_{k}})\rightarrow\mathcal{H}^{n-1}(S_{u})$. By Reshetnyak's convergence
theorem (see, e.g., 
\cite{spector}) this implies that%
\[
\lim_{k\rightarrow+
\infty}\int_{S_{z_{k}}}\psi(\nu_{z_{k}})~d\mathcal{H}%
^{n-1}=\int_{S_{u}}\psi(\nu_{u})~d\mathcal{H}^{n-1}~.
\]
Hence, using the lower semicontinuity of $\mathcal{F}^{^{\prime\prime}}%
(\cdot,\Omega)$ with respect to convergence in $L^{2}(\Omega)$ it suffices to
prove (\ref{limsup inequality}) for $u\in BV(\Omega;\{-1,1\})$ such that
$S_{u}=\Omega\cap\Sigma$ with $\Sigma$ a bounded polyhedral set of dimension
$n-1$.

In this case, for every $\sigma>0$ let $0<\delta<\rho$ and
$v_{j}\in W^{1,2}(\Sigma_{\rho})$ be as in Lemma \ref{lemma sigma rho}. Define
$u_{j}:=v_{j}$ on $\Sigma_{\rho}$ and $u_{j}:=u$ on $\Omega\setminus
\Sigma_{\rho}$. The properties of $v_{j}$ imply that $u_{j}:=u$ on
$\Omega\setminus\Sigma_{\rho-\delta}$ for all $j$ sufficiently large. Hence,
by (\ref{W1}) we have%
\begin{equation}
\mathcal{W}_{\varepsilon_{j}}(u_{j},\Omega)\leq\mathcal{W}_{\varepsilon_{j}%
}(u_{j},\Sigma_{\rho})~. \label{ls 107}%
\end{equation}

To estimate $\mathcal{J}_{\varepsilon_{j}}(u_{j},\Omega)$ we consider the
inclusion%
\begin{align}
\Omega\times\Omega\subset &  (\Sigma_{\rho}\times\Sigma_{\rho})\cup
(\Sigma_{\rho-\delta}\times(\Omega\setminus\Sigma_{\rho}))\cup((\Omega
\setminus\Sigma_{\rho})\times\Sigma_{\rho-\delta})\label{sets omega}\\
&  \cup((\Omega\setminus\Sigma_{\rho-\delta})\times(\Omega\setminus
\Sigma_{\rho-\delta}))~.\nonumber
\end{align}
Since $\nabla u_{j}=\nabla u=0$ on $\Omega\setminus\Sigma_{\rho-\delta}$, in
view of (\ref{sets omega}) we obtain
\begin{equation}
\mathcal{J}_{\varepsilon_{j}}(u_{j},\Omega)\leq\mathcal{J}_{\varepsilon_{j}%
}(u_{j},\Sigma_{\rho})+\mathcal{J}_{\varepsilon_{j}}(u_{j},\Sigma_{\rho
-\delta},\Omega\setminus\Sigma_{\rho})+\mathcal{J}_{\varepsilon_{j}}%
(u_{j},\Omega\setminus\Sigma_{\rho},\Sigma_{\rho-\delta})~. \label{ls 108}%
\end{equation}
By Lemmas \ref{F53} 
and \ref{lemma separated} the last two terms tend to zero as
$j\rightarrow\infty$, and by Lemma \ref{lemma sigma rho} we deduce
\[
\limsup_{j\rightarrow+\infty}\mathcal{F}_{\varepsilon_{j}}(u_{j}
,\Sigma_{\rho
})\leq\int_{\Sigma}\psi(\nu_{u})~d\mathcal{H}^{n-1}+\sigma
~.
\]
Together with (\ref{ls 107}) and (\ref{ls 108}) this shows that
\[
\mathcal{F}^{^{\prime\prime}}(u,\Omega)\leq\limsup_{j\rightarrow+\infty
}\mathcal{F}_{\varepsilon_{j}}(u_{j},\Omega)\leq\int_{\Sigma}\psi(\nu
_{u})~d\mathcal{H}^{n-1}+\sigma
~.
\]
Letting $\sigma$ 
tend to $0$ we obtain (\ref{limsup inequality}).
\end{proof}

\section{Acknowledgements}

The authors wish to acknowledge the Center for Nonlinear Analysis (NSF PIRE
Grant No. OISE-0967140) where part of this work was carried out. The research of G. Dal Maso 
was partially funded by the European Research Council under Grant No. 290888
``Quasistatic and Dynamic Evolution Problems in Plasticity and Fracture'', the research 
of I. Fonseca and G. Leoni
was partially funded by the National Science Foundation under
Grants 
No. DMS-1411646 and No. DMS-1412095, respectively.


\begin{thebibliography}{99}                                                                                               %


\bibitem {alberti-bellettini-I-1998}G. Alberti and G. Bellettini, A nonlocal
anisotropic model for phase transitions. I. The optimal profile problem. Math.
Ann. 310 (1998) 527--560.

\bibitem {alberti-bellettini-II-1998}G. Alberti and G. Bellettini, A non-local
anisotropic model for phase transitions: asymptotic behaviour of rescaled
energies. European J. Appl. Math. 9 (1998) 261--284.

\bibitem {alberti-bellettini-cassandro-presutti1996}G. Alberti, G. Bellettini,
M. Cassandro, and E. Presutti, Surface tension in Ising systems with Kac
potentials, J. Stat. Phys. 82 (1996) 743--796.

\bibitem {alberti-bouchitte-seppecher1994}G. Alberti, G. Bouchitt\'{e}, and P.
Seppecher, Un r\'{e}sultat de perturbations singuli\`{e}res avec la norme
$H^{1/2}$. C. R. Acad. Sci. Paris S\'{e}r. I Math. 319 (1994) 333--338.

\bibitem {alberti-bouchitte-seppecher1998}G. Alberti, G. Bouchitt\'{e}, and P.
Seppecher, Phase transition with the line-tension effect. Arch. Rat. Mech.
Anal. 144 (1998) 1--46.

\bibitem {ambrosio-de-philippis}L. Ambrosio, G. De Philippis, and L.
Martinazzi, Gamma-convergence of nonlocal perimeter functionals. Manuscripta
Math. 134 (2011), no. 3-4, 377--403.

\bibitem {ambrosio-fusco-pallara2000}L. Ambrosio, N. Fusco, and D. Pallara,
Functions of bounded variation and free discontinuity problems, {Oxford
Mathematical Monographs. Oxford: Clarendon Press}, 2000.

\bibitem {baldo1990}S. Baldo, Minimal interface criterion for phase
transitions in mixtures of Cahn-Hilliard fluids. Ann. Inst. H. Poincar\'{e}
Anal. Non Lin\'{e}aire 7 (1990) 67--90.

\bibitem {barroso-fonseca1994}A.C. Barroso and I. Fonseca, Anisotropic
singular perturbations - the vectorial case. Proc. Roy. Soc. Edinburgh Sect. A
124 (1994) 527--571.

\bibitem {bouchitte1990}G. Bouchitt\'{e}, Singular perturbations of
variational problems arising from a two-phase transition model. Appl. Math.
Optim. 21 (1990) 289--314.

\bibitem {bourgain-brezis-mironescu}J. Bourgain, H. Brezis, and P. Mironescu,
Limiting embedding theorems for $W^{s,p}$ when $s\uparrow1$ and applications.
Dedicated to the memory of Thomas H. Wolff. J. Anal. Math. 87 (2002) 77--101.

\bibitem {brezis2015}H. Brezis, New approximations of the total variation and
filters in imaging. Atti Accad. Naz. Lincei Rend. Lincei Mat. Appl. 26 (2015),
no. 2, 223--240.

\bibitem {cacace-garroni}S. Cacace and A. Garroni, A multi-phase transition
model for the dislocations with interfacial microstructure. Interfaces Free
Bound. 11 (2009), no. 2, 291--316.

\bibitem {caffarelli-stinga}L.A. Caffarelli and P.R. Stinga, Pablo, Fractional
elliptic equations, Caccioppoli estimates and regularity. Ann. Inst. H.
Poincar\'{e} Anal. Non Lin\'{e}aire 33 (2016), no. 3, 767--807.

\bibitem {cahn-hilliard1958}J. W. Cahn and J.E. Hilliard, Free energy of a
nonuniform system. I. Interfacial free energy, J. Chem. Phys 28 (1958) 258--267.

\bibitem {chermisi-dalmaso-fonseca-leoni2011}M. Chermisi, G. Dal Maso, I.
Fonseca, and G. Leoni, Singular perturbation models in phase transitions for
second-order materials. Indiana Univ. Math. J. 60 (2011) 367--409.

\bibitem {cicalese-spadaro-zeppieri2011}M. Cicalese, E. Spadaro, C.I.
Zeppieri, Asymptotic analysis of a second-order singular perturbation model
for phase transitions. Calc. Var. Partial Differential Equations 41 (2011) 127--150.

\bibitem {conti-garroni-mueller}S. Conti, A. Garroni, and S. M\"{u}ller,
Singular kernels, multiscale decomposition of microstructure, and dislocation
models. Arch. Ration. Mech. Anal. 199 (2011), no. 3, 779--819.

\bibitem {dalmaso1993}G. Dal Maso, An introduction to $\Gamma$-convergence,
{Progress in Nonlinear Differential Equations and their Applications. 8.
Basel: Birkh\"{a}user}, 1993.

\bibitem {dinezza-palatucci-valdinoci2012}E. Di Nezza, G. Palatucci, and E.
Valdinoci, Hitchhiker's guide to the fractional Sobolev spaces. Bull. Sci.
Math. 136 (2012) 521--573.

\bibitem {evans-gariepy}L. Evans and R.F. Gariepy, Measure theory and fine
properties of functions. . Studies in Advanced Mathematics. CRC Press, Boca
Raton, FL, 1992.

\bibitem {fonseca-hayrapetyan-leoni-zwicknagl2016}I. Fonseca, G. Hayrapetyan,
G. Leoni and B. Zwicknagl, Domain formation in membranes near the onset of
instability. To appear in J. Nonlinear Sci.

\bibitem {fonseca-mantegazza2000}I. Fonseca and C. Mantegazza, Second order
singular perturbation models for phase transitions. SIAM J. Math. Anal. 31
(2000), no. 5, 1121--1143.

\bibitem {fonseca-tartar1989}I. Fonseca and L. Tartar, The gradient theory of
phase transitions for systems with two potential wells. Proc. Roy. Soc.
Edinburgh Sect. A 111 (1989) 89--102.

\bibitem {gagliardo1957}E. Gagliardo, Caratterizzazioni delle tracce sulla
frontiera relative ad alcune classi di funzioni in $n$ variabili, Rend. Sem.
Mat. Univ. Padova 27 (1957) 284--305.

\bibitem {garroni-mueller}A. Garroni and S. M\"{u}ller, A variational model
for dislocations in the line tension limit. Arch. Ration. Mech. Anal. 181
(2006), no. 3, 535--578.

\bibitem {garroni-palatucci2006}A. Garroni and G. Palatucci, A singular
perturbation result with a fractional norm. Variational problems in materials
science, 111--126, Progr. Nonlinear Differential Equations Appl., 68,
Birkh\"{a}user, Basel, 2006.

\bibitem {gurtin1985}M.E. Gurtin, Some results and conjectures in the gradient
theory of phase transitions. IMA, preprint 156 (1985).

\bibitem {hilhorst-peletier-schatzle2002}D. Hilhorst, L. A. Peletier, and R.
Sch\"{a}tzle, $\Gamma$-limit for the extended Fisher-Kolmogorov equation,
Proc. Roy. Soc. Edinburgh Sect. A 132 (2002) 141--162.

\bibitem {kawakatsu-andelman-kawasaki-taniguchi1993}T. Kawakatsu, D. Andelman,
K. Kawasaki, and T. Taniguchi, Phase-transitions and shapes of two-component
membranes and vesicles I: strong segregation limit, Journal de Physique II 3
(1993) 971--997.

\bibitem {leibler-andelman1987}S. Leibler and D. Andelman, Ordered and curved
meso-structures in membranes and amphiphilic films, J. Phys. (France) 48
(1987) 2013--2018.

\bibitem {leoni2009}G. Leoni, A first course in Sobolev spaces, {Graduate
Studies in Mathematics 105. Providence, RI: American Mathematical Society
(AMS)}, 2009.

\bibitem {leoni-spector}G. Leoni and D. Spector, Characterization of Sobolev
and BV spaces. J. Funct. Anal. 261 (2011), no. 10, 2926--2958. Corrigendum to
"Characterization of Sobolev and BV spaces\textquotedblright\ J. Funct. Anal.
266 (2014), no. 2, 1106--1114.

\bibitem {modica-mortola1977}L. Modica and S. Mortola, Un esempio di $\Gamma
$-convergenza. (Italian) Boll. Un. Mat. Ital. B (5) 14 (1977) 285--299.

\bibitem {modica1987}L. Modica, The gradient theory of phase transitions and
the minimal interface criterion. Arch. Rational Mech. Anal. 98 (1987) 123--142.

\bibitem {owen1988}N.C. Owen, Nonconvex variational problems with general
singular perturbations. Trans. Amer. Math. Soc. 310 (1988) 393--404.

\bibitem {owen-sternberg1991}N.C. Owen and P. Sternberg, Nonconvex variational
problems with anisotropic perturbations. Nonlinear Anal. 16 (1991) 705--719.

\bibitem {rowlinson1979}J.S. Rowlinson, Translation of J.D. Van der Waals: The
thermodynamic theory of capillarity under the hypothesis of a continuous
variation of density, J. Stat. Phys. 20 (1979) 200--244.

\bibitem {savin-valdinoci2012}O. Savin and E. Valdinoci, $\Gamma$-convergence
for nonlocal phase transitions. Ann. I. H. Poincar\'{e} 29 (2012) 479--500.

\bibitem {seul-andelman1995}M. Seul and D. Andelman, Domain shapes and
patterns - the phenomenology of modulated phases, Science 267 (1995) 476--483.

\bibitem {Sch}R. Schneider, W. Weil: Stochastic and integral geometry.
Probability and its Applications, Springer--Verlag, Berlin, 2008.

\bibitem {spector}D. Spector, Simple proofs of some results of Reshetnyak.
Proc. Amer. Math. Soc. 139 (2011), no. 5, 1681--1690.

\bibitem {swift-hohenberg1977}J. B. Swift and P. C. Hohenberg, Hydrodynamic
Fluctuations at the convective instability, Phys. Rev. A 15 (1977) 319--328.

\bibitem {sternberg1988}P. Sternberg, The effect of a singular perturbation on
nonconvex variational problems. Arch. Rational Mech. Anal. 101 (1988) 209--260.

\bibitem {sternberg1991}P. Sternberg, Vector-valued local minimizers of
nonconvex variational problems. Rocky Mountain J. Math. 21 (1991) 799--807.

\bibitem {taniguchi-kawasaki-andelman-kawakatsu1994}T. Taniguchi, K. Kawasaki,
D. Andelman, and T. Kawakatsu, Phase-transitions and shapes of two-component
membranes and vesicles II: weak segregation limit, Journal de Physique II 4
(1994) 1333--1362.

\bibitem {van-der-walls1893}J.D. van der Waals, The thermodynamic theory of
capillarity under the hypothesis of a continuous variation of density,
Zeitschrift f\"{u}r Physikalische Chemie 13 (1894) 657--725.
\end{thebibliography}
\end{document}